\documentclass[reqno,a4paper]{amsart}
\usepackage[utf8]{inputenc}
\usepackage{amsmath,amssymb}
\usepackage{mathrsfs}
\usepackage{yhmath}
\usepackage{booktabs,url}
\usepackage{xcolor}
\usepackage{enumerate}
\usepackage{cases}
\usepackage{hyperref}
\usepackage{siunitx}
\usepackage{comment}

\usepackage[left=2.5cm, right=2.5cm, bottom=3cm]{geometry}
\usepackage[font=footnotesize,labelfont=bf]{caption}

\newcommand{\de}[0]{\mathrel{\mathop:}=}
\newcommand{\ed}[0]{=\mathrel{\mathop:}}
\newcommand{\ie}[0]{\mathrm{i}}
\newcommand{\dif}[1]{\mathrm{d}#1}
\newcommand{\C}[0]{\mathbb{C}}
\newcommand{\R}[0]{\mathbb{R}}
\newcommand{\Z}[0]{\mathbb{Z}}
\newcommand{\N}[0]{\mathbb{N}}
\newcommand{\K}[0]{\mathbb{K}}
\newcommand{\Q}[0]{\mathbb{Q}}

\newcommand{\cL}[0]{\mathcal{L}}
\newcommand{\cS}[0]{\mathcal{S}}
\newcommand{\cSP}[0]{\mathcal{SP}}
\newcommand{\sdeg}[0]{\dif{}_{\mathcal{L}}}
\newcommand{\sq}[0]{\mathrm{q}_{\mathcal{L}}}
\newcommand{\dL}[1]{\mathrm{d}_{#1}}
\newcommand{\qL}[1]{\mathrm{q}_{#1}}

\usepackage{amsmath}

\newtheorem{theorem}{Theorem}
\newtheorem{lemma}{Lemma}[section]
\newtheorem{proposition}[lemma]{Proposition}
\newtheorem{corollary}[lemma]{Corollary}

\newtheorem{conjecture}{Conjecture}

\theoremstyle{definition}
\newtheorem{remark}[lemma]{Remark}

\numberwithin{equation}{section}

\begin{document}

\title[Conditional estimates for $L$-functions II]{Conditional estimates for $L$-functions in the Selberg class II}
\author[N.~Paloj\"{a}rvi and A.~Simoni\v{c}]{Neea Paloj\"{a}rvi and Aleksander Simoni\v{c}}
\address{School of Science, The University of New South Wales (Canberra), ACT, Australia}
\email{n.palojarvi@unsw.edu.au}
\address{School of Science, The University of New South Wales (Canberra), ACT, Australia}
\email{a.simonic@unsw.edu.au}
\date{\today}

\begin{abstract}
Assuming the Generalized Riemann Hypothesis, we provide uniform upper and lower bounds with explicit main terms for $\log{\left|\cL(s)\right|}$ for $\sigma \in (1/2,1)$ and for functions in the Selberg class. In particular, we focus on the region $0\leq\sigma-1/2\ll 1/\log{\log{\left(\sq|t|^{\sdeg}\right)}}$. We also provide estimates under additional assumptions on the distribution of Dirichlet coefficients of $\cL(s)$ on prime numbers. Moreover, by assuming a polynomial Euler product representation for $\cL(s)$, we establish both uniform bounds and completely explicit estimates by also assuming the strong $\lambda$-conjecture. In addition to providing estimates for a large set of functions, our results improve the best known estimates for specific functions in the Selberg class including the lower bounds for the Riemann zeta function close to the critical line.
\end{abstract}
	
\maketitle
\thispagestyle{empty}

\section{Introduction}

\subsection{Introduction to the Selberg class}
\label{subsec:Basics}
The Selberg class of functions $\cS$ was introduced in~\cite{SelbergOldAndNew} and consists of Dirichlet series
\begin{equation*}
\label{eq:DS}
\mathcal{L}(s) = \sum_{n=1}^{\infty} \frac{a(n)}{n^s},
\end{equation*}
where $s=\sigma+\ie t$ for real numbers $\sigma$ and $t$, satisfying the following axioms:
\begin{enumerate}[(i)]
  \item \emph{Ramanujan hypothesis}. We have $a(n)\ll_{\varepsilon} n^{\varepsilon}$ for any $\varepsilon>0$.
  \item \emph{Analytic continuation}. There exists $k\in\N_{0}$ such that $(s-1)^k\mathcal{L}(s)$ is an entire function of finite order.
  \item \emph{Functional equation}. The function $\mathcal{L}(s)$ satisfies $\mathfrak{L}(s)=\omega\overline{\mathfrak{L}(1-\bar{s})}$, where
        \[
        \mathfrak{L}(s)=\mathcal{L}(s)Q^s\prod_{j=1}^{f}\Gamma\left(\lambda_{j}s+\mu_j\right)
        \]
        with $\left(Q,\lambda_j\right)\in\R_{+}^2$, and $\left(\mu_j,\omega\right)\in\C^2$ with $\Re\{\mu_j\}\geq 0$ and $|\omega|=1$.
  \item \emph{Euler product}. The function $\cL(s)$ has a product representation
        \[
        \mathcal{L}(s) = \prod_{p}\exp{\left(\sum_{k=1}^{\infty}\frac{b\left(p^k\right)}{p^{ks}}\right)},
        \]
        where the coefficients $b\left(p^k\right)$ satisfy $b\left(p^k\right)\ll p^{k\theta}$ for some $0\leq\theta<1/2$.
\end{enumerate}
Examples of functions in $\mathcal{S}$ include the Riemann zeta-function $\zeta(s)$, Dirichlet $L$-functions $L(s,\chi)$ for a primitive characters $\chi$ modulo $q$ and Dedekind zeta-functions $\zeta_{\K}(s)$ for a number fields $\K$ with degree $n_{\K}=[\K:\Q]$ and discriminant $\Delta_{\K}$. 

It is well-known that the data from the functional equation is not uniquely defined by $\cL\in\cS$. However, one can show that the degree $\sdeg$, the conductor $\sq$ and $\xi$-invariant,
\[
\sdeg := 2\sum_{j=1}^{f}\lambda_j, \quad \sq:=(2\pi)^{\sdeg}Q^2\prod_{j=1}^{f}\lambda_{j}^{2\lambda_{j}}, \quad \xi_{\mathcal{L}}:=2\sum_{j=1}^f \left(\mu_j-\frac{1}{2}\right),
\]
are indeed invariants for the class $\cS$ (see, e.g.,~\cite{odzak2011} for survey). 
For example, for the above mentioned examples of functions we have:
\begin{enumerate}
\item The Riemann zeta-function $\zeta(s)$: $\dL{\zeta}=1$, $\qL{\zeta}=1$ and $\xi_\zeta=-1$.
\item Dirichlet $L$-functions $L(s,\chi)$: $\dL{L(s,\chi)}=1$, $\qL{L(s,\chi)}=q$ and $\xi_{L(s,\chi)}=-1$ if $\chi(-1)=1$ and $0$ otherwise.
\item Dedekind zeta-functions $\zeta_{\K}(s)$: $\dL{\zeta_{\K}}=n_{\K}$, $\qL{\zeta_{\K}}=\left|\Delta_{\K}\right|$ and $\xi_{\zeta_{\K}}=-(r+1)$.
\end{enumerate}

The logarithmic derivative of $\cL$ is
\[
\frac{\cL'}{\cL}(s) = -\sum_{n=1}^{\infty}\frac{\Lambda_{\cL}(n)}{n^s}
\]
for $\sigma>1$, where $\Lambda_{\cL}(n)$ is \textit{the generalized von Mangoldt function}. The Euler product representation implies $\Lambda_{\cL}(n)=b(n)\log{n}$ with $\Lambda_{\cL}(n)=0$ for $n$ not being a prime power. 
In analogy with $\log{\zeta(s)}$, we define
\[
\log{\cL(s)} = \sum_{n=2}^{\infty}\frac{\Lambda_{\cL}(n)}{n^{s}\log{n}}
\]
for $\sigma>1$, where the value of $\log{\cL(s)}$ is that which tends to $0$ as $\sigma\to\infty$ for any fixed $t$. For $\sigma<2$ and $t\neq 0$, we define $\log{\cL(s)}$ as the analytic continuation of the above function along the straight line from $2+\ie t$ to $\sigma+\ie t$, provided that $\cL(s)\neq 0$ on this segment of line. Note that GRH implies holomorphicity of $\log{\cL(z)}$ on the domain $\left\{z\in\C\colon \Re\{z\}>1/2\right\}\setminus(1/2,1]$.

In some of the results, we are able to derive sharper estimates by assuming that $\cL(s)$ has a polynomial Euler product representation. Indeed, then axiom~(iv) is replaced by a stronger hypothesis
\[
\mathcal{L}(s) = \prod_{p}\prod_{j=1}^{m}\left(1-\frac{\alpha_j(p)}{p^s}\right)^{-1},
\]
where $m\in\N$ is the \emph{order of a polynomial Euler product} and $\alpha_j(p)\in\C$, $1\leq j\leq m$, are some functions which are defined for every prime number $p$. Denote by $\cSP$ the class of such functions. If $\cL\in\cSP$, then
\[
b\left(p^k\right) = \frac{1}{k}\sum_{j=1}^{m}\left(\alpha_j(p)\right)^k,
\]
and because $\left|\alpha_j(p)\right|\leq 1$ for $1\leq j\leq m$, see~\cite[Lemma 2.2]{SteudingBook}, we have that
\begin{equation}
\label{eq:BoundOnLambda}
\left|\Lambda_{\cL}(n)\right|\leq m\Lambda(n).
\end{equation}

\subsection{On upper and lower bounds for $\cL\in\cS$}
Estimates for logarithms and logarithmic derivatives of $L$-functions have applications to many other interesting number theory questions. These include the number of zeros and the argument of the Selberg class functions~\cite{Carneiro, ChirreANote, Palojarvi,SimonicSonRH}, number of primes, primes in arithmetic progressions, Goldbach summatory functions and their generalizations~\cite{BEHP2024,Davenport,EHP,MontgomeryVaughan}, characterization of positive-definite integral quadratic forms that are co-prime-universal for $p$ \cite{BC2024}, Mertens' theorem~\cite{SimonicLfunctions}, moments of functions~\cite{BF2023,SoundMoments} and character sums~\cite{Cech}. Numerical estimates for the logarithms are helpful, for example, when verifying that the Generalized Riemann hypothesis holds up to some height (see e.g. \cite{Booker2006}).

In 1924--1925, Littlewood showed~\cite{LittlewoodOnTheZeros,LittlewoodOnTheRiemann} that under the Riemann Hypothesis (RH), we have
\begin{equation}
\label{eq:LittlewoodSpecial}
\zeta\left(\frac{1}{2}+\ie t\right) \leq \exp\left(\left(C+o(1)\right)\frac{\log t}{\log{\log{t}}}\right) \quad\text{and}\quad \log{\zeta(s)} \ll \frac{(\log t)^{2-2\sigma}}{\log{\log{t}}}
\end{equation}
for some $C>0$, $1/2+\delta\leq\sigma\leq1-\delta$, fixed $\delta\in(0,1/4)$ and $t$ large enough. It is possible to extend his techniques in~\cite{LittlewoodOnTheRiemann} (see also~\cite[Section 14.5]{Titchmarsh}) to the region right of the $1$-line, see~\cite[Section 14.33]{Titchmarsh}, and obtain
\begin{equation*}
\label{eq:Littlewood2}
    \log{\zeta(s)} \ll \frac{(\log t)^{2-2\sigma}}{\log{\log{t}}}\min\left\{\dfrac{1}{|\sigma-1|},\log\log t\right\} + \log{\log{\log{t}}}
\end{equation*}
for $1/2+1/\log{\log{t}}\leq\sigma\leq3/2$ and $t$ large enough, see~\cite[Corollaries 13.14 and 13.16]{MontgomeryVaughan}. A few years later, Titchmarsh~\cite{TitchRiemann} proved that unconditionally
\begin{equation*}
    \log{\left|\zeta(\sigma+\ie t)\right|}\gg \left(\log{t}\right)^{1-\sigma-\varepsilon},
\end{equation*}
where $\sigma\in[1/2,1)$ is fixed and $\varepsilon \in (0,1-\sigma)$. Montgomery~\cite{MontgomeryRiemann} improved this result by considering the sequence $t_n \to \infty$, and proved that his result holds with a slightly better region for each $t_n$ if RH is assumed. 

These results, under GRH and with $\sigma \in [1/2,1)$, have been improved several times, extended to other functions, and explicit results have been proved in addition to asymptotic ones (see, e.g.,~\cite{AistleitnerPankowski2017,BF2023, Carneiro, CarneiroChandee, CM2024, ChandeeExplBounds, ChandeeSound, ChirreANote, PalojarviSimonic, SimonicCS, SimonicSonRH, XiaoYang2022}). Here, \textit{explicit estimate} means that all ``hidden'' constants in a corresponding non-explicit result are computed or expressed with known parameters via some algebraic formulae. For example, the results in~\eqref{eq:LittlewoodSpecial} are not explicit since we do not know the constants in front of the estimates, but the estimate in~\cite[Corollary 1.1]{ChandeeExplBounds} is explicit. It should be noted that there is variation how the term \emph{explicit} is used in the literature. Sometimes, it refers to results where the main term in the estimate is explicitly given and the rest of the terms are asymptotically estimated (see, e.g.,~\cite{chirreGoncalves,ChirreANote,Yang2024}). Note that these results are not explicit according to the definition in this paper.

It is known~\cite[Propositon 1.1]{AistleitnerPankowski2017} that if GRH and 
\begin{equation}
\label{eq:SelbergNormality}
    \sum_{p \leq x} \left|a(p)\right|^2 =\left(\kappa+o(1)\right)\frac{x}{\log{x}}, \quad \kappa>0,
\end{equation}
hold for $\cL$, then for every fixed constant $\sigma \in [1/2,1)$ there exists a constant $c>0$ such that 
\begin{equation}
\label{eq:AP}
\cL(\sigma+\ie t) \ll \exp\left(c\frac{(\log{t})^{2-2\sigma}}{\log{\log{t}}}\right)
\end{equation}
for $t$ large enough. Recently, the authors~\cite{PalojarviSimonic} extended methods that had been used to provide explicit bounds for logarithmic derivatives of Dirichlet $L$-functions~\cite{ChirreSimonicHagen}, and provided uniform and explicit upper bounds for moduli of the logarithm and logarithmic derivative for the Selberg class functions assuming GRH and additional hypothesis. For example, we provided a precise form of~\eqref{eq:AP} under the same assumptions and when $\sigma\in(1/2,1)$ is fixed. The results in~\cite{PalojarviSimonic} were stated for the region $\sigma \geq 1/2+\alpha_1/\log{\log{\tau}}$, where $\tau=\sq|t|^{\sdeg}$ and $\alpha_1>\frac{1}{2}\log{2}$. This leaves the case 
\begin{equation}
\label{eq:toBeInvestigated}
    \frac{1}{2}<\sigma<\frac{1}{2}+\frac{\alpha_1}{\log{\log{\tau}}}
\end{equation} 
to be investigated in this paper. In particular, we also provided a precise form of~\eqref{eq:AP} for $\sigma=1/2$.

Assuming that a Selberg class function $\mathcal{L}(s)$ has a polynomial Euler product representation (of order $m$) and that~\eqref{eq:SelbergNormality} holds, Aistleitner and Pa\'{n}kowski~\cite{AistleitnerPankowski2017} provided explicit estimates for the main term of the maximum of the lower bound of the function $\mathcal{L}(s)$. Their statement is that
\begin{equation*}
    \max_{t \in [T,2T]}\left|\mathcal{L}(\sigma+\ie t)\right| \geq \exp\left(\left(\kappa^\sigma m^{1-2\sigma}\frac{(3-2\sigma)^{3/2-\sigma}}{2\sqrt{2\sigma-1}}+o(1)\right)\frac{\left(\log{T}\right)^{1-\sigma}}{\log{\log{T}}}\right) ,
\end{equation*} 
where $\sigma \in (1/2,1)$ and $T$ is large enough. They also provided a similar type of bound in the case $\sigma=1/2$. The current best explicit lower bound for the absolute value of the Riemann zeta function in the region $\sigma \in (1/2,1)$ and under RH is provided by the second author in~\cite{SimonicCS}, where the case $1\leq \sigma<3/2$ is also considered. In general, there seems to be very few known conditional and explicit lower bounds for other Selberg class functions than the Riemann zeta function when $\sigma \in (1/2,1)$.

A typical method to derive upper bounds for the Selberg class functions is to find a suitable way to manipulate these into sums of primes and zeros and then estimate the sums itself in meaningful ways (see e.g. \cite{CarneiroChandee,Carneiro,ChandeeSound,ChirreANote,chirreGoncalves,ChirreSimonicHagen,PalojarviSimonic, SoundMoments}). However, finding a good way to write the sums or estimate them, is not trivial and also depends on the region under consideration. The approach that recasts the sums using (a modified version of) Selberg's moment formula \cite[Section 13.2]{MontgomeryVaughan} and \cite[Theorems 5.17 and 5.19]{IKANT}, as for example in \cite{ChirreSimonicHagen,PalojarviSimonic}, or directly using the functional equation (e.g. \cite{CarneiroChandee,ChandeeSound,chirreGoncalves}), have proven to be successful and allow many generalizations. To estimate the terms coming from the sums, \text{bandlimited majorants} are often used and then the Guinand-Weil explicit formula is applied to the majorants. \textit{Bandlimited} majorants refer to majorants that have compactly supported Fourier transforms and the Guinand-Weil explicit formula connects a sum over zeros to Fourier transforms. The method of applying majorants and the Guinand-Weil explicit formula in this setting was first used by Goldston and Gonek~\cite{GoldstonGonek} to estimate the argument of the Riemann zeta function and by Chandee and Soundararajan~\cite{ChandeeSound} to estimate the Riemann zeta function on the critical line. 

Explicit lower bounds can be derived using similar ideas of bandlimited minorants and the Guinand--Weil explicit formula as in the case with the upper bounds~\cite{CarneiroChandee}. However, several other methods have been successfully used to derive lower bounds. These include applying suitable integral representations of $\log{\zeta(s)}$, see~\cite{SimonicCS}, the resonance method~\cite{AistleitnerPankowski2017,SoundResonance2008,YangD2023} and Diophantine approximations~\cite{AistleitnerPankowski2017,PankowskiSteuding2013}. The resonance method is based on an idea to reduce the problem of estimating certain weighted means of an $L$-function, see~\cite{AistleitnerPankowski2017} for the comparison between the resonance method and Diophantine approximations, and is a powerful method to deduce $\Omega$-results.

In this paper, we apply Carneiro--Chandee's methods~\cite{CarneiroChandee} for the Riemann zeta function combined with refinements by Bui and Florea~\cite{BF2023} and our own ones to derive conditional upper and lower bounds for $\log{\left|\cL(s)\right|}$ for functions in the Selberg class with $\sigma \in (1/2,1)$. In addition, upper bounds are derived in the case of $\sigma=1/2$, see Corollaries~\ref{cor:RiemannZeta},~\ref{cor:RiemannZeta2} and~\ref{cor:CLFullSelberg}, and upper and lower bounds in the case of $0<\sigma-1/2\ll 1/\log{\log{\tau}}$ where $\tau$ is defined by~\eqref{def:tau}, see Corollary~\ref{cor:RiemannZeta3} and Theorems~\ref{thm:Polynomial},~\ref{thm:ExplicitUpper} and~\ref{thm:ExplicitLower}. In particular, we improve~\cite[Equation~(2.8)]{BF2023} and also~\cite[Corollary~1.4]{ChandeeExplBounds} for $t\geq\exp{\left(\exp{(19)}\right)}$. The results contain asymptotic and explicit estimates, and the proofs involve computational methods. We apply the functional equation to recast the question to estimate a certain sum over zeros, and then apply bandlimited majorants and minorants as well as the Guinand--Weil explicit formula to estimate the sum. The main results are described in Section~\ref{sec:Main}. They are uniform in $\cL$, that is, the implied $O$-constants are independent of the parameters from the axioms of the Selberg class. 
It should be noted (as mentioned in~\cite[p.~3]{PalojarviSimonic}) that in~\cite[Theorems~1 and~2 (cases ``otherwise'')]{CarneiroChandee} the coefficients of the main error terms should be 
\begin{equation*}
\pm \frac{1}{2}\left(1+\frac{2\sigma-1}{\sigma(1-\sigma)}\right) \quad\textrm{instead of}\quad \pm \left(\frac{1}{2}+\frac{2\sigma-1}{\sigma(1-\sigma)}\right), 
\end{equation*}
see also Theorem~\ref{thm:Polynomial}. In addition to this, for $\sigma\in[1/2+\delta,1-\delta]$ and some $\delta>0$, results presented here are asymptotically better than the relevant results in~\cite{PalojarviSimonic} for $m\leq 2$. Interestingly, this is no longer true if $m\geq 3$.

The outline of this paper is as follows. In Section~\ref{sec:Main} we state the main results. In Section~\ref{sec:Sum} we show how $\log{\left|\cL(s)\right|}$ can be estimated with a certain sum over the non-trivial zeros, see Proposition~\ref{prop:LogSum}. This sum is estimated in a general setting and by means of the Guinand--Weil formula (Lemma~\ref{lem:GW}) in Section~\ref{sec:Estimatef}, see Proposition~\ref{prop:FourierGamma} and Corollary~\ref{corollary:UpperGW}. Explicit estimates for extremal functions are provided in Section~\ref{sec:gmExplicitBounds}, while estimates concerning the generalized von Mangoldt function are provided in~\ref{sec:primes}. The proofs of the main results are given in Section~\ref{sec:ProofsMain}.

\subsection*{Acknowledgements}
The work of Neea Palojärvi was carried out primarily at the University of Helsinki. She was supported by the Emil Aaltonen and Finnish Cultural Foundations.

\subsection{Notation}
\label{subsec:Def}
We have already seen a few definitions in Section~\ref{subsec:Basics}. Let us also denote other definitions used in this paper. We start by defining several parameters that come from the functional equation:
\begin{gather}
\tau\de\sq\left(\frac{|t|}{2\pi}\right)^{\sdeg}, \quad \lambda^{-} \de \min_{1\leq j\leq f}\left\{\lambda_{j}\right\}, \quad \lambda^{+} \de \max_{1\leq j\leq f}\left\{\lambda_{j}\right\}, \label{def:tau} \\
a^{+} \de \max_{1\leq j\leq f}\left\{\frac{1}{\lambda_{j}}\Re\left\{\mu_{j}\right\}\right\}, \quad 
b^{+} \de \max_{1\leq j\leq f}\left\{\frac{1}{\lambda_{j}}\left|\Im\left\{\mu_{j}\right\}\right|\right\}. \label{def:a+b+}
\end{gather}
Note that $\tau$ from~\eqref{def:tau} differs slightly from $\tau$ which is used in~\cite{PalojarviSimonic}. Let $\varepsilon \in (0,1/2)$. By the Ramanujan hypothesis and the Euler product representation, we have positive numbers $\mathcal{C}_{\cL}(\varepsilon)$, $\mathcal{C}_{\cL}^{E}$ and $\theta \in [0,1/2)$ for which
\begin{equation}
\label{def:CER}
    |a(n)|\leq\mathcal{C}^R_{\cL}(\varepsilon)n^{\varepsilon} \quad\text{and} \quad\left|b(p^k)\right| \leq \mathcal{C}_{\cL}^{E} p^{k\theta}
\end{equation}
for all $n\geq 2$ and all prime numbers $p$ and all $k\geq 1$. 

In addition, we need some functions that appear in extremal problems. Let
\begin{equation}
\label{def:falpha}
  f_{\sigma}(x) \de \log{\left(\frac{4+x^2}{\left(\sigma-\frac{1}{2}\right)^{2}+x^2}\right)} \quad\text{and}\quad  F_\Delta(x):=\log\left(\frac{x^2+4\Delta^2}{x^2+\left((\sigma-1/2)\Delta\right)^2}\right),
\end{equation}
where $x, \sigma$ and $\Delta>0$ are real numbers. Note that $f_\sigma(x)=F_\Delta\left(\Delta x\right)$. Let us also define
\begin{equation}
\label{eq:DefGDelta}
    G_\Delta(z)=\left(\frac{\cos(\pi z)}{\pi}\right)^2 \sum_{n=-\infty}^\infty \left(\frac{F_\Delta(n-1/2)}{(z-n+1/2)^2}+\frac{F_\Delta'(n-1/2)}{(z-n+1/2)}\right) \quad\text{and}\quad g_\Delta(z):=G_\Delta(\Delta z),
\end{equation}
and further
\begin{equation}
\label{def:Mandm}
    M_{\Delta}(z) := \sum_{n=-\infty}^{\infty}\left(\frac{\sin\left(\pi(z-n)\right)}{\pi(z-n)}\right)^{2}\left(f_\sigma\left(\frac{n}{\Delta}\right)+\frac{z-n}{\Delta}f_\sigma'\left(\frac{n}{\Delta}\right)\right) \quad\text{and}\quad  m_{\Delta}(z):=M_\Delta(\Delta z).
\end{equation}

Let $\widehat{g}_\Delta(\xi)$ be the Fourier transform of $g_\Delta$, namely
\begin{equation*}
    \widehat{g}_\Delta(\xi):=\int_{-\infty}^\infty g_\Delta(x)e^{-2\pi i \it x \xi}\dif{x},
\end{equation*}
where $\xi$ is a real number. Similarly, let $\widehat{m}_\Delta(\xi)$ be the Fourier transform of $m_\Delta$. By~\cite[Lemmas 5(ii) and 8(ii)]{CarneiroChandee} they are continuous real-valued functions supported on the interval $[-\Delta,\Delta]$. 

To formulate precise results, we also define
\begin{equation}
    \label{def:theta}
    \theta_1(u) \de \frac{e^{u}-u-1}{u^2}, \quad \Theta_{\theta, \varepsilon} (\sigma)\de \frac{1}{2}+\max\left\{\theta, \varepsilon\right\}-\sigma, 
\end{equation}
as well as
\begin{equation}
\label{def:A1}
    A_1\left(a,\varepsilon, \sigma\right) \de \frac{a\left(2\sigma-1\right)(1+\varepsilon)}{(\sigma+\varepsilon)(1-\sigma+\varepsilon)},
\end{equation}
\begin{equation}
\label{def:Ai}
    A_2\left(a,\varepsilon, \theta, \sigma, x\right) \de a\left(1+\left|\Theta_{\theta, \varepsilon} (\sigma)\right| x^{\Theta_{\theta, \varepsilon} (\sigma)}\log{x}\right)\min\left\{\frac{1}{\left|\Theta_{\theta, \varepsilon} (\sigma)\right|^{2}}, \log^{2}{x}\right\}
\end{equation}
and
\begin{equation}
\label{def:A6}
A_3(a,\sigma,x):=a\left(\log{x} + (1-\sigma)\log{x}\log{\log{x}} + \frac{x^{1-\sigma}}{(1-\sigma)\log{x}}\right).
\end{equation}
Let the order of the pole of $\cL(s)$ at $s=1$ be $m_{\cL}$. We define also
\begin{flalign}
\label{def:A4}
A_4 &\de \left|\log{\left|\mathcal{L}\left(\frac{5}{2}+\ie t\right)\right|}\right| + \frac{m_{\cL}\log{\tau}\log\log{\tau}}{|t|} \nonumber \\ 
&+ \frac{\sdeg (1+a^++b^+)\max\left\{1,\lambda^{+}\right\}}{\min\{1, (\lambda^-)^4\}}\sqrt{\frac{\log\log{\tau}}{|t|}}\log{\left(1+\sqrt{|t|\log\log{\tau}}\right)} \nonumber \\ 
&+ \max\left\{\sdeg,2\sdeg+\Re\left\{\xi_{\cL}\right\}\right\}\log{\left(1+\frac{1+a^{+}+b^{+}}{|t|}\right)} + \frac{\sdeg+m_{\cL}+f/\min\left\{1,\left(\lambda^{-}\right)^{3}\right\}}{t^2}
\end{flalign}
and
\begin{flalign}
\label{def:A5}
A_5 &\de \left|\log{\left|\mathcal{L}\left(\frac{5}{2}+\ie t\right)\right|}\right| + \frac{m_\cL\log{\tau}\log\log{\tau}}{|t|} + \frac{m_\cL \log{\tau}}{t^2(\log\log{\tau})^2}\log{\frac{2}{\sigma-\frac{1}{2}}} \nonumber \\
&+ \frac{\sdeg (1+a^++b^+)\max\left\{1,\lambda^{+}\right\}}{\min\{1, (\lambda^-)^4\}}\left(1+\frac{1}{\left(\log{\log{\tau}}\right)^{2}}\log{\frac{2}{\sigma-\frac{1}{2}}}\right)\sqrt{\frac{\log\log{\tau}}{|t|}}\log{\left(1+\sqrt{|t|\log\log{\tau}}\right)} \nonumber \\
&+ \sdeg\log{\left(1+\frac{1+a^{+}+b^{+}}{|t|}\right)} + \frac{\sdeg\max\left\{1,\left(a^{+}\right)^2\right\}+f/\min\left\{1,\left(\lambda^{-}\right)^{3}\right\}}{t^{2}}.
\end{flalign}
The next functions are used in explicit results from Section~\ref{sec:ExplicitResults}. In Theorem~\ref{thm:ExplicitUpper} we are using
\begin{equation}
\label{eq:calA1}
\mathcal{A}_1(\alpha,\tau) \de 2\left(1+\frac{2\alpha}{1-\left(\frac{2\alpha}{\log{\log{\tau}}}\right)^{2}}-\frac{2\alpha}{e^{2\alpha}\left(e^{2\alpha}+1\right)}\right),
\end{equation}
\begin{equation}
\label{eq:calA2}
\mathcal{A}_2(\alpha,\nu,\tau) \de 6\left(\frac{12\alpha(1+\alpha)+4\alpha\left(\frac{2\alpha}{\log{\log{\tau}}}\right)^2}{3\left(1-\left(\frac{2\alpha}{\log{\log{\tau}}}\right)^2\right)^{3}}+1+\frac{\frac{1}{\nu^4}+\left(\frac{2\log{\log{\tau}}}{\log{2}}\right)^{4}\left(\log{\tau}\right)^{\nu-1}}{\left(1-\frac{2\alpha}{\log{\log{\tau}}}\right)^{4}}\right),
\end{equation}
\begin{equation}
\label{eq:calA3}
\mathcal{A}_3(\alpha,\sigma,\tau) \de \frac{\sigma}{4\pi}+\frac{8\alpha}{3}+\frac{\frac{1}{4\pi}+\frac{8\alpha}{3}}{\log{\log{\tau}}} + \frac{1}{\left(\log{\log{\tau}}\right)^{2}}\left(2.31 + \frac{2\alpha\sqrt{2}\left(1+\frac{2-\log{2}}{2\log{\log{\tau}}}\right)}{\log{\log{\tau}}-2\alpha}\right),
\end{equation}
\begin{equation}
\label{eq:calA4}
\mathcal{A}_{4}(\nu_1,\sigma) \de \frac{2^{\frac{1}{2}-\sigma}\left(4\sigma-1+\sigma(2\sigma-1)\log{2}\right)}{4\pi(2\sigma-1)^{2}}+\log{\frac{1}{\log{2}}}+\frac{1-\sigma2^{1-\sigma}}{(1-\sigma)\log{2}}+\int_{0}^{\nu_1}\theta_{1}(u)\dif{u}+0.3,
\end{equation}
\begin{equation}
\label{eq:calA5}
\mathcal{A}_{5}(\sigma,\tau) \de \frac{1}{4\pi(2\sigma-1)} + \frac{2^{\sigma}(1-\sigma)}{\sigma\left(4\log{\log{\tau}}-\log{2}\right)(\log{\tau})^{2\sigma-1}}+\frac{0.87(\log{\tau})^{2\sigma-1}}{2\left(\log^{4}{\tau}-1\right)\log{\log{\tau}}},
\end{equation}
\begin{equation}
\label{eq:calA6}
\mathcal{A}_{6}(\sigma,\tau)\de \frac{1}{2\log{\log{\tau}}}\left(\frac{2^{1-\sigma}}{1-\sigma}+\frac{\sigma2^{\frac{5}{2}-\sigma}(8+4(2\sigma-1)\log{2}+(2\sigma-1)^{2}\log^{2}{2})}{(2\sigma-1)^3}\right),
\end{equation}
while in Theorem~\ref{thm:ExplicitLower} we are also using
\begin{equation}
\label{eq:calB1B2}
\mathcal{B}_{1}(\alpha,\tau) \de 4\left(1+\frac{\alpha}{1-\left(\frac{2\alpha}{\log{\log{\tau}}}\right)^2}\right), \quad \mathcal{B}_{2}(\tau) \de 0.48+\frac{1.76}{\log{\log{\tau}}},
\end{equation}
\begin{equation}
\label{eq:calB3B4}
\mathcal{B}_{3}(\sigma,\tau)\de \frac{1}{2^{\frac{3}{2}+\sigma}\pi} + \frac{2^{1-\sigma}\left(1-\sigma-\sigma^2\right)(2\sigma-1)}{4\sigma(1-\sigma)(\log{\log{\tau}})^{2}}, \quad \mathcal{B}_{4}(\sigma,\tau)\de \frac{\sigma}{2^{\frac{1}{2}+\sigma}\pi} + \frac{2^{2-\sigma}(2\sigma-1)}{4(\log{\log{\tau}})^2}.
\end{equation}
In addition,
\begin{multline}
\label{def:EUpper}
E_{\cL}^{\uparrow}(t)\de 0.3m + \sdeg\left(14\sqrt{\frac{\log{\log{\tau}}}{|t|}}\log{\left(1+\sqrt{\frac{|t|}{\pi}\log{\log{\tau}}}\right)}+\frac{0.02}{\left(\log{\log{\tau}}\right)^{2}}+\frac{0.52}{\sqrt{|t|\log{\log{\tau}}}}\right) 
 \\
+ \frac{0.02\Re\left\{\xi_\cL\right\}}{\left(\log{\log{\tau}}\right)^{2}} + m_{\cL}\left(\frac{28\log{\tau}\log{\log{\tau}}}{\pi|t|}+ \frac{4.3}{t^2}\right)
\end{multline}
and
\begin{multline}
\label{def:ELower}
E_{\cL}^{\downarrow}(\sigma,t)\de 0.3m + \sdeg\left(\left(8+\frac{1}{\log\log{\tau}}\log{\frac{2}{\sigma-\frac{1}{2}}}\right)\sqrt{\frac{\log\log{\tau}}{|t|}}\log{\left(1+\sqrt{\frac{|t|}{\pi}\log\log{\tau}}\right)}\right. \\
\left.+\frac{0.08}{\left(\log{\log{\tau}}\right)^{2}}+\frac{0.261}{\sqrt{|t|\log{\log{\tau}}}}\right)+m_{\cL}\left(\frac{4\log{\tau}\log\log{\tau}}{\pi |t|} + \frac{4\log{\tau}}{\left(|t|\log{\log{\tau}}\right)^{2}}\log{\frac{2}{\sigma-\frac{1}{2}}}\right).
\end{multline}
In this paper, we denote non-trivial zeros as $\rho=1/2+\gamma \ie$ since we assume the Generalized Riemann hypothesis (GRH). 

\section{The main results}
\label{sec:Main}

The main results are divided into different sections based on the assumptions used and the level of explicitness of the results. We are stating our results uniformly in $\cL$, i.e., the implied $O$-constants are independent of the parameters from the axioms of the Selberg class. We start with the explicit results for the Riemann zeta function followed by the more general ones to illustrate the size of the results.

\subsection{Results for the Riemann zeta-function}

To the best of authors' knowledge, there are only two papers~\cite{ChandeeExplBounds} and~\cite{SimonicSonRH} that are providing conditional (RH) and explicit estimates for the Riemann zeta-function on the critical line. The method in both papers is based on Soundararajan's techniques from~\cite{SoundMoments}, which implies that the leading constant $C$ from~\eqref{eq:LittlewoodSpecial} should be $\geq 0.3728\ldots$ in any of such results. For example, 
\begin{equation}
\label{eq:ChandeeEst}
\left|\zeta\left(\frac{1}{2}+\ie t\right)\right| \leq \exp{\left(\frac{3}{8}\frac{\log{t}}{\log{\log{t}}}+\frac{23}{25}\frac{\log{t}}{\left(\log{\log{t}}\right)^{2}}\right)}
\end{equation}
for $t\geq\exp{\left(\exp{(10)}\right)}$ and $\left|\zeta\left(1/2+\ie t\right)\right| \leq \exp{\left(0.995\log{t}/\log{\log{t}}\right)}$ for $t\geq7$, see~\cite[Corollary~1.4]{ChandeeExplBounds} and~\cite[Corollary~1]{SimonicSonRH}, respectively. In our case, we are able to improve the constant in front of the first term to $\frac{1}{2}\log{2} \approx 0.357$, which is currently still the best result due to Chandee and Soundararajan~\cite{ChandeeSound}. The next corollary gives a better estimate than~\eqref{eq:ChandeeEst} for all $t\geq\exp{\left(\exp{(19)}\right)}$. 

\begin{corollary}
\label{cor:RiemannZeta}
Assume the Riemann Hypothesis. Then
\[
\left|\zeta\left(\frac{1}{2}+\ie t\right)\right| \leq \exp{\left(\frac{\log{2}}{2}\frac{\log{t}}{\log{\log{t}}} + 1.455\frac{\log{t}}{\left(\log{\log{t}}\right)^2}\right)}
\]
for $t\geq\exp{\left(\exp{(19)}\right)}$.
\end{corollary}

With the techniques presented here, we can actually show the following.

\begin{corollary}
\label{cor:RiemannZeta2}
Assume the Riemann Hypothesis. Then
\[
\left|\zeta\left(\frac{1}{2}+\ie t\right)\right| \leq \exp{\left(\frac{\log{2}}{2}\frac{\log{t}}{\log{\log{t}}}+\frac{0.9541\log{t}}{\left(\log{\log{t}}\right)^2}+O\left(\frac{\log{t}}{\left(\log{\log{t}}\right)^{3}}\right)\right)}
\]
for sufficiently large $t$.
\end{corollary}

The constant in the second term in Corollary~\ref{cor:RiemannZeta2} was slightly improved to $0.8271$ in~\cite[Theorem~2]{CM2024}. However, one can easily deduce an explicit version of Corollary~\ref{cor:RiemannZeta2}, although at the expense of requiring a very large $t$.

For the lower estimate on $\log|\zeta(s)|$ when $s$ is close to the critical line, we have the following. 

\begin{corollary}
\label{cor:RiemannZeta3}
Assume the Riemann hypothesis. Let $\alpha>0$ be bounded and $0<\sigma-1/2\leq\alpha/\log{\log{t}}$, where $t>0$ is sufficiently large. Then
\begin{flalign}
\label{eq:zetalower}
\log{\left|\frac{1}{\zeta(s)}\right|} &\leq \frac{\log{t}}{2\log{\log{t}}}\log{\frac{1}{1-(\log{t})^{1-2\sigma}}} \nonumber \\
&+ \left(2\log{\frac{1}{1-(\log{t})^{1-2\sigma}}}+4\left(1+\alpha\right)+O\left(\frac{1}{(\log{\log{t}})^2}\right)\right)\frac{\log{t}}{(\log{\log{t}})^2} \nonumber 
\\
&+ O\left((\log{\log{t}})^2\log{\frac{1}{1-(\log{t})^{1-2\sigma}}} + (\log{t})\sqrt{\frac{\log{\log{t}}}{t}}\left(1+\frac{1}{\log{\log{t}}}\log{\frac{2}{\sigma-\frac{1}{2}}}\right)\right).
\end{flalign}
\end{corollary}

Note that 
\[
\log{\frac{2}{\sigma-\frac{1}{2}}} \leq \log{\frac{1}{1-\left(\log{t}\right)^{1-2\sigma}}} + \log{\left(4\log{\log{t}}\right)}
\]
for $t\geq 3$ and $\sigma>1/2$, and hence~\eqref{eq:zetalower} is an improvement over~\cite[Equation~(2.8)]{BF2023}.

\subsection{General estimates for $\cL\in\cS$}
In this section we state upper and lower estimates for $\log{\left|\cL(s)\right|}$ and $\sigma\in(1/2,1)$ for the full Selberg class $\cS$, without additional assumptions apart from GRH for $\cL(s)$ and RH for $\zeta(s)$. We are requiring that
\begin{gather}
|t|\geq \max\left\{b^{+} + \sqrt{\left(1+a^{+}\right)\left(\frac{5}{2}+a^{+}\right)}, 4\left(a^{+}+b^{+}+1\right)\right\}, \label{eq:conditions1} \\ \sqrt{\frac{|t|}{\log{\log{\tau}}}} \geq 17, \quad |t|-\sqrt{|t|\log{\log{\tau}}}-b^{+}\geq 17. \label{eq:conditions2}
\end{gather}
As in~\cite{PalojarviSimonic}, by \emph{sufficiently large $\tau$ and $|t|$} we mean that there exists an absolute $T>0$ such that the results are valid for all $\tau\geq T$ and $|t|\geq T$.

\begin{theorem}
\label{thm:MainFullSelberg}
Let $\cL\in\cS$. Assume the Riemann Hypothesis for $\zeta(s)$ and the Generalized Riemann Hypothesis for $\cL(s)$. Then
\begin{multline*}
\log{\left|\mathcal{L}\left(\sigma+\ie t\right)\right|} \leq \frac{\log{\tau}}{2\log{\log{\tau}}}\log{\left(1+\left(\log{\tau}\right)^{1-2\sigma}\right)} + A_1\left(\mathcal{C}_{\cL}^{R}(\varepsilon), \varepsilon,\sigma\right)\frac{(\log{\tau})^{2(1-\sigma+\varepsilon)}}{2\log\log{\tau}} \\
+ (1+\varepsilon)\mathcal{C}_{\cL}^{R}(\varepsilon)\log{\left(2\log{\log{\tau}}\right)}
+ O\left(\frac{\mathcal{C}_{\cL}^{R}(\varepsilon)\left(\log{\tau}\right)^{2(1-\sigma+\varepsilon)}}{(1-\sigma+\varepsilon)^{2}\left(\log{\log{\tau}}\right)^{2}} + A_2 + A_4\right)
\end{multline*}
and
\begin{multline*}
\log{\left|\mathcal{L}\left(\sigma+\ie t\right)\right|} \geq \frac{\log{\tau}}{2\log{\log{\tau}}}\log{\left(1-\left(\log{\tau}\right)^{1-2\sigma}\right)} - A_1\left(\mathcal{C}_{\cL}^{R}(\varepsilon), \varepsilon,\sigma\right)\frac{(\log{\tau})^{2(1-\sigma+\varepsilon)}}{2\log\log{\tau}} \\
- (1+\varepsilon)\mathcal{C}_{\cL}^{R}(\varepsilon)\log{\left(2\log{\log{\tau}}\right)} - O\left(\left(1+\log{\frac{1}{1-(\log{\tau})^{1-2\sigma}}}\right)\left(\frac{\mathcal{C}_{\cL}^{R}(\varepsilon)(\log{\tau})^{2(1-\sigma+\varepsilon)}}{(1-\sigma+\varepsilon)^2(\log{\log{\tau}})^{2}}+A_2\right) + A_5\right)
\end{multline*}
for $1/2<\sigma<1$, $\varepsilon\in(0,1/2)$, and for sufficiently large $\tau$ and $|t|$ that obey~\eqref{eq:conditions1} and~\eqref{eq:conditions2}. Here, $A_2=A_2\left(\mathcal{C}_{\cL}^{R}(\varepsilon)+\mathcal{C}_{\cL}^{E},\varepsilon, \theta, \sigma, (\log{\tau})^{2}\right)$, $A_4$ and $A_5$ are defined by~\eqref{def:Ai},~\eqref{def:A4} and~\eqref{def:A5}, respectively.
\end{theorem}

\begin{remark}
Note that for fixed $\cL$ we have $A_4\ll 1$ and 
\[
A_5\ll 1+\sqrt{\frac{1}{|t|\left(\log{\log{\tau}}\right)^{3}}}\log{\left(1+\sqrt{|t|\log\log{\tau}}\right)}\log{\frac{2}{\sigma-\frac{1}{2}}}.
\]
for sufficiently large $\tau$ and $|t|$. The next result is the conditional estimate for $\cL(s)$ on the critical line.
\end{remark}

\begin{corollary}
\label{cor:CLFullSelberg}
Let $\cL\in\cS$. Assume the Riemann Hypothesis for $\zeta(s)$ and the Generalized Riemann Hypothesis for $\cL(s)$. Then
\[
\left|\cL\left(\frac{1}{2}+\ie t\right)\right| \leq \exp{\left(\frac{(1+2\varepsilon)\log{2}}{2}\frac{\log{\tau}}{\log{\log{\tau}}}+O\left(\frac{\mathcal{C}_{\cL}^{R}(\varepsilon)\log{\tau}}{\left(\log{\log{\tau}}\right)^{2}}+A_2+A_4\right)\right)}
\]
for $\varepsilon \in (0,1/2)$ and for sufficiently large $\tau$ and $|t|$ that obey~\eqref{eq:conditions1} and~\eqref{eq:conditions2}. Here, 
\[
A_2=A_2\left(\mathcal{C}_{\cL}^{R}(\varepsilon)+\mathcal{C}_{\cL}^{E},\varepsilon, \theta,1/2, (\log{\tau})^{2/(1+2\varepsilon)}\right) \ll \left(\mathcal{C}_{\cL}^{R}(\varepsilon)+\mathcal{C}_{\cL}^{E}\right)\left(\log{\tau}\right)^{\frac{2\max\{\theta,\varepsilon\}}{1+2\varepsilon}}\left(\log{\log{\tau}}\right)^{3}
\]
and $A_4$ are defined by~\eqref{def:Ai} and~\eqref{def:A4}, respectively.
\end{corollary}

Corollary~\ref{cor:CLFullSelberg} is, strictly speaking, not a corollary to Theorem~\ref{thm:MainFullSelberg} because the error term there is larger than the main term when taking $\sigma\to1/2$. However, essentially the same method is used to prove both results. In particular, Corollary~\ref{cor:CLFullSelberg} shows that, for fixed $\cL\in\cS$, one can recover the conditional estimate for $\cL(s)$ on the critical line, which is of the same asymptotic strength as one for the Riemann zeta-function.

\subsection{Asymptotic results under additional conjectures}
In the previous section, our estimates contain terms that include $\varepsilon$ from the Ramanujan hypothesis (axiom~(i)). Hence, we are studying the problem under additional assumptions. These assumptions contain a polynomial Euler product, Selberg's conjecture, Selberg's normality conjecture, as well as other known conjectures. We refer to Sections~2.2 and~2.3 in~\cite{PalojarviSimonic} for a more elaborate introduction to these assumptions.

The first two conjectures describe the sum of the second powers of the coefficients $a(p)$.

\begin{conjecture}
\label{conj:SelbergVariant}
Let $\cL\in\cS$. Then 
\[
\sum_{p\leq x}\left|a(p)\right|^2 \leq \mathcal{C}_{\cL}^{P_1}(x)\frac{x}{\log{x}} + \mathcal{C}_{\cL}^{P_2}\frac{x}{\log^{2}{x}}
\]
for all $x\geq 2$, where $0<\mathcal{C}_{\cL}^{P_1}(x)=o(x^{\varepsilon})$ for every $\varepsilon>0$, $\mathcal{C}_{\cL}^{P_1}(x)$ is an increasing continuous function, there exists an absolute $x_0>0$ such that $x^{-1/2}\mathcal{C}_{\cL}^{P_1}(x)$ is a decreasing function for $x\geq x_0$, and $\mathcal{C}_{\cL}^{P_2}$ is a non-negative constant.
\end{conjecture}

\begin{conjecture}
\label{conj:SelbergVariant2}
Let $\cL\in\cS$. Then 
\[
\sum_{p\leq x}\left|a(p)\right|\leq \widehat{\mathcal{C}_{\cL}^{P_1}}(x)\frac{x}{\log{x}} + \widehat{\mathcal{C}_{\cL}^{P_2}}\frac{x}{\log^{2}{x}}
\]
for all $x\geq 2$, where $0<\widehat{\mathcal{C}_{\cL}^{P_1}}(x)=o(x^{\varepsilon})$ for every $\varepsilon>0$, $\widehat{\mathcal{C}_{\cL}^{P_1}}(x)$ is an increasing continuous function, there exists an absolute $x_0>0$ such that $x^{-1/4}\widehat{\mathcal{C}_{\cL}^{P_1}}(x)$ is a decreasing function for $x\geq x_0$, and $\widehat{\mathcal{C}_{\cL}^{P_2}}$ is a non-negative constant.
\end{conjecture}

Let us introduce one more notation before describing the result under the previous conjectures. 
\begin{enumerate}
    \item If Conjecture~\ref{conj:SelbergVariant} is true for $\cL$, then
    \begin{equation}
    \label{eq:m123}
    m_1(x)\de \sqrt{\mathcal{C}_{\cL}^{P_1}\left(x\right)}, \quad m_2(x)\de \sqrt{\mathcal{C}_{\cL}^{P_1}\left(x\right)+\mathcal{C}_{\cL}^{P_2}}, \quad
    m_3(x)\de \frac{m_2(x)}{\sqrt{\log{x}}}.
    \end{equation}
    \item If Conjecture~\ref{conj:SelbergVariant2} is true for $\cL$, then
    \begin{equation}
    \label{eq:m123V2}
    m_1(x)\de \widehat{\mathcal{C}_{\cL}^{P_1}}\left(x\right), \quad m_2(x)\de \widehat{\mathcal{C}_{\cL}^{P_2}}, \quad
    m_3(x)\de \frac{m_2(x)}{\log{x}}.
    \end{equation}
\end{enumerate}

Now we are ready to present the results under Conjectures~\ref{conj:SelbergVariant} and ~\ref{conj:SelbergVariant2}.

\begin{theorem}
\label{thm:MainConjectures}
Let $\cL\in\cS$. Assume the Riemann Hypothesis for $\zeta(s)$ and the Generalized Riemann Hypothesis for $\cL(s)$. Let $\sigma\in(1/2,1)$, and let $\widetilde{m}_{i}(\tau)\de m_i\left(\log^{2}{\tau}\right)$, where $m_i(x)$ are defined by~\eqref{eq:m123} if Conjecture~\ref{conj:SelbergVariant} holds, and by~\eqref{eq:m123V2} if Conjecture~\ref{conj:SelbergVariant2} holds. Then
\begin{multline*}
\log{\left|\cL(\sigma+\ie t)\right|} \leq \frac{\log{\tau}}{2\log{\log{\tau}}}\log{\left(1+\left(\log{\tau}\right)^{1-2\sigma}\right)} + A_1\left(\widetilde{m}_{1}(\tau),0,\sigma\right)\frac{(\log{\tau})^{2-2\sigma}}{2\log{\log{\tau}}} \\
+\widetilde{m}_{1}(\tau)\log\left(2\log{\log{\tau}}\right)
+O\left(\frac{\widetilde{m}_{1}(\tau)\left(\log{\tau}\right)^{2-2\sigma}}{(1-\sigma)^2\left(\log{\log{\tau}}\right)^2}+A_2+A_3+A_4\right)
\end{multline*}
and
\begin{multline*}
\log{\left|\cL(\sigma+\ie t)\right|} \geq \frac{\log{\tau}}{2\log{\log{\tau}}}\log{\left(1-\left(\log{\tau}\right)^{1-2\sigma}\right)} - A_1\left(\widetilde{m}_{1}(\tau), 0,\sigma\right)\frac{(\log{\tau})^{2-2\sigma}}{2\log{\log{\tau}}} \\ 
- \widetilde{m}_{1}(\tau)\log{\left(2\log\log{\tau}\right)}
-O\left(\left(1+\log{\frac{1}{1-(\log{\tau})^{1-2\sigma}}}\right)\left(\frac{\widetilde{m}_{1}(\tau)(\log{\tau})^{2-2\sigma}}{(1-\sigma)^2\left(\log{\log{\tau}}\right)^{2}}+A_2+A_3\right)+A_5\right)
\end{multline*}
for sufficiently large $\tau$ and $|t|$ that obey~\eqref{eq:conditions1} and~\eqref{eq:conditions2}. Here, the functions $A_2=A_2\left(\mathcal{C}_{\cL}^{E},0,\theta,\sigma,\log^{2}{\tau}\right)$, $A_3=A_3\left(\widetilde{m}_{3}(\tau),\sigma,\log^{2}{\tau}\right)$, $A_4$ and $A_5$ are defined by~\eqref{def:Ai},~\eqref{def:A6},~\eqref{def:A4} and~\eqref{def:A5}, respectively.
\end{theorem}

Lastly, we present asymptotic results assuming that $\cL$ has a polynomial Euler product.

\begin{theorem}
\label{thm:Polynomial}
Let $\cL \in \cSP$. Assume the Riemann Hypothesis for $\zeta(s)$ and the Generalized Riemann Hypothesis for $\cL(s)$. Assume also that $\sigma\in(1/2,1)$, and that $t$ and $\tau$ satisfy the conditions~\eqref{eq:conditions1} and~\eqref{eq:conditions2}. 

\medskip

\noindent (a) If $\sigma-1/2\ll 1/\log{\log{\tau}}$, then
\begin{equation}
\label{eq:Polynomial1}
\log{\left|\mathcal{L}\left(\sigma+\ie t\right)\right|} \leq \frac{\log{\tau}}{2\log\log{\tau}}\log{\left(1+(\log{\tau})^{1-2\sigma}\right)} + O\left(\frac{m\log{\tau}}{(\log\log{\tau})^2}+A_4\right)
\end{equation}
and
\begin{equation}
\label{eq:Polynomial2}
\log{\left|\mathcal{L}\left(\sigma+\ie t\right)\right|}
\geq \frac{\log{\tau}}{2\log\log{\tau}}\log{\left(1-(\log{\tau})^{1-2\sigma}\right)} - O\left(\frac{m\log{\tau}}{\left(\log{\log{\tau}}\right)^2}\log{\frac{1}{1-\left(\log{\tau}\right)^{1-2\sigma}}}+A_5\right)
\end{equation}
for sufficiently large $|t|$ and $\tau$, where $A_4$ and $A_5$ are defined by~\eqref{def:A4} and~\eqref{def:A5}, respectively. 

\medskip

\noindent (b) If $\sigma\in(1/2,1)$, then
\begin{multline}
\label{eq:Polynomial3}
\log{\left|\mathcal{L}\left(\sigma+\ie t\right)\right|} \leq \frac{1}{2}\left(1+m\frac{2\sigma-1}{\sigma(1-\sigma)}\right)\frac{(\log{\tau})^{2-2\sigma}}{\log\log{\tau}}+m\log{(2\log\log{\tau})} \\ +O\left(\frac{m(\log{\tau})^{2-2\sigma}}{(1-\sigma)^2(\log\log{\tau})^2} + \frac{m}{(2\sigma-1)^2}\left(1+\frac{(\log{\tau})^{1-2\sigma}}{(1-\sigma)(2\sigma-1)\log{\log{\tau}}}\right) + A_4\right)
\end{multline}
and
\begin{flalign}
\label{eq:Polynomial4}
\log{\left|\mathcal{L}\left(\sigma+\ie t\right)\right|}
&\geq \frac{\log{\tau}}{2\log{\log{\tau}}}\log{\left(1-(\log{\tau})^{1-2\sigma}\right)} - \frac{m(2\sigma-1)}{2\sigma(1-\sigma)}\left(1+\frac{1}{(\log{\tau})^{2\sigma-1}-1}\right)\frac{(\log{\tau})^{2-2\sigma}}{\log{\log{\tau}}} \nonumber \\
&-m\log{\left(2\log{\log{\tau}}\right)} - O\left(\left(1+\log{\frac{1}{1-\left(\log{\tau}\right)^{1-2\sigma}}}\right)\frac{m\left(\log{\tau}\right)^{2-2\sigma}}{(1-\sigma)^{2}\left(\log{\log{\tau}}\right)^{2}}\right) \nonumber \\
&-O\left(\frac{m}{2\sigma-1}\left(\frac{1}{\left(\log{\tau}\right)^{2\sigma-1}-1}\right)\log{\log{\tau}}+\frac{m}{(2\sigma-1)^{2}} + A_5\right)
\end{flalign}
for sufficiently large $|t|$ and $\tau$, where $A_4$ and $A_5$ are defined by~\eqref{def:A4} and~\eqref{def:A5}, respectively.
\end{theorem}

\begin{remark}
The first term in the definitions of $A_4$ and $A_5$, see~\eqref{def:A4} and~\eqref{def:A5}, can be estimated as $\leq 0.3m$ in the case of $\cL\in\cSP$, see~\eqref{eq:Thm6logL}.
\end{remark}

\subsection{Explicit results}
\label{sec:ExplicitResults}

In this section, we state explicit bounds by assuming a polynomial Euler product representation and the strong $\lambda$-conjecture. Let us first state the upper bounds.

\begin{theorem}
\label{thm:ExplicitUpper}
Let $\cL \in \cSP$. Assume the Riemann Hypothesis for $\zeta(s)$, the Generalized Riemann Hypothesis and the strong $\lambda$-conjecture for $\cL(s)$. Assume also that $\sigma\in(1/2,1)$ and $t\in\R$ such that $\frac{1}{\pi}|t|\log\log{\tau} \geq 10^4$, $\frac{1}{\pi}\log\log{\tau} \geq 0.8$, and
\begin{equation*}
289(a^++b^++1)^2\leq \frac{|t| \pi}{\log\log{\tau}}, \quad \frac{1}{2}+a^{+}\leq\log{\left(1+\sqrt{\frac{|t|}{\pi}\log\log{\tau}}\right)}, \quad
|t|-2\sqrt{\frac{|t|}{\pi}\log\log{\tau}}-b^{+}\geq 17,
\end{equation*}
where $a^+$ and $b^+$ are as in~\eqref{def:a+b+}.

\smallskip

\noindent (a) If $\sigma-1/2\leq \alpha/\log{\log{\tau}}$ for $\alpha>0$ such that $\log{\log{\tau}}\geq 1+2\alpha$, and $\nu\in(0,1)$, then
\begin{multline}
\label{eq:ExplicitUpper1}
\log{\left|\mathcal{L}\left(\sigma+\ie t\right)\right|} \leq \frac{\log{\tau}}{2\log{\log{\tau}}}\log{\left(1+\left(\log{\tau}\right)^{1-2\sigma}\right)} + m\mathcal{A}_{1}(\alpha,\tau)\frac{\log{\tau}}{\left(\log{\log{\tau}}\right)^{2}} \\ +m\left(\mathcal{A}_{2}(\alpha,\nu,\tau)\frac{\log{\tau}}{\left(\log{\log{\tau}}\right)^{4}}+\mathcal{A}_{3}(\alpha,\sigma,\tau)(\log{\log{\tau}})^2\right) + E_{\cL}^{\uparrow}(t),
\end{multline}
where $\mathcal{A}_1(\alpha,\tau)$, $\mathcal{A}_2(\alpha,\nu,\tau)$ and $\mathcal{A}_3(\alpha,\sigma,\tau)$ are defined by~\eqref{eq:calA1},~\eqref{eq:calA2} and~\eqref{eq:calA3}, respectively, and $E_{\cL}^{\uparrow}(t)$ is defined by~\eqref{def:EUpper}.

\smallskip

\noindent (b) If $\sigma\in(1/2,1)$, $\nu_1>0$ and $\nu_2>1$, then
\begin{multline}
\label{eq:ExplicitUpper2}
\log{\left|\mathcal{L}\left(\sigma+\ie t\right)\right|} \leq \frac{1}{2}\left(1+m\frac{2\sigma-1}{\sigma(1-\sigma)}\right)\frac{(\log{\tau})^{2-2\sigma}}{\log{\log{\tau}}} + \frac{m\left(\left(\nu_2\sigma\right)^{2}+(1-\sigma)^2\right)}{\left(2\sigma(1-\sigma)\right)^{2}}\frac{(\log{\tau})^{2-2\sigma}}{(\log{\log{\tau}})^{2}} \\ 
+ \frac{m}{\nu_{1}^{2}}(\log{\tau})^{\frac{2(1-\sigma)}{\nu_2}} + m\log{\left(2\log{\log{\tau}}\right)} + m\left(\mathcal{A}_{4}(\nu_1,\sigma) + \frac{\mathcal{A}_{5}(\sigma,\tau)+\mathcal{A}_{6}(\sigma,\tau)}{(\log{\tau})^{2\sigma-1}}\right) + E_{\cL}^{\uparrow}(t),
\end{multline}
where $\mathcal{A}_{4}(\nu_1,\sigma)$, $\mathcal{A}_{5}(\sigma,\tau)$ and $\mathcal{A}_{6}(\sigma,\tau)$ are defined by~\eqref{eq:calA4},~\eqref{eq:calA5} and~\eqref{eq:calA6}, respectively, and $E_{\cL}^{\uparrow}(t)$ is defined by~\eqref{def:EUpper}.
\end{theorem}

\begin{corollary}
\label{cor:CLforSP}
Assume the conditions and assumptions on $\cL(s)$, $t$ and $\tau$ as in Theorem~\ref{thm:ExplicitUpper}. Then
\begin{multline}
\label{eq:CriticalLine}
\left|\mathcal{L}\left(\frac{1}{2}+\ie t\right)\right| \leq \exp\Bigg(\frac{\log{2}}{2}\cdot\frac{\log{\tau}}{\log\log{\tau}}+\frac{2m\log{\tau}}{\left(\log\log{\tau}\right)^2} + \left(1+\frac{1}{\nu^4}\right)\frac{6m\log{\tau}}{\left(\log\log{\tau}\right)^4} \\
+ 6m\left(\frac{2}{\log{2}}\right)^{4}\left(\log{\tau}\right)^{\nu} + \frac{m}{8\pi}\left(\log{\log{\tau}}\right)^{2} + \frac{m}{4\pi}\log{\log{\tau}} + 2.31m +E_{\cL}^{\uparrow}(t)\Bigg),
\end{multline}
where $\nu\in(0,1)$ and $E_{\cL}^{\uparrow}(t)$ is defined by~\eqref{def:EUpper}. If $E_{\cL}^{\uparrow}(t)$ is replaced with $A_4$ as defined by~\eqref{def:A4}, then~\eqref{eq:CriticalLine} is also valid under the conditions of Theorem~\ref{thm:Polynomial}.
\end{corollary}

Lastly, we also derive the lower bounds.

\begin{theorem}
\label{thm:ExplicitLower}
Let $\cL \in \cSP$. Assume the Riemann Hypothesis for $\zeta(s)$, the Generalized Riemann Hypothesis and the strong $\lambda$-conjecture for $\cL(s)$, and all conditions on $\sigma$, $t$ and $\tau$ from Theorem~\ref{thm:ExplicitUpper}.

\medskip

\noindent (a) If $\sigma-1/2\leq \alpha/\log{\log{\tau}}$ for $\alpha>0$ such that $\log{\log{\tau}}\geq 1+2\alpha$, and $\nu\in(0,1)$, then
\begin{multline}
\label{eq:ExplicitLower1}
\log{\left|\mathcal{L}\left(\sigma+\ie t\right)\right|}
\geq \frac{\log{\tau}}{2\log\log{\tau}}\log{\left(1-(\log{\tau})^{1-2\sigma}\right)} - m\left(2\log{\frac{1}{1-\left(\log{\tau}\right)^{1-2\sigma}}}+\mathcal{B}_1(\alpha,\tau)\right)\frac{\log{\tau}}{\left(\log{\log{\tau}}\right)^2} \\
-m\left(\mathcal{A}_{2}(\alpha,\nu,\tau)\frac{\log{\tau}}{\left(\log{\log{\tau}}\right)^4} + \left(0.2\log{\frac{1}{1-\left(\log{\tau}\right)^{1-2\sigma}}}+\mathcal{B}_{2}(\tau)\right)\left(\log{\log{\tau}}\right)^{2}\right) - E_{\cL}^{\downarrow}(\sigma,t),
\end{multline}
where $\mathcal{B}_{1}(\alpha,\tau)$ and $\mathcal{B}_{2}(\tau)$ are defined by~\eqref{eq:calB1B2}, $\mathcal{A}_{2}(\alpha,\nu,\tau)$ is defined by~\eqref{eq:calA2}, and $E_{\cL}^{\downarrow}(\sigma,t)$ is defined by~\eqref{def:ELower}.

\medskip

\noindent (b) If $\sigma\in(1/2,1)$, $\nu_1>0$ and $\nu_2>1$, then
\begin{flalign}
\label{eq:ExplicitLower2}
\log{\left|\mathcal{L}\left(\sigma+\ie t\right)\right|}
&\geq \frac{\log{\tau}}{2\log{\log{\tau}}}\log{\left(1-(\log{\tau})^{1-2\sigma}\right)} - \frac{m(2\sigma-1)}{2\sigma(1-\sigma)}\left(1+\frac{1}{(\log{\tau})^{2\sigma-1}-1}\right)\frac{(\log{\tau})^{2-2\sigma}}{\log{\log{\tau}}} \nonumber \\
&-\frac{m}{4(1-\sigma)^{2}}\left(\frac{(\sigma\nu_2)^{2}+(1-\sigma)^{2}}{\sigma^2}+2\log{\frac{1}{1-(\log{\tau})^{1-2\sigma}}}\right)\frac{(\log{\tau})^{2-2\sigma}}{(\log{\log{\tau}})^{2}} - \frac{m}{\nu_{1}^{2}}(\log{\tau})^{\frac{2(1-\sigma)}{\nu_2}} \nonumber \\
&-\frac{2m}{2\sigma-1}\left(\frac{\mathcal{B}_{3}(\sigma,\tau)}{(\log{\tau})^{2\sigma-1}-1}+\mathcal{B}_{4}(\sigma,\tau)\log{\frac{1}{1-(\log{\tau})^{1-2\sigma}}}\right)\log{\log{\tau}} \nonumber \\ 
&- m\log{\left(2\log{\log{\tau}}\right)}
-m\mathcal{A}_{4}(\nu_1,\sigma) - \frac{m\mathcal{A}_{5}(\sigma,\tau)}{(\log{\tau})^{2\sigma-1}} - E_{\cL}^{\downarrow}(\sigma,t),
\end{flalign}
where $\mathcal{B}_{3}(\sigma,\tau)$ and $\mathcal{B}_{4}(\sigma,\tau)$ are defined by~\eqref{eq:calB3B4}, while $\mathcal{A}_{4}(\nu_1,\sigma)$, $\mathcal{A}_{5}(\nu_1,\sigma)$ and $E_{\cL}^{\downarrow}(\sigma,t)$ are defined by~\eqref{eq:calA4},~\eqref{eq:calA5} and~\eqref{def:ELower}, respectively.
\end{theorem}

\begin{remark}
Let $\sigma-1/2\gg 1/\log{\log{\tau}}$. Note that then
\begin{equation*}
\frac{(\log{\tau})^{2\sigma-1}}{(\log{\tau})^{2\sigma-1}-1}=1+O\left(\frac{1}{(\log{\tau})^{2\sigma-1}}\right)
\end{equation*}
as $\tau\to\infty$, and hence the first two terms on the right-hand side of~\eqref{eq:ExplicitLower2} are of similar size (but negative) as the first term on the right-hand side of~\eqref{eq:ExplicitUpper2}. They are the same (but negative) if $\sigma\in(1/2,1)$ is fixed. The same remark also holds for part (b) of Theorem~\ref{thm:Polynomial}.
\end{remark}

\section{Connection between $\log{|\cL(\sigma+\ie t)|}$ and $f_{\sigma}(t-\gamma)$}
\label{sec:Sum}

In this section we derive conditional (GRH) relation between $\log{|\cL(\sigma+\ie t)|}$ and $f_{\sigma}(t-\gamma)$, see Proposition~\ref{prop:LogSum}, where
$f_{\sigma}(x)$ is as in~\eqref{def:falpha}. We will use definitions for parameters $f$, $d_\cL$, $m_\cL$, $\lambda^-$, $a^+$ and $b^+$ from Sections~\ref{subsec:Basics} and~\ref{subsec:Def}. 

Our approach is essentially the same as that in~\cite{CarneiroChandee}, where~\eqref{eq:MainLogSum} is proved for $\cL(s)=\zeta(s)$ and with $L\ll 1/t$, see~\cite[Equation~2.20]{CarneiroChandee}. Our estimate for $L$ is asymptotically the same, but fully explicit and uniform in $\cL$. 

In the proof of Proposition~\ref{prop:LogSum} we are going to use the following two lemmas; the first one is on bounding the size of a quotient of two Gamma functions, while the second one is on Hadamard's factorisation theorem for the Selberg class.

\begin{lemma}
\label{lemma:Gamma}
Assume that real numbers $x_1$, $x_2$ and $y$ satisfy $x_1>0$, $x_2>0$ and $y\neq 0$. Then
\begin{multline}
\label{eq:gammaQ}
\log{\left|\frac{\Gamma(x_1+\ie y)}{\Gamma(x_2+\ie y)}\right|} = \left(x_1-\frac{1}{2}\right)\log{|x_1+\ie y|} - \left(x_2-\frac{1}{2}\right)\log{|x_2+\ie y|} \\ 
+ \frac{\left(x_1-x_2\right)\left(y^2-x_1x_2\right)}{12|x_1+\ie y|^2|x_2+\ie y|^2} + \frac{x_1x_2\left(x_2-x_1\right)}{y^2+x_1x_2} + R(x_1,x_2,y),
\end{multline}
where
\begin{equation}
\label{eq:defR1}
\left|R(x_1,x_2,y)\right| \leq \frac{|x_2-x_1|^3}{|y|^2} + \frac{1}{45|y|^3}.
\end{equation}
\end{lemma}

\begin{proof}
Let $\Re\{z\}>0$. By Stirling's formula we have
\[
\log{\Gamma(z)} = \left(z-\frac{1}{2}\right)\log{z} - z + \frac{1}{2}\log{(2\pi)} + \frac{1}{12z} + R_1(z), \quad \left|R_1(z)\right| \leq \frac{1}{90|z|^3},
\]
see~\cite[p.~294]{Olver}. Write $z_1=x_1+\ie y$ and $z_2=x_2+\ie y$. Then, if $y \neq 0$
\begin{multline}
\label{eq:gammaQ1}
\log{\left|\frac{\Gamma(z_1)}{\Gamma(z_2)}\right|} = \left(x_1-\frac{1}{2}\right)\log{|z_1|} - \left(x_2-\frac{1}{2}\right)\log{|z_2|} + \frac{1}{12}\left(\frac{x_1}{|z_1|^2} - \frac{x_2}{|z_2|^2}\right)\\ 
- y\left(\arg{(z_1)}-\arg{(z_2)}\right) + \left(x_2-x_1\right) + \Re\left\{R_1(z_1)-R_1(z_2)\right\}.
\end{multline}
Note that $\arctan{u}=u+R_2(u)$, where $u\in\R$ and $\left|R_2(u)\right|\leq |u|^3$. Therefore, 
\begin{equation}
\label{eq:arg}
\arg{(z_1)}-\arg{(z_2)} = \arctan\left(\frac{x_2-x_1}{y\left(1+\frac{x_1x_2}{y^2}\right)}\right) = \frac{x_2-x_1}{y\left(1+\frac{x_1x_2}{y^2}\right)} + R_2\left(\frac{x_2-x_1}{y\left(1+\frac{x_1x_2}{y^2}\right)}\right).
\end{equation}
We obtain~\eqref{eq:gammaQ} after combining~\eqref{eq:gammaQ1} and~\eqref{eq:arg}, and taking $|z_1|\geq |y|$, $|z_2|\geq |y|$.
\end{proof}


\begin{lemma}
\label{lemma:Hadamard}
Let $\cL\in\cS$ with $\cL(1)\neq 0$. Then
\begin{equation}
\label{eq:Hadamard}
s^{m_{\cL}}(s-1)^{m_{\cL}}\mathfrak{L}(s) = e^{A_\mathcal{L}+B_\mathcal{L}s} \prod_\rho \left(1-\frac{s}{\rho}\right)e^{s/\rho},
\end{equation}
where $\rho$ are the non-trivial zeros of $\cL(s)$ and $A_{\cL}$ and $B_{\cL}$ are some constants which depend on $\cL$. Moreover, $\Re\left\{B_{\cL}\right\} = -\sum_{\rho}\Re\left\{1/\rho\right\}$.
\end{lemma}

\begin{proof}
We know that $(s-1)^{m_{\cL}}\cL(s)$ is entire function of order $1$, see~\cite[Lemma 4.2]{Chaubey23}. Conditions of Lemma~\ref{lemma:Hadamard} imply that $s^{m_{\cL}}(s-1)^{m_{\cL}}\mathfrak{L}(s)$ is also entire function of order $1$, with only zeros at the non-trivial zeros $\rho\notin\{0,1\}$ of $\cL(s)$. Then Hadamard's factorization theorem guarantees~\eqref{eq:Hadamard}, with some constants $A_{\cL}$ and $B_{\cL}$ which depend on $\cL$. 

We closely follow~\cite[p.~103]{IKANT}. By~\eqref{eq:Hadamard} and the functional equation we have
\begin{flalign*}
e^{A_\mathcal{L}+B_\mathcal{L}s} \prod_\rho \left(1-\frac{s}{\rho}\right)e^{s/\rho} &= s^{m_{\cL}}(s-1)^{m_{\cL}}\mathfrak{L}(s) \\ 
&= \omega\overline{\left(\bar{s}^{m_{\cL}}\left(\bar{s}-1\right)^{m_{\cL}}\mathfrak{L}(1-\bar{s})\right)}
= \omega e^{\overline{A_\mathcal{L}}+\overline{B_\mathcal{L}}(1-s)} \prod_\rho \left(1-\frac{1-s}{\bar{\rho}}\right)e^{(1-s)/\bar{\rho}}.
\end{flalign*}
After taking logarithmic derivative on both sides of the above equality, we obtain
\[
2\Re\left\{B_{\cL}\right\} = -\sum_{\rho}\left(\frac{1}{s-\rho}+\frac{1}{1-s-\bar{\rho}}+\frac{1}{\rho}+\frac{1}{\bar{\rho}}\right)
= -\sum_{\rho}\left(\frac{1}{s-\rho}+\frac{1}{1-s-\bar{\rho}}\right) - 2\sum_{\rho}\Re\left\{\frac{1}{\rho}\right\}
\]
since the second series is absolutely convergent. Taking $s=0$ in the above equality, and observing that $\rho$ is the non-trivial zero if and only if $1-\bar{\rho}$ is the non-trivial zero, gives $\Re\left\{B_{\cL}\right\} = -\sum_{\rho}\Re\left\{1/\rho\right\}$.
\end{proof}

\begin{proposition}
\label{prop:LogSum}
Assume the Generalized Riemann Hypothesis. Let $\sigma\in[1/2,1]$ and assume that if $\sigma =1/2$, then $t$ is distinct from an ordinate of
the non-trivial zero. In addition, suppose 
\begin{equation}
\label{eq:condition}
|t|\geq b^{+} + \sqrt{\left(1+a^{+}\right)\left(\frac{5}{2}+a^{+}\right)}.
\end{equation}
Then
\begin{equation}
\label{eq:MainLogSum}
\log{\left|\mathcal{L}\left(\sigma+\ie t\right)\right|} = \left(\frac{5}{4}-\frac{\sigma}{2}\right)\log{\tau} - \frac{1}{2}\sum_\gamma f_{\sigma}(t-\gamma) + \log{\left|\mathcal{L}\left(\frac{5}{2}+\ie t\right)\right|} + L,
\end{equation}
where
\begin{gather*}
L \leq L_1\left(m_{\cL};t\right) + L_2^{\uparrow}\left(\sdeg,a^{+},b^{+};t\right) + L_3^{\uparrow}\left(a^{+},b^{+},\left\{\mu_{j}\right\};t\right) + L_4^{\uparrow}\left(\sdeg,b^{+},f,\lambda^{-};t\right), \\
L \geq L_2^{\downarrow}\left(\sdeg,a^{+},b^{+};t\right) + L_3^{\downarrow}\left(a^{+},b^{+},\left\{\mu_{j}\right\};t\right) + L_4^{\downarrow}\left(\sdeg,b^{+},f,\lambda^{-};t\right)
\end{gather*}
and $f_\sigma$ is as in \eqref{def:falpha}.
Here, $L_1$ is defined by~\eqref{eq:L1}, $L_2^{\uparrow}$ and $L_2^{\downarrow}$ are defined by~\eqref{eq:G1up} and~\eqref{eq:G1down}, $L_3^{\uparrow}$ and $L_3^{\downarrow}$ are defined by~\eqref{eq:G2up} and~\eqref{eq:G2down}, and $L_4^{\uparrow}$ and $L_4^{\downarrow}$ are defined by~\eqref{eq:G3up} and~\eqref{eq:G3down}, respectively. Functions $L_1$, $L_2^{\uparrow}$, $L_3^{\uparrow}$ and $L_4^{\uparrow}$ are decreasing in $|t|$, and functions $L_2^{\downarrow}$, $L_3^{\downarrow}$ and $L_4^{\downarrow}$ are increasing in $|t|$.  
\end{proposition}

\begin{proof}
Using Lemma~\ref{lemma:Hadamard} and the functional equation, we can see that under GRH one can write
\begin{equation*}
\left|\frac{\mathfrak{L}(\sigma+\ie t)}{\mathfrak{L}(5/2+\ie t)}\right|=\left|\frac{\mathfrak{L}(\sigma+\ie t)}{\mathfrak{L}(-3/2+\ie t)}\right|=\left|\frac{\left(5/2+\ie t\right)\left(3/2+\ie t\right)}{\left(\sigma+\ie t\right)\left(1-\sigma+\ie t\right)}\right|^{m_{\cL}}\prod_\gamma \left(\frac{(\sigma-1/2)^2+(t-\gamma)^2}{4+(t-\gamma)^2}\right)^{1/2}.
\end{equation*}
Taking logarithms on both sides of the above equality and using the functional equation again, we obtain
\begin{multline}
\label{eq:logFormula2}
\log{\left|\mathcal{L}\left(\sigma+\ie t\right)\right|}=
\frac{m_{\cL}}{2}\log{\left|\frac{\left(5/2+\ie t\right)\left(3/2+\ie t\right)}{\left(\sigma+\ie t\right)\left(1-\sigma+\ie t\right)}\right|^{2}} 
-\frac{1}{2}\sum_\gamma f_{\sigma}(t-\gamma) \\
+\log{\left|\mathcal{L}\left(\frac{5}{2}+\ie t\right)\right|}+\left(\frac{5}{2}-\sigma\right)\log{Q} + \sum_{j=1}^{f}\log{\left|\frac{\Gamma\left(\lambda_{j}(5/2+\ie t)+\mu_j\right)}{\Gamma\left(\lambda_{j}(\sigma+\ie t)+\mu_j\right)}\right|}.
\end{multline}
We are going to estimate the first and the last term on the right-hand side of~\eqref{eq:logFormula2}.

Concerning the first term, straightforward estimation gives us
\begin{equation}
\label{eq:L1}
0 < \frac{m_{\cL}}{2}\log{\left|\frac{\left(5/2+\ie t\right)\left(3/2+\ie t\right)}{\left(\sigma+\ie t\right)\left(1-\sigma+\ie t\right)}\right|^{2}} 
\leq \frac{m_{\cL}}{2}\log{\left(1+\frac{17}{2t^2}+\left(\frac{15}{4t^2}\right)^{2}\right)} \ed L_1\left(m_{\cL};t\right).
\end{equation}

Estimation of the last term is more involved. Let $x_{1,j}=(5/2)\lambda_{j}+\Re\left\{\mu_{j}\right\}$, $x_{2,j}=\sigma\lambda_{j}+\Re\left\{\mu_{j}\right\}$ and $y_j=t\lambda_{j}+\Im\left\{\mu_j\right\}$. By Lemma~\ref{lemma:Gamma} we can write
\begin{equation}
\label{eq:sumLog}
\sum_{j=1}^{f}\log{\left|\frac{\Gamma\left(\lambda_{j}(5/2+\ie t)+\mu_j\right)}{\Gamma\left(\lambda_{j}(\sigma+\ie t)+\mu_j\right)}\right|} = \sum_{j=1}^{f}G_{1,j} + \sum_{j=1}^{f}G_{2,j} + \sum_{j=1}^{f}G_{3,j},
\end{equation}
where
\begin{gather*}
G_{1,j} \de \frac{5}{2}\lambda_{j}\log{\left|\lambda_{j}\left(\frac{5}{2}+\ie t\right)+\mu_j\right|} - \sigma\lambda_{j}\log{\left|\lambda_{j}(\sigma+\ie t)+\mu_j\right|}, \\
G_{2,j} \de \frac{1}{2}\left(\Re\left\{\mu_{j}\right\}-\frac{1}{2}\right)\log{\left|\frac{\lambda_{j}(5/2+\ie t)+\mu_j}{\lambda_{j}(\sigma+\ie t)+\mu_j}\right|^{2}}, \\
G_{3,j} \de 
\frac{1}{12}\left(\frac{5}{2}-\sigma\right)\lambda_{j}\frac{y_{j}^2-x_{1,j}x_{2,j}}{|x_{1,j}+\ie y_{j}|^2|x_{2,j}+\ie y_{j}|^2} - \left(\frac{5}{2}-\sigma\right)\lambda_j\frac{x_{1,j}x_{2,j}}{y_j^2+x_{1,j}x_{2,j}} + R\left(x_{1,j},x_{2,j},y_j\right),
\end{gather*}
and $R\left(x_{1,j},x_{2,j},y_j\right)$ is from~\eqref{eq:defR1}. We need to obtain upper and lower bounds for each of the three sums from the right-hand side of~\eqref{eq:sumLog}. 

We will start with $\sum_{j=1}^{f}G_{1,j}$. For $\sigma>0$ we have 
\[
\log{\left|\lambda_{j}\left(\sigma+\ie t\right)+\mu_{j}\right|} = \log{\left(2\lambda_{j}\pi|t|\right)} - \log{(2\pi)} + A(\sigma, j), \quad A(\sigma, j) \de \log{\left|\frac{\sigma}{|t|}+\ie\frac{t}{|t|}+\frac{\mu_{j}}{\lambda_{j}|t|}\right|}.
\]
Then
\[
\sum_{j=1}^{f}G_{1,j} = \left(\frac{5}{4}-\frac{\sigma}{2}\right)\left(\log{\tau}-2\log{Q}\right) + \sum_{j=1}^f 2\lambda_j\left(\frac{5}{4}A\left(\frac{5}{2}, j\right)-\frac{\sigma}{2}A(\sigma, j)\right).
\]
Note that
\[
\log{\left(1-\frac{b^{+}}{|t|}\right)} \leq A(\sigma, j) \leq \log{\left(1+\frac{1}{|t|}\left(\sigma+a^{+}+b^{+}\right)\right)}
\]
for all $j$. Therefore,
\begin{equation}
\label{eq:G1Up}
\sum_{j=1}^{f}G_{1,j} \leq \left(\frac{5}{4}-\frac{\sigma}{2}\right)\left(\log{\tau}-2\log{Q}\right) + L_2^{\uparrow}\left(\sdeg,a^{+},b^{+};t\right)
\end{equation}
and
\begin{equation}
\label{eq:G1Down}
\sum_{j=1}^{f}G_{1,j} \geq \left(\frac{5}{4}-\frac{\sigma}{2}\right)\left(\log{\tau}-2\log{Q}\right) + L_2^{\downarrow}\left(\sdeg,a^{+},b^{+};t\right),
\end{equation}
where
\begin{gather}
L_2^{\uparrow}\left(\sdeg,a^{+},b^{+};t\right) \de \sdeg\left(\frac{5}{4}\log{\left(1+\frac{5/2+a^{+}+b^{+}}{|t|}\right)}-\frac{1}{2}\log{\left(1-\frac{b^{+}}{|t|}\right)}\right), \label{eq:G1up} \\
L_2^{\downarrow}\left(\sdeg,a^{+},b^{+};t\right) \de \sdeg\left(\frac{5}{4}\log{\left(1-\frac{b^{+}}{|t|}\right)}-\frac{1}{2}\log{\left(1+\frac{1+a^{+}+b^{+}}{|t|}\right)}\right). \label{eq:G1down}
\end{gather}
Turning to the estimation of $\sum_{j=1}^{f}G_{2,j}$, observe that $|t|-b^{+}>0$ and
\[
1\leq \left|\frac{\lambda_{j}(5/2+\ie t)+\mu_j}{\lambda_{j}(\sigma+\ie t)+\mu_j}\right|^{2} \leq 1 + \left(\frac{5/2+a^{+}}{|t|-b^{+}}\right)^{2}.
\]
Hence, we have
\begin{equation}
\label{eq:G2up}
\sum_{j=1}^{f}G_{2,j} \leq \frac{1}{2}\log{\left(1 + \left(\frac{5/2+a^{+}}{|t|-b^{+}}\right)^{2}\right)}\sum_{j=1}^{f}\max\left\{0,\Re\left\{\mu_{j}\right\}-\frac{1}{2}\right\} \ed L_3^{\uparrow}\left(a^{+},b^{+},\left\{\mu_{j}\right\};t\right)
\end{equation}
and
\begin{equation}
\label{eq:G2down}
\sum_{j=1}^{f}G_{2,j} \geq \frac{1}{2}\log{\left(1 + \left(\frac{5/2+a^{+}}{|t|-b^{+}}\right)^{2}\right)}\sum_{j=1}^{f}\min\left\{0,\Re\left\{\mu_{j}\right\}-\frac{1}{2}\right\} \ed L_3^{\downarrow}\left(a^{+},b^{+},\left\{\mu_{j}\right\};t\right).
\end{equation}
In order to derive bounds for $\sum_{j=1}^{f}G_{3,j}$, note that~\eqref{eq:condition} guarantees that $y_{j}^2-x_{1,j}x_{2,j}\geq 0$ since 
\[
x_{1,j}x_{2,j}\leq \lambda_{j}^{2}\left(1+a^{+}\right)\left(\frac{5}{2}+a^{+}\right), \quad 
|y_j|\geq \lambda_{j}\left(|t|-b^{+}\right)>0.
\]
Therefore,
\begin{equation}
\label{eq:G3up}
\sum_{j=1}^{f}G_{3,j} \leq \frac{1}{\left(|t|-b^{+}\right)^2}\left(4\sdeg + \frac{f}{6\lambda^{-}}+\frac{f}{45\left(\lambda^{-}\right)^{3}\left(|t|-b^{+}\right)}\right) \ed L_4^{\uparrow}\left(\sdeg,b^{+},f,\lambda^{-};t\right)
\end{equation}
and
\begin{equation}
\label{eq:G3down}
\sum_{j=1}^{f}G_{3,j} \geq -\frac{1}{\left(|t|-b^{+}\right)^2}\left(\sdeg\left(\left(a^{+}\right)^{2}+\frac{7}{2}a^{+}+\frac{13}{2}\right)+\frac{f}{45\left(\lambda^{-}\right)^{3}\left(|t|-b^{+}\right)}\right) \ed L_4^{\downarrow}\left(\sdeg,a^{+},b^{+},f,\lambda^{-};t\right).
\end{equation}

The proof is complete after employing estimates~\eqref{eq:G1Up},~\eqref{eq:G1Down},~\eqref{eq:G2up},~\eqref{eq:G2down},~\eqref{eq:G3up} and~\eqref{eq:G3down} into \eqref{eq:sumLog}, and then using the resulting bound, together with~\eqref{eq:L1}, in~\eqref{eq:logFormula2}.
\end{proof}


\section{Estimating $\sum_{\gamma}f_{\sigma}(t-\gamma)$}
\label{sec:Estimatef}

In this section we are going to estimate $\sum_{\gamma} f_{\sigma}(t-\gamma)$ that appears in~\eqref{eq:MainLogSum}, where $f_{\sigma}(x)$ is given by~\eqref{def:falpha}. By~\cite[Corollary 18]{CLV2013}, we know that there exist a unique extremal minorant $g_\Delta(x)$ and a unique extremal majorant $m_\Delta(x)$ for $f_{\sigma}(x)$, see~\eqref{eq:DefGDelta} and~\eqref{def:Mandm} for definitions, that is 
\[
g_\Delta(x) \leq f_{\sigma}(x) \leq m_\Delta(x)
\]
for all $x\in\R$. The idea is to apply the Guinand--Weil formula (Lemma~\ref{lem:GW}) to these extremal functions, and then estimate the terms that are coming from the sum.

\begin{lemma}
\label{lem:GW}
Assume the Generalized Riemann Hypothesis. Let $\cL\in\cS$ and let $h(z)$ be a holomorphic function on a strip
\[
\left\{z \in \mathbb{C}\colon -\frac{1}{2}-\varepsilon< \Im\{z\} <\frac{1}{2}+\varepsilon\right\}
\]
for some $\varepsilon >0$, such that $h(z)\left(1+|z|\right)^{1+\delta}$ is bounded for some $\delta>0$ and $h(z)$ is real for $z\in\R$. Then
\begin{multline}
\label{eq:GuinandWeil}
\sum_{\gamma} h(t-\gamma) = 2m_{\cL}\Re\left\{h\left(t+\frac{1}{2\ie}\right)\right\} + \frac{\log{Q}}{\pi}\widehat{h}(0) - \frac{1}{\pi}\sum_{n=2}^{\infty}\frac{1}{\sqrt{n}}\Re\left\{\overline{\Lambda_{\cL}(n)}\cdot\widehat{h}\left(\frac{\log{n}}{2\pi}\right)e^{\ie t\log{n}}\right\} \\
+ \frac{1}{\pi}\sum_{j=1}^{f}\int_{-\infty}^{\infty} h(u)\lambda_{j}\Re\left\{\frac{\Gamma'}{\Gamma}\left(\frac{\lambda_j}{2}+\mu_j+\lambda_j(t-u)\ie\right)\right\}\dif{u},
\end{multline}
where $\widehat{h}(\xi)=\int_{-\infty}^{\infty}h(u)e^{-2\pi\ie u\xi}\dif{u}$ denotes the Fourier transform of $h(z)$ and the sum runs over the imaginary parts of the non-trivial zeros of $\cL(s)$.
\end{lemma}

\begin{proof}
This follows from~\cite[Remark 1]{PalojarviSimonic} after employing the result for $z\mapsto h(t-z)$. 
\end{proof}

Let us first estimate the logarithmic derivative of the Gamma function.
\begin{lemma}
\label{lem:LogGamma}
Let $\Re\{z\}>0$. Then
\[
\Re\left\{\frac{\Gamma'}{\Gamma}(z)\right\} = \log{|z|} + A(z),
\]
where
\begin{equation}
\label{eq:ErrorInLogGamma}
\left|A(z)\right| \leq \frac{1}{2}\left(\Re\{z\}+\frac{1}{6}+\frac{\sqrt{2}}{15|z|^2}\right)\frac{1}{|z|^2}.
\end{equation}
\end{lemma}

\begin{proof}
By~\cite[Exercise 4.2 on p.~295]{Olver} we have 
\[
\left|\Re\left\{\frac{\Gamma'}{\Gamma}(z)\right\}-\log{|z|}\right| \leq \left|\Re\left\{\frac{1}{2z}\right\}\right| + \frac{1}{12|z|^2} + \frac{1}{15\sqrt{2}|z|^4}
\]
for $\Re\{z\}>0$. Note that $\Re\left\{1/(2z)\right\}=\Re\{z\}/\left(2|z|^2\right)$. Now the result easily follows.
\end{proof}

Using the previous estimate, we prove upper bounds for the last sum in the Guinand--Weil explicit formula \eqref{eq:GuinandWeil}, and then for the whole formula. The results are conditional and depend on upper bounds of the term $|h|$. These are derived in Section \ref{subsec:Explicitgm}.

\begin{proposition}
\label{prop:FourierGamma}
Let $h(z)$ be as in Lemma~\ref{lem:GW}, and let $\Delta>0$ and $u$ be real numbers. Assume that there exist non-negative real numbers $\mathfrak{m}_{1}$, $\mathfrak{m}_{1}'$, $\mathfrak{m}_2$, $\mathfrak{m}_2'$ and $\mathfrak{M}$ that do not depend on $u$, such that 
\begin{equation}
\label{eq:LinearOnh}
\left|h(u)\right| \leq \left(\mathfrak{m}_1+\mathfrak{m}_1'\frac{1+\Delta|u|}{1+\Delta^2 u^2}\right)\frac{\Delta^2}{1+\Delta|u|}
\end{equation}
for every $u\in\R$, and that 
\begin{equation}
\label{eq:QuadOnh}
\left|h(u)\right| \leq \left(\mathfrak{m}_2+\frac{\mathfrak{m}_2'}{\Delta^2u^2}\right)\frac{1}{u^2} \quad \textrm{for} \quad |u|\geq\mathfrak{M}.
\end{equation} 
Additionally, assume that $0<\mathfrak{a}\leq \frac{1}{4}\sqrt{\Delta|t|}$, $\mathfrak{b}>0$ and that
\begin{equation}
\label{assump:lowert2}
|t|\geq 4\left(a^{+}+b^{+}\right)+4, \quad \mathfrak{a}\sqrt{\frac{|t|}{\Delta}}\geq\mathfrak{M}, \quad 
|t|-\frac{\mathfrak{b}}{\lambda^{-}}\sqrt{\Delta|t|}-b^{+}\geq\mathfrak{M},
\end{equation}
where $a^+$ and $b^+$ are as in \eqref{def:a+b+}.
Then
\begin{equation}
\label{eq:IntLogGamma}
\sum_{j=1}^{f}\int_{-\infty}^{\infty} h(u)\lambda_{j}\Re\left\{\frac{\Gamma'}{\Gamma}\left(\frac{\lambda_j}{2}+\mu_j+\lambda_j(t-u)\ie\right)\right\}\dif{u} = \widehat{h}(0)\left(\frac{1}{2}\log{\tau}-\log{Q}\right)
+ \mathcal{I},
\end{equation}
where $\left|\mathcal{I}\right|\leq \widehat{\mathcal{I}}_1 + \widehat{\mathcal{I}}_2$ with
\begin{multline*}
\widehat{\mathcal{I}}_1 \de \frac{\sdeg}{\mathfrak{a}}\log{\left(2|t|\right)}\sqrt{\frac{\Delta}{|t|}}\left(\mathfrak{m}_2+\frac{\mathfrak{m}_2'}{3\mathfrak{a}^2 \Delta|t|}\right) +\sdeg\sqrt{\frac{\Delta}{|t|}}\left(\mathfrak{m}_1\log{\left(1+\mathfrak{a}\sqrt{\Delta|t|}\right)}+\mathfrak{m}_1'\arctan\left(\mathfrak{a}\sqrt{\Delta|t|}\right)\right)\times \\
\times\left(\mathfrak{a}+\left(a^{+}+b^{+}+\frac{1}{2}\right)\sqrt{\frac{\Delta}{|t|}}\right)\left(1+\frac{\mathfrak{a}}{\sqrt{\Delta|t|}}+\frac{a^{+}+b^{+}+\frac{1}{2}}{|t|}\right)
\end{multline*}
and
\begin{flalign*}
\widehat{\mathcal{I}}_2 &\de 
\frac{\sdeg}{2\mathfrak{b}^2 |t|}\left(\lambda^{+}\left(\frac{1}{2}+a^{+}\right)+\frac{1}{6}+\frac{\sqrt{2}}{15\mathfrak{b}^2\Delta|t|}\right)\left(\mathfrak{m}_1\log{\left(1+\Delta|t|\right)}+\mathfrak{m}_1'\Delta\arctan{\left(\Delta |t|\right)}+\frac{\mathfrak{m}_2}{\Delta|t|}+\frac{\mathfrak{m}_2'}{3\Delta^3 |t|^3}\right) \\
&+\sdeg\left(\lambda^{+}\left(\frac{1}{2}+a^{+}\right)+\frac{1}{6}+\frac{\sqrt{2}}{15}\left(\frac{2}{\lambda^{-}}\right)^{2}\right)\left(\frac{2}{\lambda^{-}}\right)^{2}\frac{\frac{\mathfrak{b}}{\lambda^{-}}\sqrt{\Delta|t|}+b^{+}}{t^2-\left(\frac{\mathfrak{b}}{\lambda^{-}}\sqrt{\Delta|t|}+b^{+}\right)^2}\times \\
&\times\left(\mathfrak{m}_2+\frac{\mathfrak{m}_2'\left(3t^2+\left(\frac{\mathfrak{b}}{\lambda^{-}}\sqrt{\Delta|t|}+b^{+}\right)^2\right)}{3\Delta^2 \left(t^2-\left(\frac{\mathfrak{b}}{\lambda^{-}}\sqrt{\Delta|t|}+b^{+}\right)^2\right)^2}\right).
\end{flalign*}
\end{proposition}

\begin{proof}
Using Lemma~\ref{lem:LogGamma}, it is not hard to see that~\eqref{eq:IntLogGamma} holds with
\[
\mathcal{I} = \int_{-\infty}^{\infty}h(u)\sum_{j=1}^{f}\lambda_{j}\log{\left|\frac{1}{2t}+\frac{\mu_j}{\lambda_{j}t}+\left(1-\frac{u}{t}\right)\ie\right|}\dif{u} + \int_{-\infty}^{\infty}h(u)\sum_{j=1}^{f}\lambda_{j}A\left(\frac{\lambda_j}{2}+\mu_j+\lambda_j(t-u)\ie\right)\dif{u}.
\]
Denote the first integral and the second integral by $\mathcal{I}_1$ and $\mathcal{I}_2$, respectively. We are going to show that $\left|\mathcal{I}_1\right|\leq\widehat{\mathcal{I}}_1$ and $\left|\mathcal{I}_2\right|\leq\widehat{\mathcal{I}}_2$. 

We will start with $\mathcal{I}_1$. We divide the integral to different parts and observe that
\[
\left|\mathcal{I}_1\right| \leq \left(\int_{-\mathfrak{a}\sqrt{|t|/\Delta}}^{\mathfrak{a}\sqrt{|t|/\Delta}}+\left(\int_{-\infty}^{-\mathfrak{a}\sqrt{|t|/\Delta}}+\int_{\mathfrak{a}\sqrt{|t|/\Delta}}^{\infty}\right)\right)\left|h(u)\right|\sum_{j=1}^{f}\lambda_{j}\left|\log{\left|\frac{1}{2t}+\frac{\mu_j}{\lambda_{j}t}+\left(1-\frac{u}{t}\right)\ie\right|}\right|\dif{u}. 
\]
Denote by $\mathcal{I}_{11}$ and $\mathcal{I}_{12}$ the first integral and the last two integrals in the above inequality, respectively. In order to estimate $\mathcal{I}_{11}$, first note that due to assumptions \eqref{assump:lowert}, we have
\[
\left|\left|\frac{1}{2t}+\frac{\mu_j}{\lambda_{j}t}+\left(1-\frac{u}{t}\right)\ie\right|-1\right| \leq \frac{\mathfrak{a}}{\sqrt{\Delta|t|}}+\left(a^{+}+b^{+}+\frac{1}{2}\right)\frac{1}{|t|} < \frac{1}{2}
\]
for $|u|\leq\mathfrak{a}\sqrt{|t|/\Delta}$. This implies that  
\[
\left|\log{\left|\frac{1}{2t}+\frac{\mu_j}{\lambda_{j}t}+\left(1-\frac{u}{t}\right)\ie\right|}\right| \leq \frac{1}{\sqrt{\Delta|t|}}\left(\mathfrak{a}+\left(a^{+}+b^{+}+\frac{1}{2}\right)\sqrt{\frac{\Delta}{|t|}}\right)\left(1+\frac{\mathfrak{a}}{\sqrt{\Delta|t|}}+\frac{a^{+}+b^{+}+\frac{1}{2}}{|t|}\right)
\]
for $|u|\leq\mathfrak{a}\sqrt{|t|/\Delta}$ since $\left|\log{(1+x)}\right|\leq |x|(1+|x|)$ for $|x|\leq 1/2$ and $x\in\R$. Therefore,
\begin{multline}
\label{eq:I11}
\mathcal{I}_{11} \leq \sdeg\left(\mathfrak{a}+\left(a^{+}+b^{+}+\frac{1}{2}\right)\sqrt{\frac{\Delta}{|t|}}\right)\left(1+\frac{\mathfrak{a}}{\sqrt{\Delta|t|}}+\frac{a^{+}+b^{+}+\frac{1}{2}}{|t|}\right)\times \\
\times\sqrt{\frac{\Delta}{|t|}}\left(\mathfrak{m}_1\log{\left(1+\mathfrak{a}\sqrt{\Delta|t|}\right)}+\mathfrak{m}_1'\arctan\left(\mathfrak{a}\sqrt{\Delta|t|}\right)\right),
\end{multline}
where we used also~\eqref{eq:LinearOnh}. In order to estimate $\mathcal{I}_{12}$, let us first find upper lower bounds for the absolute value inside the logarithm applying \eqref{assump:lowert}. We obtain
\[
\left|\log{\left|\frac{1}{2t}+\frac{\mu_j}{\lambda_{j}t}+\left(1-\frac{u}{t}\right)\ie\right|}\right| \leq \max\left\{\log{\left(2|t|\right)},\log{\left(2+\frac{|u|}{|t|}\right)}\right\}
\]
for all $u\in\R$. 
By~\eqref{eq:QuadOnh} and \eqref{assump:lowert}, we see that then
\[
\mathcal{I}_{12} \leq \sdeg\max\left\{\log{\left(2|t|\right)}\int_{\mathfrak{a}\sqrt{|t|/\Delta}}^{\infty}\frac{\left(\mathfrak{m}_2+\frac{\mathfrak{m}_2'}{\Delta^2u^2}\right)}{u^2} \dif{u},\int_{\mathfrak{a}\sqrt{|t|/\Delta}}^{\infty}\frac{\left(\mathfrak{m}_2+\frac{\mathfrak{m}_2'}{\Delta^2u^2}\right)\log{\left(2+\frac{u}{|t|}\right)}}{u^2}\dif{u}\right\}.
\]
Because
\begin{flalign*}
\int_{\mathfrak{a}\sqrt{|t|/\Delta}}^{\infty}\frac{\log{\left(2+\frac{u}{|t|}\right)}}{u^2}\dif{u} &= \frac{1}{\mathfrak{a}}\sqrt{\frac{\Delta}{|t|}}\log{\left(2+\frac{\mathfrak{a}}{\sqrt{\Delta|t|}}\right)} + \frac{1}{2|t|}\log{\left(1+\frac{2}{\mathfrak{a}}\sqrt{\Delta|t|}\right)}\\
&\leq \frac{2}{\mathfrak{a}}\sqrt{\frac{\Delta}{|t|}} \leq \frac{1}{\mathfrak{a}}\sqrt{\frac{\Delta}{|t|}}\log{\left(2|t|\right)}
\end{flalign*}
and
\begin{align*}
\int_{\mathfrak{a}\sqrt{|t|/\Delta}}^{\infty}\frac{\log{\left(2+\frac{u}{|t|}\right)}}{u^4}\dif{u}&=\frac{8\Delta^{3/2}|t|^{3/2}\log{\left(2+\frac{\mathfrak{a}}{\sqrt{\Delta |t|}}\right)}+\mathfrak{a}^3\log{\left(1+\frac{2\sqrt{\Delta |t|}}{\mathfrak{a}}\right)}-2\mathfrak{a} \Delta\left(|t|-\mathfrak{a}\sqrt{\frac{|t|}{\Delta}}\right)}{24\mathfrak{a}^3|t|^{3}} \\
& \leq \frac{\Delta^{3/2}}{3\mathfrak{a}^3 |t|^{3/2}}\left(\log{\left(2.25\right)}+\frac{1}{64}\right)<\frac{\Delta^{3/2}}{3\mathfrak{a}^3 |t|^{3/2}}\log{\left(2|t|\right)},
\end{align*}
due to the facts that $0<\mathfrak{a}\leq\frac{1}{4}\sqrt{\Delta|t|}$, $\log(1+x) \leq x$ if $x \geq 0$, and $|t|\geq 4$,
it follows
\begin{equation}
\label{eq:I12}
\mathcal{I}_{12} \leq \frac{\sdeg}{\mathfrak{a}}\log{\left(2|t|\right)}\sqrt{\frac{\Delta}{|t|}}\left(\mathfrak{m}_2+\frac{\mathfrak{m}_2'}{3\mathfrak{a}^2 \Delta|t|}\right).
\end{equation}
The stated estimate for $\mathcal{I}_1$ follows from $\left|\mathcal{I}_1\right|\leq \mathcal{I}_{11}+\mathcal{I}_{12}$, and inequalities~\eqref{eq:I11} and~\eqref{eq:I12}.

Turning to $\mathcal{I}_2$, we can again divide the integral to different parts
\begin{multline*}
\left|\mathcal{I}_{2}\right| \leq \left(\int_{-|t|+\mathfrak{b}\frac{1}{\lambda^{-}}\sqrt{\Delta|t|}+b^{+}}^{|t|-\mathfrak{b}\frac{1}{\lambda^{-}}\sqrt{\Delta|t|}-b^{+}}+\left(\int_{-|t|-\mathfrak{b}\frac{1}{\lambda^{-}}\sqrt{\Delta|t|}-b^{+}}^{-|t|+\mathfrak{b}\frac{1}{\lambda^{-}}\sqrt{\Delta|t|}+b^{+}}+\int_{|t|-\mathfrak{b}\frac{1}{\lambda^{-}}\sqrt{\Delta|t|}-b^{+}}^{|t|+\mathfrak{b}\frac{1}{\lambda^{-}}\sqrt{\Delta|t|}+b^{+}}\right)\right. \\
\left.+\left(\int_{-\infty}^{-|t|-\mathfrak{b}\frac{1}{\lambda^{-}}\sqrt{\Delta|t|}-b^{+}}+\int_{|t|+\mathfrak{b}\frac{1}{\lambda^{-}}\sqrt{\Delta|t|}+b^{+}}^{\infty}\right)\right)\left|h(u)\right|\sum_{j=1}^{f}\lambda_{j}\left|A\left(\frac{\lambda_j}{2}+\mu_j+\lambda_j(t-u)\ie\right)\right|\dif{u}.
\end{multline*}
Denote by $\mathcal{I}_{21}$, $\mathcal{I}_{22}$ and $\mathcal{I}_{23}$ the above integrals, grouped as indicated by brackets. In order to estimate $\mathcal{I}_{21}$ and $\mathcal{I}_{23}$, first note that by~\eqref{eq:ErrorInLogGamma} one has
\[
\left|A\left(\frac{\lambda_j}{2}+\mu_j+\lambda_j(t-u)\ie\right)\right| \leq \frac{1}{2}\left(\lambda^{+}\left(\frac{1}{2}+a^{+}\right)+\frac{1}{6}+\frac{\sqrt{2}}{15\mathfrak{b}^2\Delta|t|}\right)\frac{1}{\mathfrak{b}^2\Delta|t|}
\]
for $|u|\leq |t|-\left(\mathfrak{b}/\lambda^{-}\right)\sqrt{\Delta|t|}-b^{+}$ or $|u|\geq |t|+\left(\mathfrak{b}/\lambda^{-}\right)\sqrt{\Delta|t|}+b^{+}$ since
\[
\left|\frac{\lambda_j}{2}+\mu_j+\lambda_j(t-u)\ie\right| \geq \lambda_j\left(|t-u|-\frac{1}{\lambda_j}\left|\Im\left\{\mu_j\right\}\right|\right) \geq \lambda^{-}\left(|t-u|-b^{+}\right) \geq \mathfrak{b}\sqrt{\Delta|t|}.
\]
Then
\begin{equation}
\label{eq:I21}
\mathcal{I}_{21} \leq \frac{\sdeg}{2\mathfrak{b}^2 |t|}\left(\lambda^{+}\left(\frac{1}{2}+a^{+}\right)+\frac{1}{6}+\frac{\sqrt{2}}{15\mathfrak{b}^2\Delta|t|}\right)\left(\mathfrak{m}_1\log{\left(1+\Delta|t|\right)}+\mathfrak{m}_1'\Delta\arctan{(\Delta|t|)}\right)
\end{equation}
by~\eqref{eq:LinearOnh}, and
\begin{equation}
\label{eq:I23}
\mathcal{I}_{23} \leq \frac{\sdeg}{2\mathfrak{b}^2\Delta t^2}\left(\lambda^{+}\left(\frac{1}{2}+a^{+}\right)+\frac{1}{6}+\frac{\sqrt{2}}{15\mathfrak{b}^2\Delta|t|}\right)\left(\mathfrak{m}_2+\frac{\mathfrak{m}_2'}{3\Delta^2 t^2}\right)
\end{equation}
by~\eqref{eq:QuadOnh}. In order to estimate $\mathcal{I}_{22}$, note that
\[
\left|A\left(\frac{\lambda_j}{2}+\mu_j+\lambda_j(t-u)\ie\right)\right| \leq \frac{1}{2}\left(\lambda^{+}\left(\frac{1}{2}+a^{+}\right)+\frac{1}{6}+\frac{\sqrt{2}}{15}\left(\frac{2}{\lambda^{-}}\right)^{2}\right)\left(\frac{2}{\lambda^{-}}\right)^{2}
\]
for all $t$ and $u$. Then, by~\eqref{eq:QuadOnh},
\begin{multline}
\label{eq:I22}
\mathcal{I}_{22} \leq \sdeg\left(\lambda^{+}\left(\frac{1}{2}+a^{+}\right)+\frac{1}{6}+\frac{\sqrt{2}}{15}\left(\frac{2}{\lambda^{-}}\right)^{2}\right)\left(\frac{2}{\lambda^{-}}\right)^{2}\frac{\frac{\mathfrak{b}}{\lambda^{-}}\sqrt{\Delta|t|}+b^{+}}{t^2-\left(\frac{\mathfrak{b}}{\lambda^{-}}\sqrt{\Delta|t|}+b^{+}\right)^2}\times \\
\times\left(\mathfrak{m}_2+\frac{\mathfrak{m}_2'\left(3t^2+\left(\frac{\mathfrak{b}}{\lambda^{-}}\sqrt{\Delta|t|}+b^{+}\right)^2\right)}{3\Delta^2 \left(t^2-\left(\frac{\mathfrak{b}}{\lambda^{-}}\sqrt{\Delta|t|}+b^{+}\right)^2\right)^2}\right).
\end{multline}
The stated estimate for $\mathcal{I}_2$ follows from $\left|\mathcal{I}_2\right|\leq \mathcal{I}_{21}+\mathcal{I}_{22}+\mathcal{I}_{23}$, and inequalities~\eqref{eq:I21},~\eqref{eq:I23} and~\eqref{eq:I22}.
\end{proof}

\begin{corollary}
\label{corollary:UpperGW}
Assume the same assumptions and notation as in Proposition~\ref{prop:FourierGamma}. Additionally, assume that there exist non-negative real numbers $\mathfrak{m}_3$, $\mathfrak{m}_3'$ and $\mathfrak{M}_2$ that do not depend on $z\in\C$, such that 
\[
\left|h(z)\right| \leq \left(\mathfrak{m}_3+\mathfrak{m}_3'\frac{1+\Delta|z|}{1+\Delta^2 |z|^2}\right)\frac{\Delta^2}{1+\Delta|z|}e^{2\pi\Delta\left|\Im\left\{z\right\}\right|} \quad \textrm{for} \quad |z|\geq \mathfrak{M}_2.
\]
Then
\[
\frac{1}{2}\sum_{\gamma}h(t-\gamma) = \frac{1}{4\pi}\widehat{h}(0)\log{\tau} + \frac{1}{2\pi}\mathcal{I} + \mathcal{I}_3 + \mathcal{I}_4,
\]
for, beside conditions~\eqref{assump:lowert2}, $|t|\geq\mathfrak{M}_2$, where $\mathcal{I}$ is as in Proposition \ref{prop:FourierGamma},
\begin{equation}
\label{def:I3I4}
\left|\mathcal{I}_3\right| \leq m_{\cL}\Delta^2\left(\frac{\mathfrak{m}_3}{1+\Delta|t|}+\frac{\mathfrak{m}_3'}{1+\Delta^2 t^2}\right)e^{\pi\Delta}, \quad 
\left|\mathcal{I}_4\right| \leq \frac{1}{2\pi}\sum_{n=2}^{\infty}\frac{\left|\Lambda_{\cL}(n)\right|}{\sqrt{n}}\left|\widehat{h}\left(\frac{\log{n}}{2\pi}\right)\right|.
\end{equation}
\end{corollary}

\begin{proof}
This result immediately follows from Lemma~\ref{lem:GW} and Proposition~\ref{prop:FourierGamma}.
\end{proof}

\section{Explicit estimates for $g_\Delta(z)$ and $m_{\Delta}(z)$}
\label{sec:gmExplicitBounds}

In this section, we derive explicit estimates for the extremal functions $g_\Delta(z)$ and $m_{\Delta}(z)$ that can be used to derive explicit bounds in Proposition~\ref{prop:FourierGamma} and Corollary~\ref{corollary:UpperGW}. Remember that $g_\Delta(z)$ and $m_{\Delta}(z)$ are given in~\eqref{eq:DefGDelta} and~\eqref{def:Mandm}, respectively. We follow the ideas and structure of~\cite[Section 4]{CarneiroChandee}, where various non-explicit estimates for $g_\Delta(z)$ and $m_{\Delta}(z)$ have been provided. We firstly provide a preparatory lemma on the modulus of $(1/z)\sin{z}$, and then derive formulas that are later used to prove upper bounds for $G_\Delta(x)=g_\Delta\left(x/\Delta\right)$ and $M_\Delta(x)=m_\Delta\left(x/\Delta\right)$. In Section~\ref{subsec:Explicitgm} we prove explicit estimates.

\subsection{Preliminary estimates for $g_\Delta(z)$ and $m_{\Delta}(z)$}
\label{subsec:preliminarymg}

We start with a lemma about an estimate of the modulus of $(1/z)\sin{z}$ for $z\in\C$.

\begin{lemma}
\label{lemma:Estsinz}
Let $\xi\in\C$. Then
\begin{equation}
\label{eq:deltadef}
\left|\frac{\sin\left(\pi\xi\right)}{\pi \xi}\right|^2
\leq \delta{(\xi)}\frac{e^{2\pi|\Im\{\xi\}|}}{1+|\xi|^2}, \quad 
\delta{(\xi)} \de 
\begin{cases}
      1, & |\xi|<1, \\
      \frac{2}{\pi^2}, & |\xi|\geq 1.
\end{cases}
\end{equation}
\end{lemma}

\begin{proof}
Firstly, we will prove the lemma for $|\xi|<1$. For $z\in\Omega^{+}\de\left\{w\in\C\colon |w|<1,\Im\{w\}>0\right\}$ define 
\[
f^{+}(z) \de \left|\frac{e^{2\pi\ie z}-1}{2\pi\ie z}\right|^{2}\left(1+|z|^2\right),
\]
and for $z\in\Omega^{-}\de\left\{w\in\C\colon\bar{w}\in\Omega^{+}\right\}$ define
\[
f^{-}(z) \de \left|\frac{e^{-2\pi\ie z}-1}{2\pi\ie z}\right|^{2}\left(1+|z|^2\right).
\]
Observe that $f^{+}$ and $f^{-}$ have natural extensions to closed domains $\overline{\Omega^{+}}$ and $\overline{\Omega^{-}}$ in $\C$, respectively, with $f^{+}(0)=f^{-}(0)=1$. Note that $f^{-}(z)=f^{+}(\bar{z})$ for $z\in\overline{\Omega^{-}}$. We are going to prove that $f^{+}(z)\leq 1$ for $z\in\overline{\Omega^{+}}$. Assuming this, we then have $f^{-}(z)\leq 1$ for $z\in\overline{\Omega^{-}}$, and therefore
\[
f^{+}(\xi) = \left|\frac{\sin{(\pi\xi)}}{\pi\xi}\right|^2 e^{-2\pi\Im\{\xi\}}\left(1+|\xi|^2\right) \leq 1
\]
for $\xi\in\overline{\Omega^{+}}$, and
\[
f^{-}(\xi) = \left|\frac{\sin{(\pi\xi)}}{\pi\xi}\right|^2 e^{2\pi\Im\{\xi\}}\left(1+|\xi|^2\right) \leq 1
\]
for $\xi\in\overline{\Omega^{-}}$. From the previous two inequalities we can deduce that for $|\xi|\leq 1$ one has
\[
\left|\frac{\sin{(\pi\xi)}}{\pi\xi}\right|^2 \leq \frac{e^{2\pi\left|\Im\{\xi\}\right|}}{1+|\xi|^2}.
\]
In order to prove $f^{+}(z)\leq 1$ for $z\in\overline{\Omega^{+}}$, note that $f^{+}(z)$ is a subharmonic function on domain $\Omega^{+}$. By the maximum principle we thus have
\[
f^{+}(z) \leq \max\left\{\sup_{x\in[-1,1]}\left\{f^{+}(x)\right\},\sup_{\varphi\in[0,\pi]}\left\{f^{+}\left(e^{\ie\varphi}\right)\right\}\right\}
\]
for every $z\in\overline{\Omega^{+}}$. Because
\[
f^{+}(x) = \frac{2\left(1-\cos{(2\pi x)}\right)\left(1+x^2\right)}{(2\pi x)^{2}} \leq 1,
\quad
f^{+}\left(e^{\ie\varphi}\right) = \frac{e^{-4\pi\sin{\varphi}}-2e^{-2\pi\sin{\varphi}}\cos{\left(2\pi\cos{\varphi}\right)}+1}{2\pi^2} \leq \frac{2}{\pi^2} < 1
\]
for $x\in[-1,1]$ and $\varphi\in[0,\pi]$, respectively, the result for $|\xi|<1$ now follows.

Let $|\xi| \geq 1$. The preceding paragraph also shows that then
\begin{equation*}
   \left|\frac{\sin(\pi\xi)}{\pi\xi}\right|^2 \leq \frac{1}{\pi^2}\left(1+\frac{1}{|\xi|^2}\right)\frac{e^{2\pi|\Im\{\xi\}|}}{1+|\xi|^2}.
\end{equation*}
The final result follows.
\end{proof}

Next, we derive some bounds for the functions $G_\Delta(z)$ and $M_\Delta(z)$ that are defined by~\eqref{eq:DefGDelta} and~\eqref{def:Mandm}, respectively. We are distinguishing two cases, namely when $z$ is an arbitrary complex number, and when $z$ is a real number.

\begin{lemma}
\label{lemma:GMSums}
Let $\Delta>0$, $|\sigma-1/2|\leq 2$, $\sigma \neq 1/2$ and $z=x+\ie y$ with $(x,y)\in\R^2$. Also, let $\delta(\xi)$ be defined as in~\eqref{eq:deltadef}. Then
\begin{multline}
\label{eq:upperForG}
\left|G_{\Delta}(z)\right| \leq \frac{\Delta^2}{1+|z|}e^{2\pi|y|}\left(\sum_{n=1}^{\infty}\delta\left(\xi_1\right)\frac{1+|z|}{1+\left|\xi_1\right|^2}\left(\frac{4}{\left(n-\frac{1}{2}\right)^2}+\frac{8\left|\xi_1\right|}{\left(n-\frac{1}{2}\right)^3}\right)\right. \\
\left.+\sum_{n=1}^{\infty}\delta\left(\xi_2\right)\frac{1+|z|}{1+\left|\xi_2\right|^2}\left(\frac{4}{\left(n-\frac{1}{2}\right)^2}+\frac{8\left|\xi_2\right|}{\left(n-\frac{1}{2}\right)^3}\right)\right),
\end{multline}
where 
\begin{equation}
\label{eq:xi12}
\xi_1=\xi_1(z,n)\de z-n+\frac{1}{2}, \quad \xi_2=\xi_2(z,n)\de z+n-\frac{1}{2},
\end{equation}
and  
\begin{multline}
\label{eq:upperForM}
\left|M_{\Delta}(z)\right| \leq \frac{\Delta^2}{1+|z|}e^{2\pi|y|}\left(\frac{2\delta(z)\left(1+|z|\right)}{\left(1+|z|^2\right)\Delta^{2}}\log{\frac{2}{\left|\sigma-\frac{1}{2}\right|}}+\sum_{n=1}^{\infty}\delta\left(\xi_3\right)\frac{1+|z|}{1+\left|\xi_3\right|^2}\left(\frac{4}{n^2}+\frac{8\left|\xi_3\right|}{n^3}\right)\right. \\
\left.+\sum_{n=1}^{\infty}\delta\left(\xi_4\right)\frac{1+|z|}{1+\left|\xi_4\right|^2}\left(\frac{4}{n^2}+\frac{8\left|\xi_4\right|}{n^3}\right)\right),
\end{multline}
where 
\begin{equation}
\label{eq:xi34}
\xi_3=\xi_3(z,n)\de z-n, \quad \xi_4=\xi_4(z,n)\de z+n.
\end{equation}
Also, for $x\in\R$ we have
\begin{equation}
\label{eq:lowerForG}
G_\Delta(x) \geq -\frac{\Delta^2}{\Delta^2 + x^2}\sum_{n=1}^{\infty}\frac{16\left(\Delta^2+x^2\right)\sqrt{\delta\left(\xi_1(x,n)\right)\delta\left(\xi_2(x,n)\right)}}{\left(\left(n-\frac{1}{2}\right)^{2}+4\Delta^2\right)\sqrt{\left(1+\left|\xi_1(x,n)\right|^{2}\right)\left(1+\left|\xi_2(x,n)\right|^{2}\right)}}
\end{equation}
and
\begin{flalign}
\frac{\Delta^2+x^2}{\Delta^2}M_{\Delta}(x) &\leq \left(\frac{1}{1+x^2}+\frac{1}{\Delta^2}\right)2\delta(x)\log{\frac{2}{\left|\sigma-\frac{1}{2}\right|}} \nonumber \\
&+\sum_{n=1}^{\infty}\frac{4\left(\Delta^2+x^2\right)}{n^2+\left(\left(\sigma-\frac{1}{2}\right)\Delta\right)^2}\left(\frac{\delta\left(\xi_3(x,n)\right)}{1+\left|\xi_3(x,n)\right|^2}+\frac{\delta\left(\xi_4(x,n)\right)}{1+\left|\xi_4(x,n)\right|^2}\right) \nonumber \\
&+\sum_{n=1}^{\infty}\frac{16\left(\Delta^2+x^2\right)\sqrt{\delta\left(\xi_3(x,n)\right)\delta\left(\xi_4(x,n)\right)}}{\left(n^{2}+4\Delta^2\right)\sqrt{\left(1+\left|\xi_3(x,n)\right|^{2}\right)\left(1+\left|\xi_4(x,n)\right|^{2}\right)}}. \label{eq:upperForMreal}
\end{flalign}
\end{lemma}

\begin{proof}
Both $G_{\Delta}(z)$ and $M_{\Delta}(z)$ depend on the function $f_{\sigma}(x)$ and its derivative, so we need to estimate them first. The following inequalities
\begin{equation}
\label{eq:fBounds}
    0 \leq f_\sigma(x) \leq \frac{4}{x^2+\left(\sigma-\frac{1}{2}\right)^2} \quad\text{and}\quad \left|f_\sigma '(x)\right| \leq \frac{8|x|}{\left(x^2+4\right)\left(x^2+\left(\sigma-\frac{1}{2}\right)^2\right)}.
\end{equation} 
hold, see~\cite[Equation (4.6)]{CarneiroChandee}. Also, note that $f_{\sigma}(0)=2\log{(2/\left|\sigma-1/2\right|)}$. By~\eqref{eq:fBounds} and Lemma~\ref{lemma:Estsinz} we thus obtain inequalities \eqref{eq:upperForG} and \eqref{eq:upperForM}.

Let as now consider the case where $z=x\in\R$. Note that $f_\sigma(x)=f_\sigma(-x)$. Then pairing terms $n$ with terms $1-n$ when $n \geq 1$ gives
\[
G_\Delta(x) \geq -\sum_{n=1}^\infty \left|\frac{\sin\left(\pi\left(x-n+\frac{1}{2}\right)\right)}{\pi\left(x-n+\frac{1}{2}\right)}\right|\cdot 
\left|\frac{\sin\left(\pi\left(x+n-\frac{1}{2}\right)\right)}{\pi\left(x+n-\frac{1}{2}\right)}\right|\cdot\frac{2\left(n-\frac{1}{2}\right)}{\Delta}\cdot\left|f_\sigma'\left(\frac{n-\frac{1}{2}}{\Delta}\right)\right|.
\]
Combining terms $n$ with terms $-n$ when $n \geq 1$ gives
\begin{multline*}
M_{\Delta}(x) \leq f_{\sigma}(0)\left(\frac{\sin{(\pi x)}}{\pi x}\right)^{2} + \sum_{n=1}^{\infty}\left(\left(\frac{\sin\left(\pi(x-n)\right)}{\pi(x-n)}\right)^{2}+\left(\frac{\sin\left(\pi(x+n)\right)}{\pi(x+n)}\right)^{2}\right)f_{\sigma}\left(\frac{n}{\Delta}\right) \\
+ \sum_{n=1}^{\infty}\left|\frac{\sin\left(\pi(x-n)\right)}{\pi(x-n)}\right|\cdot\left|\frac{\sin\left(\pi(x+n)\right)}{\pi(x+n)}\right|\cdot\frac{2n}{\Delta}\cdot\left|f_{\sigma}'\left(\frac{n}{\Delta}\right)\right|.
\end{multline*}
Then~\eqref{eq:fBounds} and Lemma~\ref{lemma:Estsinz} imply inequalities \eqref{eq:lowerForG} and \eqref{eq:upperForMreal}.
\end{proof}

\subsection{Explicit upper bounds}
\label{subsec:Explicitgm}

Now we are ready to prove explicit upper bounds for the terms $\left|g_{\Delta}(z)\right|$ and $\left|m_{\Delta}(z)\right|$. We need three different type of estimates: for all real numbers $z=x$, for complex numbers with absolute values that are large enough, and for real numbers with absolute values that are large enough. The bounds for the first two cases are derived in Lemmas~\ref{lemma:m2g} and~\ref{lemma:mUpperRealComplex}, and the last one in Lemmas~\ref{lemma:gAbsolute} and~\ref{lemma:mUpperReal}.

All of the proofs use essentially the same idea. First, we estimate the wanted term using Lemma~\ref{lemma:GMSums}. Then we divide the obtained sum to five cases depending on how close $n$ and $|z|$ are to each other, and after that we estimate these sums. Due to the similarity of the ideas, only Lemma~\ref{lemma:m2g} us proved with full detail. However, since in order to get explicit estimates, it is very important how the terms in the five sums are estimated and these vary a little bit from case to case, we provide the main estimates in other proofs, too.

\begin{lemma}
\label{lemma:m2g}
Let $\Delta>0$, $|\sigma-1/2|\leq 2$ and $\sigma \neq 1/2$. For every $x\in\R$ we have
\begin{equation}
\label{eq:UpperBoundReal}
\left|g_{\Delta}(x)\right| \leq \frac{121\Delta^2}{1+\Delta|x|}.
\end{equation}
For $z=x+\ie y$ with $(x,y)\in\R^2$ and $|z|\geq 350/\Delta$, we have
\begin{equation}
\label{eq:UpperBoundImaginary}
\left|g_{\Delta}(z)\right| \leq \frac{28\Delta^2}{1+\Delta\left|z\right|}e^{2\pi\Delta|y|}.
\end{equation}
\end{lemma}

\begin{proof}
Remember that $g_\Delta(z)=G_\Delta(\Delta z)$. Let use estimate $G_\Delta(\Delta z)$. Define
\[
S(z,n) \de \delta\left(\xi_1\right)\frac{1+|z|}{1+\left|\xi_1\right|^2}\left(\frac{4}{\left(n-\frac{1}{2}\right)^2}+\frac{8\left|\xi_1\right|}{\left(n-\frac{1}{2}\right)^3}\right)
+\delta\left(\xi_2\right)\frac{1+|z|}{1+\left|\xi_2\right|^2}\left(\frac{4}{\left(n-\frac{1}{2}\right)^2}+\frac{8\left|\xi_2\right|}{\left(n-\frac{1}{2}\right)^3}\right),
\]
where $\xi_1$ and $\xi_2$ are defined in~\eqref{eq:xi12} and $\delta$ as in \eqref{eq:deltadef}. By~\eqref{eq:upperForG} we then have 
\begin{equation}
\label{eq:GwithS}
\left|G_{\Delta}(z)\right| \leq \frac{\Delta^2}{1+|z|}e^{2\pi|y|}\sum_{n=1}^{\infty}S(z,n).
\end{equation}
We need to estimate the sum in~\eqref{eq:GwithS}. Take $N\in\N$, $0<\nu_1<1$ and $\nu_2>1$. We can write $\sum_{n=1}^{\infty}S(z,n) = \left(\sum_{n\leq N}+\sum_{n>N}\right)S(z,n)$ and
\begin{multline}
\label{eq:sumpart}
\sum_{n>N}S(z,n) \leq \left(\sum_{N-\frac{1}{2}<n-\frac{1}{2}\leq\nu_1|z|}+\sum_{\nu_1|z|<n-\frac{1}{2}\leq|z|-1} 
\vphantom{\sum_{\max\left\{\frac{1}{2}\right\}}}\right. \\
\left.+\sum_{\max\left\{N-\frac{1}{2},|z|-1\right\}<n-\frac{1}{2}\leq|z|+1}+\sum_{|z|+1<n-\frac{1}{2}\leq\nu_{2}|z|}+\sum_{\max\left\{N-\frac{1}{2},\nu_{2}|z|, |z|+1\right\}<n-\frac{1}{2}}\right)S(z,n).
\end{multline}
Denote by $S_{i}$, $i\in\{1,\ldots,5\}$, the sums on the right hand-side of~\eqref{eq:sumpart}, written in the same order. If for some index $i$ there are no integers $n$ that are in the required domain, then we set $S_i=0$. In particular, if $z=0$, then we have $\sum_{n=1}^{\infty}S(z,n)$ without terms $S_i$. We are going to estimate each of the sums $S_i$.

First, note that $S_1=0$ if $|z|< \max\left\{(N-1/2)/\nu_1,3/2\right\}$. Further, note that $N-1/2<n-1/2\leq\nu_1|z|$ implies $n>N$ and $1\leq\left(1-\nu_1\right)|z|\leq\left|\xi_i\right|\leq \left(1+\nu_1\right)|z|$ for $i\in\{1,2\}$. Therefore,
\begin{equation}
\label{eq:S1}
S_{1} \leq 
\begin{cases}
0, & |z| < \frac{N-1/2}{\nu_1}, \\
4\left(1+\frac{2}{\pi^2}\right)\left(\frac{1+|z|}{1+\left(1-\nu_1\right)^{2}|z|^{2}}\sum\limits_{n>N}\frac{1}{\left(n-\frac{1}{2}\right)^{2}}+\frac{2\left(1+\nu_1\right)\left(1+|z|\right)|z|}{1+\left(1-\nu_1\right)^{2}|z|^{2}}\sum\limits_{n>N}\frac{1}{\left(n-\frac{1}{2}\right)^{3}}\right), & |z| \geq \frac{N-1/2}{\nu_1}.
\end{cases}
\end{equation}

Let us now consider $S_2$. First, note that $S_2=0$ if $|z|<1/(1-\nu_1)$. Particularly, $S_2=0$ if $|z|\leq 1$. Hence, we can assume that $|z| >1$. Further, note that $\nu_1|z|<n-\frac{1}{2}\leq|z|-1$ implies $\left|\xi_i\right|\geq 1$ and $\left|\xi_i\right|\leq 2|z|-1$ for $i\in\{1,2\}$. Therefore, using
\begin{equation}
\label{eq:harmonic}
\sum_{1\leq X<n\leq Y}\frac{1}{n} \leq \log{\frac{Y}{X}} + \frac{1}{2Y} + \frac{2\left(\log{2}+\gamma-1\right)}{X},
\end{equation}
see~\cite{Young}, we get
\begin{flalign}
S_{2} &\leq \frac{8}{\pi^2}\left(1+\frac{4}{\nu_1}\right)\frac{1+|z|}{\nu_{1}|z|}\sum_{\nu_{1}|z|+\frac{1}{2}<n\leq|z|-\frac{1}{2}}\frac{1}{n-\frac{1}{2}} \nonumber \\ 
&\leq \frac{8}{\pi^2}\left(1+\frac{4}{\nu_1}\right)\frac{1+|z|}{\nu_{1}|z|}\left(1+\frac{1}{2\nu_{1}|z|}\right)\sum_{\nu_{1}|z|<n\leq|z|}\frac{1}{n} \nonumber \\
&\leq \frac{8}{\pi^2}\left(1+\frac{4}{\nu_1}\right)\frac{1}{\nu_1}\left(1+\frac{1}{|z|}\right)\left(1+\frac{1}{2\nu_{1}|z|}\right)\left(\log{\frac{1}{\nu_1}}+\frac{1}{|z|}\left(\frac{1}{2}+\frac{2}{\nu_1}\left(\log{2}+\gamma-1\right)\right)\right) \label{eq:S2}
\end{flalign}
if $|z|\geq 1/(1-\nu_1)$, where $\gamma$ is the Euler--Mascheroni constant. 

Now, we note that $S_3=0$ if $N \geq 3$ and $|z|<3/2$. Further, note that $|z|-1<n-\frac{1}{2}\leq|z|+1$ implies $\left|\xi_i\right|\leq 2|z|+1$ for $i\in\{1,2\}$, and also that there are at most two such positive integers $n$. Therefore,
\begin{equation}
\label{eq:S3}
S_3 \leq 
\begin{cases}
0, & |z|<N-\frac{1}{2}, \\
16\frac{1+|z|}{\left(|z|-1\right)^2}\left(5+\frac{6}{|z|-1}\right), & |z|\geq N-\frac{1}{2}.
\end{cases}
\end{equation}

For $S_4$, first note that we have $S_4=0$ if $|z|<1/(\nu_2-1)$. Particularly, $S_4=0$ if $z=0$. Assume that $z \neq 0$. Further, note that $|z|+1<n-\frac{1}{2}\leq\nu_{2}|z|$ implies $\left|\xi_i\right|>1$ and $\left|\xi_i\right|\leq \left(1+\nu_2\right)|z|$ for $i\in\{1,2\}$. Taking similar approach as in~\eqref{eq:S2}, we then have
\begin{flalign}
S_4 &\leq \frac{8}{\pi^2}\left(1+2\left(1+\nu_2\right)\right)\left(1+\frac{1}{2\left(|z|+1\right)}\right)\sum_{|z|+\frac{1}{2}<n\leq\nu_2\left(|z|+\frac{1}{2}\right)}\frac{1}{n} \nonumber \\
&\leq \frac{8}{\pi^2}\left(1+2\left(1+\nu_2\right)\right)\left(1+\frac{1}{2\left(|z|+1\right)}\right)\left(\log{\nu_2}+\frac{1}{|z|}\left(\frac{1}{2\nu_2}+2\left(\log{2}+\gamma-1\right)\right)\right), \label{eq:S4}
\end{flalign}
if $|z| \geq 1/(\nu_2-1)$, by again using~\eqref{eq:harmonic}.

Let us now consider $S_5$. If $z=0$, there is no need to divide the sum to cases $n\leq N$ and $n>N$, and we can assume that $z \neq 0$. Moreover, the fact that $\max\left\{\nu_{2}|z|, |z|+1\right\}<n-\frac{1}{2}$ implies 
\[
\max\left\{1,\left(\nu_2-1\right)|z|\right\}\leq\left|\xi_i\right|\leq\left(1+\frac{1}{\nu_2}\right)\left(n-\frac{1}{2}\right)
\]
for $i\in\{1,2\}$. Because
\[
\sum_{\nu_2|z|<n-\frac{1}{2}}\frac{1}{\left(n-\frac{1}{2}\right)^{2}} \leq \left(\frac{1}{\nu_2|z|}\right)^{2} + \int_{\nu_2|z|}^{\infty}\frac{\dif{u}}{u^2} = \frac{1}{\nu_2|z|}\left(1+\frac{1}{\nu_2|z|}\right),
\]
we obtain
\begin{equation}
\label{eq:S5}
S_5 \leq 
\frac{16}{\pi^2}\left(1+2\left(1+\frac{1}{\nu_2}\right)\right)\left(1+\frac{1}{\nu_2|z|}\right)\frac{1+|z|}{\nu_2|z|\left(1+\left(\nu_2-1\right)^{2}|z|^2\right)}.
\end{equation}

Observe that for fixed $N$, $\nu_1$ and $\nu_2$, the sum of the upper bounds from~\eqref{eq:S1},~\eqref{eq:S2},~\eqref{eq:S3},~\eqref{eq:S4} and~\eqref{eq:S5} is a decreasing function in $|z|$ if $|z|$ is sufficiently large. 

Let us now prove inequality \eqref{eq:UpperBoundImaginary}. First, let us consider the case $350\leq |z|< 699.5$. Set $N=700$, $\nu_1=0.99999$ and $\nu_2=1.03167$. Now, $S_1=S_2=S_3=0$, and let us use~\eqref{eq:S4} and~\eqref{eq:S5} to bound $S_4$ and $S_5$, respectively. Note that $S_4+S_5$ with such parameters is a decreasing function in $|z|$ and hence its maximum is obtained at $|z|=350$. Further, for $n \leq 348$, we have
\begin{equation}
\label{def:Ssmall}
    S(z,n) \leq  \frac{4}{\pi^2}\frac{1+|z|}{1+\left(|z|-n+\frac{1}{2}\right)^2}\left(\frac{4}{\left(n-\frac{1}{2}\right)^2}+\frac{8\left(\left|z\right|+n-\frac{1}{2}\right)}{\left(n-\frac{1}{2}\right)^3}\right)=: S_{\text{small}}(z,n).
\end{equation}
For $349 \leq n \leq 700$ we can replace the coefficient $4/\pi^2$ with $1+2/\pi^2$ and we can denote this estimate by $S_{\text{large}}(z,n)$. Hence, for $350\leq |z|< 699.5$ we obtain\footnote{Here, and in the following paragraphs, we are using \texttt{FindMaximum} in \emph{Mathematica} to calculate $\max\{\cdot\}$ over bounded regions.}
\begin{flalign*}
\sum_{n=1}^\infty S(z,n) &\leq \sum_{n=1}^{700} S(z,n) +S_4+S_5 \\
&\leq \max_{350\leq |z|<699.5}\left\{\sum_{n=1}^{348} S_{\text{small}}(z,n)+\sum_{n=349}^{700} S_{\text{large}}(z,n)\right\} +S_4+S_5<27.725+0.216<28.
\end{flalign*}
Let us now move to the case $|z| \geq 699.5$. Choose $N=700$, $\nu_1=0.97658$ and $\nu_2=1.02002$. Then, using upper bounds for the terms $S_i$, $\sum_{i=1}^5 S_i$ is decreasing for $|z| \geq 699.5$. Similarly as in the previous case, we can deduce that
\begin{equation*}
    \sum_{n=1}^\infty S(z,n)\leq \max_{|z|\geq 699.5}\left\{\sum_{n=1}^{698} S_{\text{small}}(z,n)+\sum_{n=699}^{700} S_{\text{large}}(z,n)\right\}+\sum_{i=1}^5 S_i<27.4599+0.403<28.
\end{equation*}
Hence, we have proved inequality~\eqref{eq:UpperBoundImaginary}.

Finally, let us prove inequality~\eqref{eq:UpperBoundReal}. Hence, we assume that $z \in \R$. First, if $0 \leq |z| \leq 1/2$, then 
\begin{equation*}
    \sum_{n=1}^{\infty} S(z,n) \leq \max_{0 \leq |z| \leq 1/2} \left\{S(z,1)\right\}+\frac{3}{\pi^2}\sum_{n=2}^\infty\left( \frac{4}{13}\left(\frac{4}{\left(n-\frac{1}{2}\right)^2}+\frac{8n}{\left(n-\frac{1}{2}\right)^3}\right)
+\frac{6}{\left(n-\frac{1}{2}\right)^2}\right)<95.
\end{equation*}

Consider now the case $1/2<|z|\leq 1$. Now, in the case $n>N\geq 3$, the non-zero contributions can come from the terms $S_1$, $S_4$ and $S_5$. We choose $N=50$, $\nu_1=2.69261\cdot 10^{-6}$ and $\nu_2=2.2002$, which in turn implies also $S_1=0$. Then
\[
\sum_{n=1}^{\infty} S(z,n) \leq \max_{1/2<|z|\leq 1}\left\{\sum_{n=1}^{50}S(z,n)\right\} + S_4 + S_{5} < 112
\]
for $1/2<|z|\leq 1$. 

Let us now consider the case $1<|z|<3/2$. First we prove that if $n \geq 4$ is fixed, then $S(z,n)$ gets its maximum value when $|z| \to (3/2)^-$. We have
$S(z,n) \leq S_{\text{small}}(z,n)$, where $S_{\text{small}}$ is given as in \eqref{def:Ssmall}.
Now the numerator of the derivative of $S_{\text{small}}$ with respect to $|z|$ is
\begin{multline*}
    \frac{4}{\pi^2}\left(17+24n^3+4|z|(16+3|z|)+4n^2(7+8|z|)-2n(11+4|z|(10+7|z|))\right) \\
    >\frac{4}{\pi^2}\left(17+24n^3+4(16+3)+4n^2(7+8)-2n(11+6(10+10.5))\right)>0
\end{multline*}
for all $n \geq 4$. Since the denominator is also positive, the function $S_{\text{small}}$ is increasing when $1\leq |z|\leq 3/2$ and $n \geq 4$. Further, considering the cases $n=1, 2$ and $3$ separately, we can conclude that in the cases $n=1$ and $n=3$ the $S(z,n)$ is increasing with respect to $|z|$ and in the cases $n=2$ obtains its maximum at point $|z| \approx 1.222$. Thus we can conclude
\begin{equation*}
    \sum_{n=1}^{\infty} S(z,n) \leq \lim_{x \to (3/2)^-} \left(S(z,1)+S(z,3)\right)+\max_{1<|z|<3/2} S(z,2)+\sum_{n=4}^{\infty} S_{\text{small}}(3/2,n)<120.430+0.388<121.
\end{equation*}

Finally, let us consider the case $3/2\leq |z| <350$. Let us again choose $N=350$, $\nu_1=0.999$ and $\nu_2= 1.72537$. Now $S_1=S_2=S_3=0$. The terms $S_4$ and $S_5$ are now decreasing for $|z|$. Hence,
\begin{equation*}
    \sum_{n=1}^{\infty} S(z,n) \leq \max_{3/2 \leq |z| < 350} \left\{\sum_{n=1}^{350} S(z,n)\right\}+ S_4+S_5<40.754+11.032<52.
\end{equation*}
Thus bound \eqref{eq:UpperBoundReal} is proved.
\end{proof}


\begin{remark}
Since $\sum_{n=1}^{100} S(1.499, n)\approx 120.002>120$, $121$ is the smallest integer that can be used in the right-hand side of \eqref{eq:UpperBoundReal} using terms $S(z,n)$.
\end{remark}

\begin{remark}
Since $\sum_{n=1}^{400} S(350, n)\approx 27.4047>27$, $28$ is the smallest integer that can be used in the right-hand side of \eqref{eq:UpperBoundImaginary} using terms $S(z,n)$.
\end{remark}

\begin{lemma}
\label{lemma:mUpperRealComplex}
Let $\Delta$ denote a positive real number, and $|\sigma-1/2|\leq 2$ and $\sigma \neq 1/2$. For every $x\in\R$ we have
\begin{equation}
\label{eq:UpperBoundRealCase2}
\left|m_{\Delta}(x)\right| \leq \left( \frac{2\left(1+\Delta|x|\right)}{\left(1+\Delta^2|x|^2\right)\Delta^2}\log{\frac{2}{\left|\sigma-\frac{1}{2}\right|}}+24\right)\frac{\Delta^2}{1+\Delta|x|}.
\end{equation}
For $z \in \C$ and $|z|\geq \num{10000}/\Delta$, we have
\begin{equation}
\label{eq:UpperBoundImaginaryCase2}
\left|m_{\Delta}(z)\right| \leq \left( \frac{4\left(1+\Delta|z|\right)}{\pi^2\left(1+\Delta^2|z|^2\right)\Delta^2}\log{\frac{2}{\left|\sigma-\frac{1}{2}\right|}}+4\right)\frac{\Delta^2}{1+\Delta\left|z\right|}e^{2\pi\Delta|\Im(z)|}.
\end{equation}
\end{lemma}

\begin{proof}
The claim follows similarly as in Lemma~\ref{lemma:m2g}. Remember that $m_\Delta(z)=M_\Delta(\Delta z)$. Let use estimate $M_\Delta(\Delta z)$. Let us first set $z=x+\ie y$ and define
\[
S(z,n) \de \delta\left(\xi_3\right)\frac{1+|z|}{1+\left|\xi_3\right|^2}\left(\frac{4}{n^2}+\frac{8\left|\xi_3\right|}{n^3}\right)
+\delta\left(\xi_4\right)\frac{1+|z|}{1+\left|\xi_4\right|^2}\left(\frac{4}{n^2}+\frac{8\left|\xi_4\right|}{n^3}\right),
\]
where $\xi_3$ and $\xi_4$ are defined in~\eqref{eq:xi34} and $\delta$ in~\eqref{eq:deltadef}. By~\eqref{eq:upperForM}, we have
\begin{equation}
\label{eq:M}
\left|M_{\Delta}(z)\right| \leq \frac{\Delta^2}{1+|z|}e^{2\pi|y|}\left(\frac{2\delta(z)\left(1+|z|\right)}{\left(1+|z|^2\right)\Delta^2}\log{\frac{2}{\left|\sigma-\frac{1}{2}\right|}}+\sum_{n=1}^{\infty} S(n,z)\right).
\end{equation}
Next, we derive a bound for the sum $\sum_{n=1}^{\infty} S(n,z)$.

Take $N\in\N$, $0<\nu_1<1$ and $\nu_2>1$. Let us divide the sum $\sum_{n>N} S(n,z)$ to parts $S_1, S_1, S_3, S_4$ and $S_5$ similarly as in~\eqref{eq:sumpart} but instead of $N-1/2$ and $n-1/2$ let us use $N$ and $n$. Similarly as in the proof of Lemma~\ref{lemma:m2g}, we obtain
\begin{equation*}
S_{1} \leq 
\begin{cases}
0, & |z| \leq \frac{N}{\nu_1}, \\
4\left(1+\frac{2}{\pi^2}\right)\left(\frac{1+|z|}{1+\left(1-\nu_1\right)^{2}|z|^{2}}\sum\limits_{n>N}\frac{1}{n^{2}}+\frac{2\left(1+\nu_1\right)\left(1+|z|\right)|z|}{1+\left(1-\nu_1\right)^{2}|z|^{2}}\sum\limits_{n>N}\frac{1}{n^{3}}\right), & |z| > \frac{N}{\nu_1},
\end{cases}
\end{equation*}   
\begin{equation*}
S_{2} \leq 
\begin{cases}
    0, & |z| \leq \frac{1}{1-\nu_1}, \\
    \frac{8}{\pi^2}\left(1+\frac{4}{\nu_1}\right)\frac{1}{\nu_1}\left(1+\frac{1}{|z|}\right)\left(\log{\frac{1}{\nu_1}}+\frac{1}{|z|}\left(\frac{1}{2}+\frac{2}{\nu_1}\left(\log{2}+\gamma-1\right)\right)\right), & |z|> \frac{1}{1-\nu_1},
\end{cases}
\end{equation*}
\begin{equation*}
S_3 \leq 
\begin{cases}
0, & |z|\leq N-1, \\
16\frac{1+|z|}{\left(|z|-1\right)^2}\left(5+\frac{6}{|z|-1}\right), & |z|> N-1,
\end{cases}
\end{equation*}
\begin{equation*}
S_4 \leq 
\begin{cases}
0, & |z|\leq \frac{1}{\nu_2-1}, \\
\frac{8}{\pi^2}\left(1+2\left(1+\nu_2\right)\right)\left(\log{\nu_2}+\frac{1}{|z|}\left(\frac{1}{2\nu_2}+2\left(\log{2}+\gamma-1\right)\right)\right), & \frac{1}{\nu_2-1}< |z|,
\end{cases}
\end{equation*}
and
\begin{equation*}
S_5 \leq 
\frac{16}{\pi^2}\left(1+2\left(1+\frac{1}{\nu_2}\right)\right)\left(1+\frac{1}{\nu_2|z|}\right)\frac{1+|z|}{\nu_2|z|\left(1+\left(\nu_2-1\right)^{2}|z|^2\right)}.
\end{equation*}
Define
\begin{equation*}
    S_{\text{small}}(z,n):=  \frac{4}{\pi^2}\frac{1+|z|}{1+\left(|z|-n\right)^2}\left(\frac{4}{n^2}+\frac{8\left(\left|z\right|+n\right)}{n^3}\right), 
\end{equation*}
and $S_{\text{large}}(z,n)$ is like $S_{\text{small}}(z,n)$, just $4/\pi^2$ is replaced by $1+2/\pi^2$.

Let us first consider the case $|z| \geq \num{10000}$. We choose $N=\num{10000}$, $\nu_1=0.995884$ and $\nu_2=1.00342$. Now the non-zero estimates for the terms $S_1,\ldots, S_5$ are decreasing with respect to $|z|$, so we can choose to use $|z|=\num{10000}$ in them. Hence, we obtain
\begin{equation*}
    \sum_{n=1}^\infty S(z,n) \leq \max_{|z| \geq \num{10000}}\left\{\sum_{n=1}^{9999} S_{\text{small}}(z,n)+S_{\text{large}}(z,\num{10000})\right\}+\sum_{i=1}^5 S_i <3.903+0.055<4.
\end{equation*}
The estimate~\eqref{eq:UpperBoundImaginaryCase2} follows.

Let us now move to the case $|z|<\num{10000}$. Moreover, we can assume that $z$ is a real number. The first term in the estimate follows from \eqref{eq:M}. Hence, we can concentrate on estimating the sum $\sum_{n=1}^\infty S(z,n)$.
We start with the case $0 \leq |z| \leq 1$. Now,
\begin{equation*}
    \sum_{n=1}^\infty S(z,n) \leq \max_{0 \leq |z|\leq 1}\left\{ S(z,1)\right\}+\sum_{n=2}^\infty \frac{2}{\pi^2}\left(\frac{2}{1+(n-1)^2}\cdot \frac{12}{n^2}+\frac{2}{1+n^2}\cdot \frac{12}{n^2}\right)<11.551+1.103<13.
\end{equation*}

Let us now consider the case $1<|z|<2$. Similarly as in the second last paragraph of the proof of Lemma \ref{lemma:m2g}, we can deduce that $S(z,n)$ is increasing in $|z|$ when $n \geq 4$ or $n\in\{1,3\}$. Hence, we have
\begin{equation*}
    \sum_{n=1}^\infty S(z,n)=\lim_{z \to 2^-} \left(S(z,1)+S(z,3)+\sum_{n=4}^\infty S(z,n)\right)+\max_{1<|z|<2} \left\{S(z,2)\right\}<20.111+3.413<24.
\end{equation*}

Finally, let us consider the case $2 \leq |z| \leq \num{10000}$. Let us now choose $N=\num{10000}$, $\nu_1=0.99999$ and $\nu_2=1.86921$. Now, $S_1=S_2=S_3$, and we can use the non-zero upper bound for $S_4$. Then, $S_4$ and $S_5$ are decreasing. Hence, we get
\begin{equation*}
    \sum_{n=1}^\infty S(z,n) \leq \max_{2 \leq |z| \leq \num{10000}} \left\{S(z,n)\right\}+S_4+S_5<7.868+7.293<16.
\end{equation*}
This finishes the proof of~\eqref{eq:UpperBoundRealCase2}.
\end{proof}

\begin{remark}
Since $\sum_{n=1}^{100} S(1.9, n)\approx 23.359>23$, $24$ is the smallest integer that can be used in the right-hand side of~\eqref{eq:UpperBoundRealCase2} using terms $S(z,n)$.
\end{remark}

\begin{remark}
Since $\sum_{n=1}^{100} S(\num{10000}, n)\approx 3.898>3$, $4$ is the smallest integer that can be used in the right-hand side of~\eqref{eq:UpperBoundImaginaryCase2} using terms $S(z,n)$.
\end{remark}

We move on to estimate the cases when $z=x\in\R$ and $|x|$ is large enough.

\begin{lemma}
\label{lemma:gAbsolute}
Let $\Delta \geq 0.8$, $|\sigma-1/2|\leq 2$ and $\sigma \neq 1/2$. For $x\in\R$ and $|x| \geq \num{10000}/\Delta$ we have 
\begin{equation*}
    \left|g_\Delta(x)\right| \leq \frac{4}{x^2}.
\end{equation*}
\end{lemma}

\begin{proof}
By~\cite[Lemma 5(i)]{CarneiroChandee} and the first estimate in~\eqref{eq:fBounds}, we have $g_\Delta(x) \leq f_\sigma(x)\leq 4/x^2$. Hence, it is sufficient to find a lower bound for $g_\Delta(x)$. Again, we estimate the function $G_\Delta(x)=g_\Delta(x/\Delta)$ instead.

    Let us now denote
    \begin{equation*}
        S(x,n):=\frac{16\left(\Delta^2+x^2\right)\sqrt{\delta\left(\xi_1(x,n)\right)\delta\left(\xi_2(x,n)\right)}}{\left(\left(n-\frac{1}{2}\right)^{2}+4\Delta^2\right)\sqrt{\left(1+\left|\xi_1(x,n)\right|^{2}\right)\left(1+\left|\xi_2(x,n)\right|^{2}\right)}}.
    \end{equation*}
    By estimate \eqref{eq:lowerForG}, the function $G_{\Delta}(x)$ can be estimated as $G_\Delta(x) \geq -\Delta^2/\left(\Delta^2+x^2\right)\sum_{n=1}^\infty S(x,n)$. We concentrate on estimating the sum of the terms $S(x,n)$.

Let $\nu_1 \in (0,0.9999)$, $\nu_2 \in (0.5,1]$ be real numbers and $N \in \Z_+$, $N \leq 9000$. Now
\begin{multline}
\label{eq:SumSiUpper}
\sum_{n>N}S(x,n) \leq \left(\sum_{N-\frac{1}{2}<n-\frac{1}{2}\leq\nu_1|x|}+\sum_{\nu_1|x|<n-\frac{1}{2}\leq|x|-1} \vphantom{\sum_{\max\left\{N-\frac{1}{2}\right\}<n-\frac{1}{2}}}\right. \\
\left.+\sum_{\max\left\{N-\frac{1}{2},|x|-1\right\}<n-\frac{1}{2}\leq|x|+1}+\sum_{|x|+1<n-\frac{1}{2}\leq 2\nu_{2}|x|}+\sum_{\max\left\{N-\frac{1}{2},2\nu_{2}|x|,|x|+1\right\}<n-\frac{1}{2}}\right)S(x,n).
\end{multline}
As before, the terms $S_1,\ldots, S_5$ denote those sums. Noting that in $S_1, S_2, S_4$ and $S_5$, we have $|\xi_i|\geq 1$ for $i=1,2$. Moreover, we also note that
\begin{equation*}
    S(x,n) \leq
    \begin{cases}
        \frac{32}{\left(\pi(1-\nu_1)\left(n-\frac{1}{2}\right)\right)^2} & \text{ in } S_1, \\
        \frac{16}{\sqrt{2}\pi^2 \nu_1^2n} & \text{ in } S_2, \\
        \frac{16\sqrt{2}\cdot 1.0004}{\pi \cdot \num{19998}} & \text{ in } S_3, \\
        \frac{32}{\pi^2\sqrt{2}(n-\frac{1}{2}+|x|)} & \text{ in } S_4, \\
        \frac{8}{\nu_2^2\pi^2\left(n-\frac{1}{2}\right)^2\left(1-\frac{1}{2\nu_2}\right)^2} &\text{ in } S_5,
    \end{cases}
\end{equation*}
since $\Delta \geq 0.8$, $|x| \geq \num{10000}$, $n-1/2 \geq |x|-1$ in $S_3$ and $\left|\xi_1-\xi_2\right|=2(n-1/2)$. Hence, as in the proof of Lemma \ref{lemma:m2g}, we can deduce
\begin{gather*}
    S_1 < \sum_{n>N} \frac{32}{\left(\pi(1-\nu_1)\left(n-\frac{1}{2}\right)\right)^2}, \quad S_2 \leq \frac{16}{\sqrt{2}\pi^2 \nu_1^2}\left(\log{\frac{1}{\nu_1}}+\frac{1}{|x|}\left(\frac{1}{2}+\frac{2}{\nu_1}\left(\log{2}+\gamma-1\right)\right)\right), \\
    S_3 \leq 2\cdot \frac{16\sqrt{2}\cdot 1.0004}{\pi \cdot \num{19998}}, \quad S_4 < \frac{32}{\pi^2\sqrt{2}}\left(\log{\left(\frac{2\nu_2+1}{2}\right)}+\frac{1}{|x|}\left(\frac{1}{2(\nu_2+1)}+\log{2}+\gamma-1\right)\right),  \\
    S_5 \leq \frac{8\left(1+2\nu_2 |x|\right)}{x^2 \nu_2^2 \pi^2\left(2\nu_2-1\right)}.
\end{gather*}

Let us now choose $N=500$, $\nu_1=0.83181$ and $\nu_2=0.508444$. Since for all $n \leq N$, we have $n-1/2 < |x|$, the term $\left(\Delta^2+x^2\right)/\left((n-1/2)^2+4\Delta^2\right)$ is decreasing with respect to $\Delta$ when $n \leq N$. Hence, we can choose $\Delta=0.8$ in this case. Note also that in the case $n \leq N$ we have 
\begin{equation}
\label{eq:Num}
    \sqrt{\left(1+\left|\xi_1(x,n)\right|^{2}\right)\left(1+\left|\xi_2(x,n)\right|^{2}\right)} \geq \sqrt{\left(1+\left(|x|-499.5\right)^2\right)\left(1+x^2\right)}.
\end{equation}
When we divide the term $x^2=0.8^2$ by the right-hand side of \eqref{eq:Num}, we get a function that is decreasing for all $|x|\geq \num{10000}$. At $|x|=\num{10000}$, the value is $<1.053$. Since also the term $S_2, S_4$ and $S_5$ are decreasing with respect to $|x|$, we can set $|x|=\num{10000}$ and obtain
\begin{equation*}
    0 \leq \sum_{n=1}^\infty S(x,n) \leq \sum_{n=1}^{3000} \frac{32\cdot1.053}{\pi^2\left(\left(n-\frac{1}{2}\right)^2+4\cdot 0.8^2\right)} +\sum_{i=1}^5 S_i<3.345+0.574<4.
\end{equation*}
The wanted upper bound follows.
\end{proof}

\begin{lemma}
\label{lemma:mUpperReal}
Let $\Delta \geq 0.8$, $|\sigma-1/2|\leq 2$ and $\sigma \neq 1/2$. For every $x\in\R$ and $|x| \geq 17$, we have 
\begin{equation*}
    \left|m_\Delta(x)\right| \leq \left(\frac{4}{\pi^2}\left(\frac{1}{1+\Delta^2x^2}+\frac{1}{\Delta^2}\right)\log{\frac{2}{\left|\sigma-\frac{1}{2}\right|}}+15\right)\dfrac{1}{1+x^2}.
\end{equation*}
\end{lemma}

\begin{proof}
We follow the same ideas as in the proof of Lemma \ref{lemma:gAbsolute}. First, we note that by \cite[Lemma 8(i)]{CarneiroChandee}, we have $0\leq f_\sigma (x) \leq m_\Delta (x)$, and hence it is sufficient to find an upper bound for $m_\Delta(x)$. We estimate the function $M_\Delta(x)=m_\Delta\left(x/\Delta\right)$. By \eqref{eq:upperForMreal}, we have
\begin{equation*}
\frac{\Delta^2+x^2}{\Delta^2}M_\Delta(x) \leq \left(\frac{1}{1+x^2}+\frac{1}{\Delta^2}\right)2\delta{(x)}\log{\frac{2}{\left|\sigma-\frac{1}{2}\right|}} + \sum_{n=1}^\infty S(x,n),
\end{equation*}
where
\begin{multline*}
    S(x,n):=\frac{4\left(\Delta^2+x^2\right)}{n^2}\left(\frac{\delta\left(\xi_3(x,n)\right)}{1+\left|\xi_3(x,n)\right|^2}+\frac{\delta\left(\xi_4(x,n)\right)}{1+\left|\xi_4(x,n)\right|^2}\right)\\
+\frac{16\left(\Delta^2+x^2\right)\sqrt{\delta\left(\xi_3(x,n)\right)\delta\left(\xi_4(x,n)\right)}}{\left(n^{2}+4\Delta^2\right)\sqrt{\left(1+\left|\xi_3(x,n)\right|^{2}\right)\left(1+\left|\xi_4(x,n)\right|^{2}\right)}}. 
\end{multline*}
Let again $0.7<\nu_1 \leq 0.9999$, $\nu_2 \in (0.5,1]$ be real numbers and $N \in \Z_+$. We divide the sum to different parts similarly as in \eqref{eq:SumSiUpper} but instead of $N-1/2$ and $n-1/2$ we use $N$ and $n$.

We consider the terms similarly as in the proof of previous lemma. In addition, in $S_3$ we recognize that the maximum of
$$
\frac{1}{(1 + y^2) (y/x + 1)^2} + \frac{1}{(1 + (y + 1)^2) (y/x + 1/x + 1)^2},
$$
where $y:=n-x$ if $x> 0$ and $y:=n-x$ if $x<0$, $y\in [-1,0]$, is at most $1.61$.
Since $|x| \geq 17\Delta$, we also have $x^2+\Delta^2\leq 290 x^2/289$. Thus
\begin{align*}
     & S_1 \leq \frac{2320}{289\pi^2}\left(\frac{1}{(1-\nu_1)^2}+\frac{4}{1-\nu_1}+1\right) \sum_{n>N} \frac{1}{n^2}  
\end{align*}
\begin{align*}
       &S_2 \leq \frac{2320}{289
    \pi^2 \nu_1^2} \left(0.5 + \frac{\pi}{4} - \frac{1}{\nu_1 x} + \frac{1}{3\nu_1^3 x^3} +  \frac{1}{1 + (1 + \nu_1)^2 x^2} - \frac{1}{2 x - 1} + \frac{1}{3 (2 x - 1)^3} + 
     \frac{1}{(1 + \nu_1) x}\right)  \\
  &\quad+\frac{9280}{289 \pi^2 \sqrt{1 + x^2}\nu_1^2} \left(\frac{1}{1 + x^2 (1 - \nu_1)^2} - 
     \frac{\sqrt{x^2 + 1}}{x + 1} + \frac{(x + 1)^{3/2}}{3 (x + 1)^3} + 
     \frac{\sqrt{x^2 + 1}}{\nu_1 x}\right),
\end{align*}
\begin{equation*}
     S_3 \leq \frac{1160}{289}\left(\left(\frac{x}{x - 1}\right)^2 \left(1.61 + \frac{2}{(1 + (2 x - 1)^2) \pi^2}+\frac{8\sqrt{2}}{\pi \sqrt{1 + (2 x - 1)^2}}\right) + \frac{2}{\pi^2 (1 + 4 x^2)} 
    \right),
\end{equation*}
\begin{align*}
    & S_4 < \frac{2320}{289\pi^2} \left(0.5 + \frac{\pi}{4} - \frac{1}{x (2 \nu_2 - 1)} + \frac{1}{3 x^3 (2 \nu_2 - 1)^3)} + \frac{1}{1 + (2 x + 1)^2} - \frac{1}{x (2 \nu_2 + 1)} + \frac{1}{2 x + 1}
    + \frac{1}{3 (2 x + 1)^3}\right) \\
 &\quad+\frac{9280}{289 \pi^2 \sqrt{x^2 - 1}} \left(-\frac{\sqrt{x^2 - 1}}{x + 1} + \frac{(x^2 - 1)^{3/2}}{3 (x + 1)^3} + \frac{\sqrt{x^2 - 1}}{2 \nu_2 }\right)
\end{align*}
and
\begin{equation*}
    S_5 \leq \frac{1160}{289 \pi^2 \nu_2 x} \left(1 + \frac{1}{2 \nu_2 x}\right)\left(\frac{1}{(2 \nu_2 - 1)^2} + \frac{1}{(2 \nu_2 + 1)^2} + \frac{2}{4\nu_2^2 - 1 }\right).
\end{equation*}

Let us now choose $N=100$, $\nu_1=0.820256$ and $\nu_2=0.999656$. We note that the terms
\begin{equation*}
   \frac{x^2}{1+\left|\xi_3(x,n)\right|^2}+\frac{x^2}{1+\left|\xi_4(x,n)\right|^2} \quad\text{and} \quad \frac{x^2}{\sqrt{\left(1+\left|\xi_3(x,n)\right|^{2}\right)\left(1+\left|\xi_4(x,n)\right|^{2}\right)}}
\end{equation*}
are decreasing with respect to $|x|$ when $n \leq N$. Hence, we choose $|x|=\num{10000}$. Noting that the terms $S_1,\ldots, S_5$ are also decreasing with respect to $|x|$, we choose that $|x|=\num{10000}$ in them too. Now, applying the facts $\Delta\geq 0.8$ and $|x| \geq 17\Delta$, we have
\begin{multline*}
    \sum_{n=1}^\infty S(x,n) < \frac{2320\cdot10000^2}{289\pi^2}\sum_{n=1}^{100} \left(\frac{1}{\left(1+\left|\xi_3(10000,n)\right|^2\right) n^2}+\frac{1}{\left(1+\left|\xi_4(10000,n)\right|^2\right) n^2} \right. \\
    \left.\quad+\frac{4}{(n^2+4\cdot0.8^2) \sqrt{\left(1+\left|\xi_3(1000,n)\right|^{2}\right)\left(1+\left|\xi_4(10000,n)\right|^{2}\right)}}\right)+\sum_{i=1}^5 S_i<5.18629+9.50235<15.
\end{multline*}
The proof of Lemma~\ref{lemma:mUpperReal} is thus complete.
\end{proof}

\section{Various sums over prime numbers}
\label{sec:primes}
In this section, we estimate the term $\mathcal{I}_4$ given in~\eqref{def:I3I4}. First, we derive estimates for $\widehat{g}_{\Delta}$ and $\widehat{m}_{\Delta}$, and then we use various upper bounds on the sum of the modulus of the generalized von Mangoldt function (Lemma~\ref{lemma:MangoldSum}) to estimate $\mathcal{I}_4$.

\subsection{On functions $\widehat{g}_\Delta$ and $\widehat{m}_\Delta$ and general estimates}
First, we provide formulas for the functions $\widehat{g}_\Delta$ and $\widehat{m}_\Delta$ to be used later in this section.

\begin{lemma}
\label{lemma:gmSums}
Let $\xi$ and $\Delta>0$ be real numbers, and $1/2<\sigma\leq 1$. For $0 \leq |\xi| \leq \Delta$ we have
\begin{multline}
\label{def:ghatSum}
\widehat{g}_\Delta(\xi)=\sum_{k=0}^\infty (-1)^k \left(\frac{k+1}{|\xi|+k\Delta}\left(e^{-2\pi(|\xi|+k\Delta)(\sigma-1/2)}-e^{-4\pi(|\xi|+k\Delta)}\right) \right. \\
\left.-\frac{k+1}{\Delta(k+2)-|\xi|}\left(e^{2\pi(|\xi|-(k+2)\Delta)(\sigma-1/2)}-e^{4\pi(|\xi|-(k+2)\Delta)}\right)\right)
\end{multline}
and
\begin{multline}
\label{def:mhatSum}
\widehat{m}_\Delta(\xi)=\sum_{k=0}^\infty \left(\frac{k+1}{|\xi|+k\Delta}\left(e^{-2\pi(|\xi|+k\Delta)(\sigma-1/2)}-e^{-4\pi(|\xi|+k\Delta)}\right) \right. \\
\left.-\frac{k+1}{\Delta(k+2)-|\xi|}\left(e^{2\pi(|\xi|-(k+2)\Delta)(\sigma-1/2)}-e^{4\pi(|\xi|-(k+2)\Delta)}\right)\right).
\end{multline}
\end{lemma}

\begin{proof}
This is contained in~\cite[Lemmas~5 and~8]{CarneiroChandee}.
\end{proof}

\begin{lemma} 
\label{lemma:mghat0}
Let $\Delta>0$ and $1/2<\sigma\leq 1$. Then
    \begin{equation*}
        \frac{1}{\pi}\widehat{g}_{\Delta}(0)=2 \left(\frac{5}{2}-\sigma-\frac{1}{\pi\Delta}\log{\left(\frac{1+e^{-(2\sigma-1)\pi\Delta}}{1+e^{-4\pi\Delta}}\right)}\right)
    \end{equation*}
    and
    \begin{equation*}
        \frac{1}{\pi}\widehat{m}_{\Delta}(0)=2 \left(\frac{5}{2}-\sigma-\frac{1}{\pi\Delta}\log{\left(\frac{1-e^{-(2\sigma-1)\pi\Delta}}{1-e^{-4\pi\Delta}}\right)}\right).
    \end{equation*}
\end{lemma}

\begin{proof}
This is~\cite[Equations~(2.7) and~(3.4)]{CarneiroChandee}.
\end{proof}

\begin{lemma}
\label{lemma:gmhatsUpper}
Let $2\leq n\leq x=e^{2\pi\Delta}$ and $1/2<\sigma\leq 1$. Then
\begin{multline}
\label{eq:ghatbound}
\frac{1}{2\pi}\left|\widehat{g}_{\Delta}\left(\frac{\log{n}}{2\pi}\right)\right| \leq 
\frac{1}{n^{\sigma-1/2}\log{n}}-\frac{n^{\sigma-1/2}}{(2\log{x}-\log{n})x^{2\sigma-1}} \\ 
- \frac{1}{\left(x^{\sigma-1/2}+1\right)\log{x}}\left(\frac{1}{n^{\sigma-1/2}}-\frac{n^{\sigma-1/2}}{x^{2\sigma-1}}\right) + \frac{1}{n^{2}\log{n}} + \frac{3}{n^{2}\left(x^2-1\right)\log{x}}
\end{multline}
and
\begin{multline}
\label{eq:mhatbound}
\frac{1}{2\pi}\left|\widehat{m}_{\Delta}\left(\frac{\log{n}}{2\pi}\right)\right| \leq \frac{1}{n^{\sigma-1/2}\log{n}}-\frac{n^{\sigma-1/2}}{\left(2\log{x}-\log{n}\right)x^{2\sigma-1}}
+ \frac{2\left(\log{x}-\log{n}\right)}{n^{\sigma-1/2}\log^{2}{x}}\log{\frac{1}{1-x^{1/2-\sigma}}} \\ 
+ \frac{1-\left(\frac{n}{x}\right)^{2\sigma-1}}{x^{\sigma-1/2}\left(1-x^{1/2-\sigma}\right)n^{\sigma-1/2}\log{x}} + \frac{1}{n^2\log{n}} + \frac{3}{n^{2}\left(x^2-1\right)\log{x}}.
\end{multline}
\end{lemma}

\begin{proof}
By~\eqref{def:ghatSum},
\begin{multline}
\label{eq:ghatdef}
\widehat{g}_{\Delta}\left(\frac{\log{n}}{2\pi}\right) = \sum_{k=0}^{\infty} (-1)^{k}\left(\frac{2\pi(k+1)}{\log{n}+2\pi k\Delta}\left(e^{-\left(\log{n}+2\pi k\Delta\right)\left(\sigma-1/2\right)}-e^{-\left(2\log{n}+4\pi k\Delta\right)}\right)\right. \\
\left.-\frac{2\pi(k+1)}{2\pi(k+2)\Delta-\log{n}}\left(e^{\left(\log{n}-2\pi(k+2)\Delta\right)\left(\sigma-1/2\right)}-e^{2\log{n}-4\pi(k+2)\Delta}\right)\right).
\end{multline}
By~\eqref{def:mhatSum}, almost the same equality as~\eqref{eq:ghatdef} holds also for $\widehat{m}_{\Delta}\left(\left(\log{n}\right)/(2\pi)\right)$, with only difference that $(-1)^{k}$ should be omitted. Noting that $x=e^{2\pi\Delta}$, we obtain
\begin{flalign}
\frac{1}{2\pi}\widehat{g}_{\Delta}\left(\frac{\log{n}}{2\pi}\right) &= \sum_{k=0}^{\infty} (-1)^{k}\left(\frac{k+1}{\left(k\log{x}+\log{n}\right)\left(nx^{k}\right)^{\sigma-1/2}}-\frac{(k+1)n^{\sigma-1/2}}{\left((k+2)\log{x}-\log{n}\right)\left(x^{k+2}\right)^{\sigma-1/2}}\right)
\nonumber \\
&-\frac{1}{n^{2}\log{n}} - \sum_{k=1}^{\infty} (-1)^{k}\frac{k+1}{\log{n}+2\pi k\Delta}e^{-\left(2\log{n}+4\pi k\Delta\right)} \nonumber \\
&+\frac{e^{2\log{n}-8\pi\Delta}}{4\pi\Delta-\log{n}}+\sum_{k=1}^{\infty} (-1)^{k}\frac{k+1}{2\pi(k+2)\Delta-\log{n}}e^{2\log{n}-4\pi(k+2)\Delta} \label{eq:ghataux}
\end{flalign}
after making obvious rearrangements of the series in~\eqref{eq:ghatdef}; note that these are justified because the series is absolutely convergent. After omitting $(-1)^k$ from the above, the resulting equality holds also for $\widehat{m}_{\Delta}$. By ~\cite[Lemma 6]{CarneiroChandee}, the first sum in \eqref{eq:ghataux} is non-negative. Hence, combining with~\cite[Estimate 2.16]{CarneiroChandee}, we have
\begin{multline}
\label{eq:ghataux1}
0\leq \sum_{k=0}^{\infty} (-1)^{k}\left(\frac{k+1}{\left(k\log{x}+\log{n}\right)\left(nx^{k}\right)^{\sigma-1/2}}-\frac{(k+1)n^{\sigma-1/2}}{\left((k+2)\log{x}-\log{n}\right)\left(x^{k+2}\right)^{\sigma-1/2}}\right)  \\
\leq \frac{1}{n^{\sigma-1/2}\log{n}}-\frac{n^{\sigma-1/2}}{(2\log{x}-\log{n})x^{2\sigma-1}}  
- \frac{1}{\left(x^{\sigma-1/2}+1\right)\log{x}}\left(\frac{1}{n^{\sigma-1/2}}-\frac{n^{\sigma-1/2}}{x^{2\sigma-1}}\right).
\end{multline}
Following~\cite[pp.~253--254]{BF2023}, we also have
\begin{flalign}
\label{eq:ghataux2}
0 &\leq \sum_{k=0}^{\infty} \left(\frac{k+1}{\left(k\log{x}+\log{n}\right)\left(nx^{k}\right)^{\sigma-1/2}}-\frac{(k+1)n^{\sigma-1/2}}{\left((k+2)\log{x}-\log{n}\right)\left(x^{k+2}\right)^{\sigma-1/2}}\right) \nonumber \\ 
&= \frac{1}{n^{\sigma-1/2}\log{n}}-\frac{n^{\sigma-1/2}}{(2\log{x}-\log{n})x^{2\sigma-1}} \nonumber \\ 
&+ \sum_{k=1}^{\infty}\frac{(k+1)x^{k(1/2-\sigma)}}{n^{\sigma-1/2}}\left(\left(\frac{1}{k\log{x}+\log{n}}-\frac{1}{(k+2)\log{x}-\log{n}}\right) + \frac{1-\left(\frac{n}{x}\right)^{2\sigma-1}}{(k+2)\log{x}-\log{n}}\right) \nonumber \\
&\leq \frac{1}{n^{\sigma-1/2}\log{n}}-\frac{n^{\sigma-1/2}}{\left(2\log{x}-\log{n}\right)x^{2\sigma-1}} + \frac{2\left(\log{x}-\log{n}\right)}{n^{\sigma-1/2}\log^{2}{x}}\sum_{k=1}^{\infty}\frac{x^{k\left(\frac{1}{2}-\sigma\right)}}{k} + \frac{1-\left(\frac{n}{x}\right)^{2\sigma-1}}{n^{\sigma-1/2}\log{x}}\sum_{k=1}^{\infty}x^{k\left(\frac{1}{2}-\sigma\right)} \nonumber \\
&= \frac{1}{n^{\sigma-\frac{1}{2}}\log{n}}-\frac{n^{\sigma-\frac{1}{2}}}{\left(2\log{x}-\log{n}\right)x^{2\sigma-1}} 
+ \frac{2\left(\log{x}-\log{n}\right)}{n^{\sigma-\frac{1}{2}}\log^{2}{x}}\log{\frac{1}{1-x^{\frac{1}{2}-\sigma}}} + \frac{1-\left(\frac{n}{x}\right)^{2\sigma-1}}{\left(1-x^{\frac{1}{2}-\sigma}\right)(nx)^{\sigma-\frac{1}{2}}\log{x}}.
\end{flalign}
Furthermore, because
\[
\frac{k+1}{\log{n}+2\pi k\Delta} = \frac{1}{2\pi\Delta}\left(\frac{k+1}{k+\frac{\log{n}}{2\pi\Delta}}\right) \leq \frac{1}{\pi\Delta}
\]
for $k\geq 1$, and
\[
\frac{k+1}{2\pi(k+2)\Delta-\log{n}} = \frac{1}{2\pi\Delta}\left(\frac{k+1}{k+2-\frac{\log{n}}{2\pi\Delta}}\right) \leq \frac{1}{2\pi\Delta}
\]
for $k\geq 0$ and $n^2\leq e^{4\pi\Delta}$, it follows that
\begin{equation}
\label{eq:ghataux3}
\left|\sum_{k=1}^{\infty} (-1)^{k}\left(-\frac{(k+1)e^{-\left(2\log{n}+4\pi k\Delta\right)}}{\log{n}+2\pi k\Delta}+\frac{(k+1)e^{2\log{n}-4\pi(k+2)\Delta}}{2\pi(k+2)\Delta-\log{n}}\right)\right| \leq \frac{3}{n^{2}2\pi\Delta\left(e^{4\pi\Delta}-1\right)}.
\end{equation}
Obviously, inequality~\eqref{eq:ghataux3} holds also when omitting $(-1)^{k}$. Finally, again from the fact that $n^2\leq e^{4\pi\Delta}$ we can deduce that
\begin{equation}
\label{eq:ghataux4}
\left|-\frac{1}{n^2\log{n}}+\frac{e^{2\log{n}-8\pi\Delta}}{4\pi\Delta-\log{n}}\right| \leq \frac{1}{n^2\log{n}}.
\end{equation}
The proof is complete after employing~\eqref{eq:ghataux1} (respectively~\eqref{eq:ghataux2}),~\eqref{eq:ghataux3} and~\eqref{eq:ghataux4} on~\eqref{eq:ghataux}
\end{proof}

The next two results are general estimates that are used when we are estimating sums over primes. 

\begin{lemma}
\label{lemma:secondIntegral}
Let $x\geq 2$ and $\sigma<1$. Then
\[
\left|\int_{2}^{x}\frac{\dif{u}}{u^{\sigma}\left(2\log{x}-\log{u}\right)}-\frac{x^{1-\sigma}}{(1-\sigma)\log{x}}\right| \leq \left(1+\frac{(1-\sigma)\log^{2}{x}}{2\log{x}-\log{2}}\left(\frac{2}{x}\right)^{1-\sigma}\right)\frac{x^{1-\sigma}}{(1-\sigma)^{2}\log^{2}{x}}.
\]
\end{lemma}

\begin{proof}
Substituting $u \mapsto x^2/u$ and integrating by parts, we see that
\begin{flalign*}
\int_{2}^{x}\frac{\dif{u}}{u^{\sigma}\left(2\log{x}-\log{u}\right)}  &= x^{2(1-\sigma)}\int_{x}^{x^2/2}\frac{\dif{u}}{u^{2-\sigma}\log{u}} \\
&= \frac{x^{1-\sigma}}{(1-\sigma)\log{x}} - \frac{2^{1-\sigma}}{(1-\sigma)\left(2\log{x}-\log{2}\right)} - \frac{x^{2(1-\sigma)}}{1-\sigma}\int_{x}^{x^2/2}\frac{\dif{u}}{u^{2-\sigma}\log^{2}{u}}.
\end{flalign*}
Also, we have
\[
\int_{x}^{x^2/2}\frac{\dif{u}}{u^{2-\sigma}\log^{2}{u}} \leq \frac{1}{\log^{2}{x}}\int_{x}^{x^2/2}\frac{\dif{u}}{u^{2-\sigma}} \leq \frac{x^{\sigma-1}}{(1-\sigma)\log^{2}{x}}.
\]
Thus, the stated inequality follows.
\end{proof}

As in~\cite{PalojarviSimonic}, we define $\widetilde{\psi}_{\cL}(x) \de \sum_{n \leq x} \left|\Lambda_\cL(n)\right|$ and $\psi(x)\de\psi_{\zeta}(x)$.

\begin{lemma}
\label{lemma:MangoldSum}
    Assume RH and let $x \geq 2$. Then
    \begin{equation*}
        \widetilde{\psi}_{\cL}(x) \leq 
        \begin{cases}
        \mathcal{C}_{\cL}^{R}(\varepsilon)x^{1+\varepsilon} + O\left(\left(\mathcal{C}_{\cL}^{R}(\varepsilon)+\mathcal{C}_{\cL}^{E}\right)x^{\frac{1}{2}+\max\left\{\theta,\varepsilon\right\}}\log^{2}{x}\right), & \text{for}\; \varepsilon \in (0,1/2), \\
        \sqrt{\mathcal{C}_{\cL}^{P_1}(x)}x + O\left(\sqrt{\mathcal{C}_{\cL}^{P_1}(x)+\mathcal{C}_{\cL}^{P_2}}\frac{x}{\sqrt{\log{x}}}\right) + O\left(\mathcal{C}_{\cL}^{E} x^{\frac{1}{2}+\theta}\log^{2}{x}\right), &\text{under Conjecture \ref{conj:SelbergVariant}}, \\
        \widehat{\mathcal{C}_{\cL}^{P_1}}(x)x + \widehat{\mathcal{C}_{\cL}^{P_2}}\frac{x}{\log{x}} + O\left(\mathcal{C}_{\cL}^{E} x^{\frac{1}{2}+\theta}\log^{2}{x}\right), & \text{under Conjecture \ref{conj:SelbergVariant2}}, \\
            m\left(x+\frac{1}{8\pi}\sqrt{x}\log^2{x}\right), &\text{if}\; \cL \in \cSP.
        \end{cases}
    \end{equation*}
    In addition to the above, we also have $ \left|\psi(x) -x\right|\leq 2\sqrt{x}\log^2{x}$. 
\end{lemma}    

\begin{proof}
    The first three inequalities are estimates~(5.1),~(5.2) and~(5.3) in~\cite{PalojarviSimonic}. The remaining two inequalities follow from the facts that $\left|\Lambda_\cL(n) \right| \leq m \Lambda(n)$ when $\cL \in \cSP$ and $\left|\psi(x)-x\right| \leq \frac{1}{8\pi}\sqrt{x}\log^2{x}$ if $x \geq 74$, see~\cite[Estimate~(6.2)]{SchoenfeldSharperRH}, and $\psi(x) \leq x+\frac{1}{8\pi}\sqrt{x}\log^2{x}$ for $2 \leq x \leq 74$. For the last inequality, we verified it for $2\leq x \leq 74$ by computer.
\end{proof}

\subsection{Sums over primes -- asymptotic results}

In this section we derive asymptotic estimates for $\mathcal{I}_4$ without additional assumptions (other than RH) and under Conjectures~\ref{conj:SelbergVariant} and~\ref{conj:SelbergVariant2}. Our first lemma describes what terms are we going to estimate. Define
\begin{equation}
\label{def:Asigma}
    A_f(x,\sigma):=\int_2^x f'(u)\left(\frac{1}{u^{\sigma}\log{u}}-\frac{u^{\sigma-1}}{\left(2\log{x}-\log{u}\right)x^{2\sigma-1}}\right)\dif{u},
\end{equation}
\begin{equation}
\label{def:Bsigma}
    B_f(x,\sigma):=\frac{1}{\left(x^{\sigma-\frac{1}{2}}+1\right)\log{x}} \int_2^x f'(u) \left(\frac{1}{u^\sigma}-\frac{u^{\sigma-1}}{x^{2\sigma-1}}\right)\dif{u},
\end{equation}
\begin{equation}
\label{def:Dsigma}
D_f(x,\sigma) \de \frac{2}{\log^{2}{x}}\int_{2}^{x}f'(u)\frac{1}{u^{\sigma}}\log{\frac{x}{u}}\dif{u}
\end{equation}
and
\begin{equation}
\label{def:Esigma}
E_{f}(x,\sigma) \de \frac{1}{\left(x^{\sigma-\frac{1}{2}}-1\right)\log{x}}\int_{2}^{x}f'(u)\frac{1}{u^{\sigma}}\left(1-\left(\frac{u}{x}\right)^{2\sigma-1}\right)\dif{u}
\end{equation}
for some $f\in\mathcal{C}^{1}((1,\infty))$. 

\begin{lemma}
\label{lemma:ABC123}
    Let $x=e^{2\pi\Delta}\geq 2$ and $1/2<\sigma \leq 1$. Assume that $\widetilde{\psi}_{\cL}(u)\leq f(u)+k(u)$ for $u \in [2,x]$, where $f\in\mathcal{C}^{1}((1,\infty))$ and $k\in\mathcal{C}([2,\infty))$ are non-negative. Then
    \begin{flalign*}
         \frac{1}{2\pi}\sum_{n=2}^{\infty}\frac{\left|\Lambda_{\cL}(n)\right|}{\sqrt{n}}\left|\widehat{g}_\Delta\left(\frac{\log{n}}{2\pi}\right)\right| &\leq A_f(x,\sigma)-B_f(x,\sigma) \\
         &+O\left(f(2)\right)+O\left(\int_2^x \frac{k(u)\dif{u}}{u^{\sigma+1}\log{u}}\right)
         +O\left(\int_2^x \frac{\left|f'(u)\right|\dif{u}}{u^{2.5}\log{u}}\right)+O\left(\frac{k(x)}{x^{2.5}\log{x}}\right)
    \end{flalign*}
    and
    \begin{multline*}
         \frac{1}{2\pi}\sum_{n=2}^{\infty}\frac{\left|\Lambda_{\cL}(n)\right|}{\sqrt{n}}\left|\widehat{m}_\Delta\left(\frac{\log{n}}{2\pi}\right)\right| \leq A_{f}(x,\sigma) + D_{f}(x,\sigma)\log{\frac{1}{1-x^{\frac{1}{2}-\sigma}}} + E_{f}(x,\sigma) \\
         +O\left(\left(1+\log{\frac{1}{1-x^{1/2-\sigma}}}\right)\left(f(2)+\int_{2}^{x}\frac{k(u)\dif{u}}{u^{1+\sigma}\log{u}}\right)\right)
         +O\left(\int_2^x \frac{\left|f'(u)\right|\dif{u}}{u^{2.5}\log{u}}\right)+O\left(\frac{k(x)}{x^{2.5}\log{x}}\right)
    \end{multline*}
    as $x \to \infty$. Here, $A_{f}(x,\sigma)$, $B_{f}(x,\sigma)$, $D_{f}(x,\sigma)$ and $E_{f}(x,\sigma)$ are defined by~\eqref{def:Asigma}--\eqref{def:Esigma}, respectively.
\end{lemma}

\begin{proof}
By partial summation,
\begin{equation}
\label{eq:pslambda}
\sum_{n\leq x} \left|\Lambda_{\cL}(n)\right|\phi{(n)} = \int_{2}^{x}f'(u)\phi(u)\dif{u} + \left(\widetilde{\psi}_{\cL}(x)-f(x)\right)\phi(x) + f(2)\phi(2) - \int_{2}^{x}\left(\widetilde{\psi}_{\cL}(u)-f(u)\right)\phi'(u)\dif{u}
\end{equation}
for some $\phi\in\mathcal{C}^{1}((1,\infty))$. Since by~\cite[Lemmas~5~(ii) and~8~(ii)]{CarneiroChandee} the functions $\widehat{g}_\Delta$ and $\widehat{m}_\Delta$ are supported on $[-\Delta, \Delta]$, it is sufficient to compute the sums from Lemma~\ref{lemma:ABC123} up to $x$. Due to Lemma~\ref{lemma:gmhatsUpper}, we are using~\eqref{eq:pslambda} for $\phi(n)$ that is given as a multiple of the function on the right-hand side of~\eqref{eq:ghatbound} (resp.~\eqref{eq:mhatbound}) with $1/\sqrt{n}$. Note that then 
\[
0\leq-\phi'(u)\ll \frac{1}{u^{1+\sigma}\log{u}}
\]
in the first case, while
\[
0\leq-\phi'(u)\ll \frac{1}{u^{1+\sigma}\log{u}}\left(1+\log{\frac{1}{1-x^{1/2-\sigma}}}\right)
\]
in the second case. Moreover, $\phi(2)\ll 1$ in the first case and $\phi(2)\ll 1-\log{\left(1-x^{1/2-\sigma}\right)}$ in the second case, while
\[
0 \leq \phi(x) \ll \frac{1}{x^{5/2}\log{x}}
\]
in both cases. The stated inequalities now easily follow from~\eqref{eq:pslambda}.
\end{proof}

Using the previous lemma, we can estimate the term $\mathcal{I}_4$ in the most general case, when $\cL \in \cS$.
\begin{lemma}
\label{lemma:I4general}
Assume RH. Let $\cL \in \cS$, $1/2<\sigma \leq 1$, $\varepsilon \in (0,1/2)$ and $x=e^{2\pi \Delta}\geq 2$. Then
\begin{multline*}
    \frac{1}{2\pi}\sum_{n=2}^{\infty}\frac{\left|\Lambda_{\cL}(n)\right|}{\sqrt{n}}\left|\widehat{g}_\Delta\left(\frac{\log{n}}{2\pi}\right)\right| \leq A_1(\mathcal{C}_{\cL}^{R}(\varepsilon), \varepsilon,\sigma)\frac{x^{1-\sigma+\varepsilon}}{\log{x}}\left(\frac{x^{\sigma-1/2}}{x^{\sigma-1/2}+1}\right) +(1+\varepsilon)\mathcal{C}_{\cL}^{R}(\varepsilon)\log\log{x} \\
        +O\left(\frac{\mathcal{C}_{\cL}^{R}(\varepsilon)x^{1-\sigma+\varepsilon}}{(1-\sigma+\varepsilon)^2\log^2{x}}\right)+O\left(A_2\left(\mathcal{C}_{\cL}^{R}(\varepsilon)+\mathcal{C}_{\cL}^{E},\varepsilon, \theta, \sigma, x\right)\right)
    \end{multline*}
    and
    \begin{multline*}
        \frac{1}{2\pi}\sum_{n=2}^{\infty}\frac{\left|\Lambda_{\cL}(n)\right|}{\sqrt{n}}\left|\widehat{m}_\Delta\left(\frac{\log{n}}{2\pi}\right)\right| \leq 
        A_1(\mathcal{C}_{\cL}^{R}(\varepsilon), \varepsilon,\sigma)\frac{x^{1-\sigma+\varepsilon}}{\log{x}} + (1+\varepsilon)\mathcal{C}_{\cL}^{R}(\varepsilon)\log\log{x} \\
        +O\left(\left(1+\log{\frac{1}{1-x^{1/2-\sigma}}}\right)\left(\frac{\mathcal{C}_{\cL}^{R}(\varepsilon)x^{1-\sigma+\varepsilon}}{(1-\sigma+\varepsilon)^2\log^2{x}}+A_2\left(\mathcal{C}_{\cL}^{R}(\varepsilon)+\mathcal{C}_{\cL}^{E},\varepsilon, \theta, \sigma, x\right)\right)\right)
    \end{multline*}
    as $x \to \infty$. Here, $\mathcal{C}_{\cL}^{R}(\varepsilon)$ and $\mathcal{C}_{\cL}^{E}$ are defined as in~\eqref{def:CER}, $A_1\left(a,\varepsilon,\sigma\right)$ as in~\eqref{def:A1}, and $A_2\left(a,\varepsilon,\theta,\sigma,x\right)$ as in~\eqref{def:Ai}.
\end{lemma}

\begin{proof}
Let us set $f(u)=\mathcal{C}_{\cL}^{R}(\varepsilon)u^{1+\varepsilon}$ and $k(u)=O(1)\cdot\left(\mathcal{C}_{\cL}^{R}(\varepsilon)+\mathcal{C}_{\cL}^{E}\right)u^{\frac{1}{2}+\max\left\{\theta,\varepsilon\right\}}\log^{2}{u}$. The first inequality now follows by Lemmas~\ref{lemma:secondIntegral},~\ref{lemma:MangoldSum} and~\ref{lemma:ABC123}, as well as by~\cite[Estimates~(5.10) and~(5.13)]{PalojarviSimonic}, the fact $1/(u^{2.5-\varepsilon}\log{u})\ll u^{-2}$ for $u\geq2$, and by having $\mathcal{C}_{\cL}^{R}(\varepsilon) \leq \mathcal{C}_{\cL}^{R}(\varepsilon)+\mathcal{C}_{\cL}^{E} \ll A_2\left(\mathcal{C}_{\cL}^{R}(\varepsilon)+\mathcal{C}_{\cL}^{E},\varepsilon,\theta,\sigma,x\right)$. The second inequality follows similarly, just that we also use
\[
D_{f}(x,\sigma)\log{\frac{1}{1-x^{\frac{1}{2}-\sigma}}} + E_{f}(x,\sigma) \ll \left(1+\log{\frac{1}{1-x^{\frac{1}{2}-\sigma}}}\right)\frac{\mathcal{C}_{\cL}^{R}(\varepsilon)x^{1-\sigma+\varepsilon}}{\left(1-\sigma+\varepsilon\right)^{2}\log^{2}{x}}.
\]
This estimate was derived by straightforward integration and also by using 
\[
x^{\sigma-1/2}-1\geq\left(\sigma-\frac{1}{2}\right)\log{x},
\]
an inequality that follows by the mean-value theorem.
\end{proof}

We obtain a very similar result also when we assume Conjecture~\ref{conj:SelbergVariant} or Conjecture~\ref{conj:SelbergVariant2}.

\begin{lemma}
\label{lemma:SumsPrimesUnderConjectures}
Assume RH. Let $\cL \in \cS$, $1/2<\sigma < 1$ and $x=e^{2\pi \Delta}\geq 2$. Assume also Conjecture~\ref{conj:SelbergVariant} or Conjecture~\ref{conj:SelbergVariant2}. Then
\begin{multline*}
\frac{1}{2\pi}\sum_{n=2}^{\infty}\frac{\left|\Lambda_{\cL}(n)\right|}{\sqrt{n}}\left|\widehat{g}_\Delta\left(\frac{\log{n}}{2\pi}\right)\right| \leq A_1(m_1(x), 0,\sigma)\frac{x^{1-\sigma}}{\log{x}}\left(\frac{x^{\sigma-1/2}}{x^{\sigma-1/2}+1}\right) +m_1(x)\log\log{x} \\
+O\left(\frac{m_1(x)x^{1-\sigma}}{(1-\sigma)^2\log^2{x}}+A_2\left(\mathcal{C}_{\cL}^{E},0,\theta,\sigma,x\right)+A_3(m_3(x),\sigma,x)\right)
\end{multline*}
and
\begin{multline*}
\frac{1}{2\pi}\sum_{n=2}^{\infty}\frac{\left|\Lambda_{\cL}(n)\right|}{\sqrt{n}}\left|\widehat{m}_\Delta\left(\frac{\log{n}}{2\pi}\right)\right| \leq A_1(m_1(x), 0,\sigma)\frac{x^{1-\sigma}}{\log{x}}+m_1(x)\log\log{x}\\
+O\left(\left(1+\log{\frac{1}{1-x^{\frac{1}{2}-\sigma}}}\right)\left(\frac{m_1(x)x^{1-\sigma}}{(1-\sigma)^2\log^2{x}}+A_2\left(\mathcal{C}_{\cL}^{E},0,\theta,\sigma,x\right)+A_3(m_3(x),\sigma,x)\right)\right)
\end{multline*}
as $x \to \infty$. Here, $m_i(x)$ are defined by~\eqref{eq:m123} if Conjecture~\ref{conj:SelbergVariant} is true, and by~\eqref{eq:m123V2} if Conjecture~\ref{conj:SelbergVariant2} is true. Moreover, $\mathcal{C}_{\cL}^{E}$ is defined by~\eqref{def:CER}, and $A_1\left(a,\varepsilon, \sigma\right)$, $A_2\left(a,\varepsilon, \theta, \sigma, x\right)$ and $A_3(a,\sigma,x)$ are defined by~\eqref{def:A1},~\eqref{def:Ai} and~\eqref{def:A6}, respectively.
\end{lemma}

\begin{proof}
We are going to prove Lemma~\ref{lemma:SumsPrimesUnderConjectures} in the case that Conjecture~\ref{conj:SelbergVariant} is true. The proof of the other case is similar and is thus left to the reader.

Let us now set 
\begin{equation*}
    f(u)=m_1(x)u, \qquad k(u)=O\left(1\right)\cdot\left(\frac{m_2(x)u}{\sqrt{\log{u}}}+\mathcal{C}_{\cL}^{E}u^{\frac{1}{2}+\theta}\log^{2}{u}\right).
\end{equation*}
Note that $\widetilde{\psi}_{\cL}(u)\leq f(u)+k(u)$ for all $u \in [2,x]$ by Lemma~\ref{lemma:MangoldSum} since $m_1(u)$ and $m_2(u)$ are increasing functions. The proof of both inequalities is similar to the proof of Lemma~\ref{lemma:I4general}, where we additionally also used
\begin{flalign*}
    \int_2^x \frac{1}{u^\sigma (\log{u})^{3/2}}\dif{u} &\leq \sqrt{\log{x}}\int_2^x \frac{1}{u^\sigma \log^2{u}}\dif{u}
    = \sqrt{\log{x}}\left(-\left[\frac{1}{u^{\sigma-1}\log{u}}\right]_2^{x}+\int_2^{x} \frac{1-\sigma}{u^\sigma\log{u}}\dif{u}\right) \nonumber \\ 
    & \ll \left(1 + (1-\sigma)\log{\log{x}} + \frac{x^{1-\sigma}}{(1-\sigma)\log^{2}{x}}\right)\sqrt{\log{x}}
\end{flalign*} 
as $x\to\infty$.
\end{proof}

\subsection{Sums over primes when $\cL \in \cSP$}
In this section, we estimate $\mathcal{I}_4$ from Corollary~\ref{corollary:UpperGW} when $h\in\left\{g_{\Delta},m_{\Delta}\right\}$ and when $\cL$ has a polynomial Euler product representation. This last restriction allows us to derive explicit results. As in the previous section, we will first describe which terms we need to estimate in order to get the desired result. Let
\begin{equation}
\label{def:C1}
    C_1(x,\sigma):=\frac{2^{1-\sigma}}{\log{2}}-\frac{2^\sigma}{\left(2\log{x}-\log{2}\right)x^{2\sigma-1}} +\int_2^{x} \left(\psi(u)-u\right)\left(\frac{u^{\sigma-1}}{x^{2\sigma-1}\left(2\log{x}-\log{u}\right)}-\frac{1}{u^\sigma \log{u}}\right)'\dif{u},
\end{equation}
\begin{equation}
\label{def:C2}
    C_2(x):=0.3+\frac{0.87}{(x^2-1)\log{x}},
\end{equation}
\begin{equation}
\label{def:C3}
C_3(x,\sigma):=\frac{1}{\left(x^{\sigma-\frac{1}{2}}+1\right)\log{x}}\left(-2^{1-\sigma}+\frac{2^\sigma}{x^{2\sigma-1}} +\int_2^{x} \left(\psi(u)-u\right)\left(\frac{1}{u^\sigma}-\frac{u^{\sigma-1}}{x^{2\sigma-1}}\right)'\dif{u}\right),
\end{equation}
\begin{equation}
\label{def:C4}
C_4(x,\sigma)\de \frac{2}{\log^{2}{x}}\left(2^{1-\sigma}\log{\frac{x}{2}} 
- \int_{2}^{x}\left(\psi(u)-u\right)\left(\frac{1}{u^{\sigma}}\log{\frac{x}{u}} 
\right)'\dif{u}\right)
\end{equation}
and
\begin{equation}
\label{def:C5}
C_5(x,\sigma)\de \frac{1}{\left(x^{\sigma-\frac{1}{2}}-1\right)\log{x}}\left(2^{1-\sigma}\left(1-\left(\frac{2}{x}\right)^{2\sigma-1}\right) - \int_{2}^{x}\left(\psi(u)-u\right)\left(\frac{1-\left(\frac{u}{x}\right)^{2\sigma-1}}{u^{\sigma}}\right)'\dif{u}\right).
\end{equation}

The next lemma provides estimates for $\mathcal{I}_4$ using previously defined functions.

\begin{lemma}
\label{lemma:ABC123Explicit}
Let $\cL \in \cSP$ with a polynomial Euler product of order $m$, $1/2<\sigma\leq 1$ and $x=e^{2\pi\Delta}\geq 2$. Then
    \begin{equation*}
         \frac{1}{2m\pi}\sum_{n=2}^{\infty}\frac{\left|\Lambda_{\cL}(n)\right|}{\sqrt{n}}\left|\widehat{g}_\Delta\left(\frac{\log{n}}{2\pi}\right)\right| \leq A_{f=u}(x,\sigma)-B_{f=u}(x,\sigma)+C_1(x,\sigma)+C_2(x)+C_3(x,\sigma)
    \end{equation*}
    and
    \begin{multline*}
         \frac{1}{2m\pi}\sum_{n=2}^{\infty}\frac{\left|\Lambda_{\cL}(n)\right|}{\sqrt{n}}\left|\widehat{m}_\Delta\left(\frac{\log{n}}{2\pi}\right)\right| \leq A_{f=u}(x,\sigma) + \left(D_{f=u}(x,\sigma)+C_4(x,\sigma)\right)\log{\frac{1}{1-x^{\frac{1}{2}-\sigma}}} \\ 
         + E_{f=u}(x,\sigma)
         +C_1(x,\sigma)+C_2(x)
         +C_5(x,\sigma),
    \end{multline*}
    where $A_f, B_f, D_f, E_f$ and $C_1,\ldots, C_5$ are defined as in~\eqref{def:Asigma},~\eqref{def:Bsigma},~\eqref{def:Dsigma},~\eqref{def:Esigma} and~\eqref{def:C1}--\eqref{def:C5}, respectively.
\end{lemma}

\begin{proof}
Again, it is sufficient to compute the sums up to $x$. The result follows by~\eqref{eq:BoundOnLambda}, by Lemma~\ref{lemma:gmhatsUpper} and by partial summation via equation~\eqref{eq:pslambda}, where we used $\widetilde{\psi}_{\cL}(u)=\psi(u)$ and $f(u)=u$. In addition, the estimate
\begin{equation*}
\sum_{n \leq x} \left(\frac{\Lambda(n)}{n^{2.5}\log{n}}+\frac{3\Lambda(n)}{n^{2.5}(x^2-1)\log{x}}\right) \leq \log{\zeta\left(\frac{5}{2}\right)} + \frac{-3\left(\zeta'/\zeta\right)(5/2)}{\left(x^2-1\right)\log{x}} \leq 0.3 + \frac{0.87}{\left(x^2-1\right)\log{x}}=C_2(x)
\end{equation*}
was also used.
\end{proof}

We are going to use Lemma~\ref{lemma:ABC123Explicit} in order to derive estimates for $\mathcal{I}_4$ according to whether $0<\sigma-1/2\leq 2\alpha/\log{x}$ or $\sigma\in(1/2,1)$. The next lemma provides estimates for the first region.

\begin{lemma}
\label{lem:I4SPcl}
Let $\cL\in\cSP$ with a polynomial Euler product of order $m$, $0<\sigma-1/2\leq 2\alpha/\log{x}$ for $\alpha>0$, $x=e^{2\pi\Delta}\geq\exp{\left(2(1+2\alpha)\right)}$ and $\nu\in(0,1)$. Then
\begin{flalign}
\label{eq:I4SPcl1}
\frac{1}{2m\pi}&\sum_{n=2}^{\infty}\frac{\left|\Lambda_{\cL}(n)\right|}{\sqrt{n}}\left|\widehat{g}_\Delta\left(\frac{\log{n}}{2\pi}\right)\right| \leq 4\left(\frac{1}{1-\left(\frac{4\alpha}{\log{x}}\right)^{2}}-\frac{1}{e^{2\alpha}\left(e^{2\alpha}+1\right)}\right)\left(\sigma-\frac{1}{2}\right)\frac{\sqrt{x}}{\frac{1}{2}\log{x}} \nonumber \\
&+ \left(2+\frac{8\alpha\left(3-\left(\frac{4\alpha}{\log{x}}\right)^{2}\right)\left(\sigma-\frac{1}{2}\right)}{\left(1-\left(\frac{4\alpha}{\log{x}}\right)^2\right)^{2}\left(\frac{1}{2}\log{x}\right)}\right)\frac{\sqrt{x}}{\left(\frac{1}{2}\log{x}\right)^{2}} + \frac{8\left(3+\left(\frac{4\alpha}{\log{x}}\right)^{2}\right)\left(\sigma-\frac{1}{2}\right)}{\left(1-\left(\frac{4\alpha}{\log{x}}\right)^2\right)^{3}}\cdot\frac{\sqrt{x}}{\left(\frac{1}{2}\log{x}\right)^{3}} \nonumber \\ 
&+ 6\left(1+\frac{\frac{1}{\nu^4}+\left(\frac{\log{x}}{\log{2}}\right)^{4}x^{\frac{\nu-1}{2}}}{\left(1-\frac{4\alpha}{\log{x}}\right)^{4}}\right)\frac{\sqrt{x}}{\left(\frac{1}{2}\log{x}\right)^{4}} +\left(\frac{1}{8\pi}\left(\frac{\sigma}{2}+\frac{1}{\log{x}}\right)+\frac{2.31}{\log^{2}{x}}\right)\log^{2}{x} \nonumber \\
&+\frac{1}{3}\left(1+\frac{2}{\log{x}}\right)\left(\sigma-\frac{1}{2}\right)\log^{3}{x} + \frac{2\sqrt{2}}{1-\frac{4\alpha}{\log{x}}}\left(1+\frac{2-\log{2}}{\log{x}}\right)\left(\sigma-\frac{1}{2}\right)
\end{flalign}
and
\begin{flalign}
\label{eq:I4SPcl2}
\frac{1}{2m\pi}\sum_{n=2}^{\infty}\frac{\left|\Lambda_{\cL}(n)\right|}{\sqrt{n}}&\left|\widehat{m}_\Delta\left(\frac{\log{n}}{2\pi}\right)\right| 
\leq 4\left(1+\frac{\alpha}{1-\left(\frac{4\alpha}{\log{x}}\right)^2}\left(1+\frac{2\alpha\left(3-\left(\frac{4\alpha}{\log{x}}\right)^{2}\right)}{\left(1-\left(\frac{4\alpha}{\log{x}}\right)^2\right)\left(\frac{1}{2}\log{x}\right)^2}\right)\right)\frac{\sqrt{x}}{\left(\frac{1}{2}\log{x}\right)^{2}} \nonumber \\ 
&+ \left(\frac{8\alpha\left(3+\left(\frac{4\alpha}{\log{x}}\right)^{2}\right)}{\left(1-\left(\frac{4\alpha}{\log{x}}\right)^2\right)^{3}} + 6\left(1+\frac{\frac{1}{\nu^4}+\left(\frac{\log{x}}{\log{2}}\right)^{4}x^{\frac{\nu-1}{2}}}{\left(1-\frac{4\alpha}{\log{x}}\right)^{4}}\right)\right)\frac{\sqrt{x}}{\left(\frac{1}{2}\log{x}\right)^{4}} \nonumber \\
&+ \left(\frac{2\sqrt{x}}{\left(\frac{1}{2}\log{x}\right)^{2}}+0.05\log^{2}{x}\right)\log{\frac{1}{1-x^{\frac{1}{2}-\sigma}}} + \left(0.12+\frac{0.88}{\log{x}}\right)\log^{2}{x}.
\end{flalign}
\end{lemma}

\begin{proof}
We are using Lemma~\ref{lemma:ABC123Explicit}, thus we need to estimate functions $A_{f=u}(x,\sigma)$, $B_{f=u}(x,\sigma)$, $D_{f=u}(x,\sigma)$ and $E_{f=u}(x,\sigma)$, as well as $C_1(x,\sigma)$ and $C_3(x,\sigma),\ldots,C_5(x,\sigma)$. 

Beginning with the first function which is defined by~\eqref{def:Asigma}, successive use of integration by parts gives
\begin{flalign*}
A_{f=u}(x,\sigma) &= \left(\frac{1}{1-\sigma}-\frac{1}{\sigma}\right)\frac{x^{1-\sigma}}{\log{x}} + \left(\frac{1}{(1-\sigma)^2}+\frac{1}{\sigma^2}\right)\frac{x^{1-\sigma}}{\log^{2}{x}} + \left(\frac{2}{(1-\sigma)^3}-\frac{2}{\sigma^3}\right)\frac{x^{1-\sigma}}{\log^{3}{x}} \\
&-\left(\frac{2^{1-\sigma}}{(1-\sigma)\log{2}}+\frac{2^{1-\sigma}}{(1-\sigma)^{2}\log^{2}{2}}+\frac{2^{2-\sigma}}{(1-\sigma)^{3}\log^{3}{2}}\right) \\
&+\frac{1}{x^{2\sigma-1}}\left(\frac{2^\sigma}{\sigma\left(2\log{x}-\log{2}\right)}-\frac{2^\sigma}{\sigma^2\left(2\log{x}-\log{2}\right)^2}+\frac{2^{1+\sigma}}{\sigma^{3}\left(2\log{x}-\log{2}\right)^{3}}\right) \\
&+\frac{6}{(1-\sigma)^3}\int_{2}^{x}\frac{\dif{u}}{u^{\sigma}\log^{4}{u}} + \frac{6}{\sigma^{3}x^{2\sigma-1}}\int_{2}^{x}\frac{\dif{u}}{u^{1-\sigma}\left(2\log{x}-\log{u}\right)^{4}}.
\end{flalign*}
Note that the sum of all terms from the second row and the third row in the above equality is not positive. The third term can be estimated as
$$
\left(\frac{2}{(1-\sigma)^3}-\frac{2}{\sigma^3}\right)\frac{x^{1-\sigma}}{\log^{3}{x}}\leq \frac{4(\sigma-\frac{1}{2})(\sigma^2-\sigma+1)}{(1-\sigma^3)\sigma^3}\cdot\frac{\sqrt{x}}{\log^3{x}} \leq \frac{8\left(3+\left(\frac{4\alpha}{\log{x}}\right)^{2}\right)\left(\sigma-\frac{1}{2}\right)}{\left(1-\left(\frac{4\alpha}{\log{x}}\right)^2\right)^{3}}\cdot\frac{\sqrt{x}}{\left(\frac{1}{2}\log{x}\right)^{3}}.
$$
Moreover, the first integral is estimated as 
\[
\left(\int_{2}^{x^\nu}+\int_{x^{\nu}}^{x}\right)\frac{\dif{u}}{u^{\sigma}\log^{4}{u}} \leq \frac{x^{\nu(1-\sigma)}}{(1-\sigma)\log^{4}{2}} + \frac{x^{1-\sigma}}{(1-\sigma)\left(\nu\log{x}\right)^{4}},
\]
and the second integral is estimated by using $2\log{x}-\log{u}\geq \log{x}$ for $u\in[2,x]$. All of this then implies
\begin{multline}
\label{eq:ABoundSigmaHalf}
A_{f=u}(x,\sigma) \leq \frac{4\left(\sigma-\frac{1}{2}\right)}{1-\left(\frac{4\alpha}{\log{x}}\right)^2}\cdot\frac{\sqrt{x}}{\frac{1}{2}\log{x}} + \left(2+\frac{8\alpha\left(3-\left(\frac{4\alpha}{\log{x}}\right)^{2}\right)\left(\sigma-\frac{1}{2}\right)}{\left(1-\left(\frac{4\alpha}{\log{x}}\right)^2\right)^{2}\left(\frac{1}{2}\log{x}\right)}\right)\frac{\sqrt{x}}{\left(\frac{1}{2}\log{x}\right)^{2}} \\ 
+ \frac{8\left(3+\left(\frac{4\alpha}{\log{x}}\right)^{2}\right)\left(\sigma-\frac{1}{2}\right)}{\left(1-\left(\frac{4\alpha}{\log{x}}\right)^2\right)^{3}}\cdot\frac{\sqrt{x}}{\left(\frac{1}{2}\log{x}\right)^{3}} + 6\left(1+\frac{\frac{1}{\nu^4}+\left(\frac{\log{x}}{\log{2}}\right)^{4}x^{\frac{\nu-1}{2}}}{\left(1-\frac{4\alpha}{\log{x}}\right)^{4}}\right)\frac{\sqrt{x}}{\left(\frac{1}{2}\log{x}\right)^{4}}.
\end{multline}

Function $B_{f=u}(x,\sigma)$ is defined by~\eqref{def:Bsigma}, and we have
\[
-B_{f=u}(x,\sigma) = \frac{1}{\left(x^{\sigma-1/2}+1\right)\log{x}}\left(-\frac{(2\sigma-1)x^{1-\sigma}}{\sigma(1-\sigma)}+\frac{2^{1-\sigma}}{1-\sigma}\left(1-\frac{1-\sigma}{\sigma}\left(\frac{2}{x}\right)^{2\sigma-1}\right)\right).
\]
Because 
\[
\frac{\dif{}}{\dif{\sigma}}\left(1-\frac{1-\sigma}{\sigma}\left(\frac{2}{x}\right)^{2\sigma-1}\right)=\frac{1}{\sigma^2}\left(\frac{2}{x}\right)^{2\sigma-1}\left(1-\sigma\log{4}+\sigma^2\log{4}+2(1-\sigma)\sigma\log{x}\right),
\]
and since the functions $\left(1-\sigma\log{4}+\sigma^2\log{4}\right)\sigma^{-2}$ and $\left(2(1-\sigma)\sigma\right)\sigma^{-2}$ are decreasing for $\sigma\geq1/2$, we obtain
\[
1-\frac{1-\sigma}{\sigma}\left(\frac{2}{x}\right)^{2\sigma-1} \leq 2\left(1+\frac{2-\log{2}}{\log{x}}\right)\left(\sigma-\frac{1}{2}\right)\log{x}
\]
by the mean-value theorem. Thus,
\begin{equation}
\label{eq:BBoundSigmaHalf}
-B_{f=u}(x,\sigma) \leq -\frac{4\left(\sigma-\frac{1}{2}\right)}{e^{2\alpha}\left(e^{2\alpha}+1\right)}\cdot\frac{\sqrt{x}}{\frac{1}{2}\log{x}} + \frac{2\sqrt{2}}{1-\frac{4\alpha}{\log{x}}}\left(1+\frac{2-\log{2}}{\log{x}}\right)\left(\sigma-\frac{1}{2}\right).
\end{equation}

Considering $D_{f=u}(x,\sigma)$, which is defined by~\eqref{def:Dsigma}, we simply have
\begin{equation}
\label{eq:DBoundSigmaHalf}
D_{f=u}(x,\sigma) \leq \frac{2}{\log^{2}{x}}\int_{2}^{x}\frac{1}{\sqrt{u}}\log{\frac{x}{u}}\dif{u} = \frac{8\sqrt{x}}{\log^{2}{x}}-\frac{4\sqrt{2}\left(2+\log{\frac{x}{2}}\right)}{\log^{2}{x}}.
\end{equation}

Finally, $E_{f=u}(x,\sigma)$, which is defined by~\eqref{def:Esigma}, can be estimated as
\begin{equation}
\label{eq:EBoundSigmaHalf}
E_{f=u}(x,\sigma) \leq \frac{1}{\left(\sigma-\frac{1}{2}\right)\log^{2}{x}}\int_{2}^{x}\frac{1}{\sqrt{u}}\left(1-\left(\frac{u}{x}\right)^{2\sigma-1}\right)\dif{u} \leq \frac{8\left(\sqrt{x}-\sqrt{2}\right)}{\log^{2}{x}},
\end{equation}
where we have used $\left(x^{\sigma-1/2}-1\right)\log{x}\geq \left(\sigma-1/2\right)\log^{2}{x}$.

Turning to $C_1(x,\sigma)$, which is defined by~\eqref{def:C1}, note first that
\begin{flalign}
\label{eq:auxC1}
0\leq \left(\frac{u^{\sigma-1}}{x^{2\sigma-1}\left(2\log{x}-\log{u}\right)}-\frac{1}{u^\sigma \log{u}}\right)' &= \frac{1+\sigma\log{u}}{u^{1+\sigma}\log^{2}{u}} + \frac{x^{1-2\sigma}(1-(1-\sigma)(2\log{x}-\log{u}))}{u^{2-\sigma}(2\log{x}-\log{u})^2} \\ 
&\leq \frac{1+\sigma\log{u}}{u^{1+\sigma}\log^{2}{u}} \leq \frac{1+\sigma\log{u}}{u^{3/2}\log^{2}{u}} \nonumber
\end{flalign}
for $u\in[2,x]$. Therefore,
\begin{flalign}
\label{eq:C1BoundSigmaHalf}
C_1(x,\sigma) &\leq 
\frac{1}{8\pi}\left(\frac{\sigma}{2}\log^{2}{x}+\log{x}\right)+ \frac{1}{8\pi}\left(\frac{\sigma}{2}\left(-\log^2{2}\right)-\log{2}\right)+ \frac{2^{1-\sigma}}{\log{2}}-\frac{1}{x\sqrt{2}\log{x}} \nonumber \\
&\leq\left(\frac{1}{8\pi}\left(\frac{\sigma}{2}+\frac{1}{\log{x}}\right)+\frac{2.01}{\log^{2}{x}}\right)\log^{2}{x}-\frac{1}{x\sqrt{2}\log{x}}
\end{flalign}
by Lemma~\ref{lemma:MangoldSum}. 

Considering $C_3(x,\sigma)$, which is defined by~\eqref{def:C3}, note first that
\[
\left(\frac{1}{u^\sigma}-\frac{u^{\sigma-1}}{x^{2\sigma-1}}\right)' = \frac{\left(1-\sigma\right)x^{1-2\sigma}}{u^{2-\sigma}} - \frac{\sigma}{u^{1+\sigma}} \leq 0
\]
for $u \in [2,x]$. This implies, by Lemma~\ref{lemma:MangoldSum}, that
\begin{flalign}
\label{eq:auxC3}
\left|\int_{2}^{x}\left(\psi(u)-u\right)\left(\frac{1}{u^\sigma}-\frac{u^{\sigma-1}}{x^{2\sigma-1}}\right)'\dif{u}\right| &\leq 2\sigma\int_{2}^{x}\frac{\log^{2}{u}}{u^{1/2+\sigma}}\dif{u} - \frac{2(1-\sigma)}{x^{2\sigma-1}}\int_{2}^{x}\frac{\log^{2}{u}}{u^{3/2-\sigma}}\dif{u} \\
&\leq \left(2\sigma-\frac{2(1-\sigma)}{x^{2\sigma-1}}\right)\int_{2}^{x}\frac{\log^{2}{u}}{u}\dif{u} \leq \frac{2}{3}\left(\sigma-\frac{1-\sigma}{x^{2\sigma-1}}\right)\log^{3}{x}. \nonumber
\end{flalign}
Denoting by $f(\sigma)$ the function in the last brackets, we see that $f(1/2)=0$ and $f'(\sigma)\leq 2+\log{x}$ for $\sigma\geq1/2$. Therefore,
\begin{equation}
\label{eq:C3BoundSigmaHalf}
C_3(x,\sigma) \leq \frac{1}{3}\left(1+\frac{2}{\log{x}}\right)\left(\sigma-\frac{1}{2}\right)\log^{3}{x}
\end{equation}
by the mean-value theorem and $x^{\sigma-1/2}+1\geq 2$.

Concerning $C_4(x,\sigma)$, we have by Lemma~\ref{lemma:MangoldSum} that
\begin{equation}
\label{eq:auxC4}
-\int_{2}^{x}\left(\psi(u)-u\right)\left(\frac{1}{u^{\sigma}}\log{\frac{x}{u}}\right)'\dif{u} \leq \frac{\log^{2}{x}}{4\pi(2\sigma-1)^2}\left(x^{\frac{1}{2}-\sigma}+2^{\frac{1}{2}-\sigma}\left(-1+\sigma(2\sigma-1)\log{\frac{x}{2}}\right)\right).
\end{equation}
Denoting by $f(\sigma)$ the function in the outer brackets on the right-hand side of the above inequality, we can observe that $f(1/2)=f'(1/2)=0$. Also, $f''(\sigma)\leq (4-2\log{2})\log{x}+\log^{2}{x}$ for $\sigma\in[1/2,1]$ since $0<4+2\sigma^2\log^2{2}-\sigma(8+\log{2})\log{2}+\log{4}\leq 4-2\log{2}$ there. It follows that $f(\sigma)\leq (1/2)\left(\left(4-2\log{2}\right)\log{x}+\log^{2}{x}\right)(\sigma-1/2)^2$ by the 2nd-order mean-value theorem. Noting that $\log{x}\geq 2$, we obtain
\begin{equation}
\label{eq:C4BoundSigmaHalf}
C_4(x,\sigma) \leq \left(0.046+\frac{2\sqrt{2}}{\log^{3}{x}}\right)\log^{2}{x}.
\end{equation}

Considering $C_5(x,\sigma)$, first note that 
\[
\left(\frac{1}{u^\sigma}\left(1-\left(\frac{u}{x}\right)^{2\sigma-1}\right)\right)' = \frac{1}{u^{1+\sigma}}\left((1-\sigma)\left(\frac{u}{x}\right)^{2\sigma-1}-\sigma\right) \leq 0
\]
for $u\in[2,x]$ and $\sigma\geq1/2$. Therefore,
\begin{flalign}
\label{eq:forC5}
- \int_{2}^{x}\left(\psi(u)-u\right)\left(\frac{1-\left(\frac{u}{x}\right)^{2\sigma-1}}{u^{\sigma}}\right)'\dif{u} &\leq \frac{\log^{2}{x}}{8\pi}\int_{2}^{x}\frac{\sigma-(1-\sigma)\left(\frac{u}{x}\right)^{2\sigma-1}}{u^{1/2+\sigma}}\dif{u} \nonumber \\
&\leq \frac{\log^{2}{x}}{8\pi(2\sigma-1)}\left(2^{\frac{3}{2}-\sigma}\left(\sigma-\left(\frac{2}{x}\right)^{\sigma-\frac{1}{2}}+(1-\sigma)\left(\frac{2}{x}\right)^{2\sigma-1}\right)\right)
\end{flalign}
by Lemma~\ref{lemma:MangoldSum}. Denote by $f(\sigma)$ the function in the outer brackets on the right-hand side of~\eqref{eq:forC5}. As before, we can observe that $f(1/2)=f'(1/2)=0$. Also, after omitting the negative terms, 
\[
f''(\sigma)\leq 2^{\frac{5}{2}+\sigma}x^{1-2\sigma}\left(\left(1-(1-\sigma)\log{2}\right)\log{x}+(1-\sigma)\log^{2}{x}\right) \leq 8\sqrt{2}\left(1-\frac{\log{2}}{4}\right)\log^{2}{x}
\]
for $\sigma\in[1/2,1]$ since $\log{x}\geq 2$. It follows that $f(\sigma)\leq \sqrt{2}\left(4-\log{2}\right)(\sigma-1/2)^{2}\log^{2}{x}$ by the 2nd-order mean-value theorem. Therefore, the right-hand side of~\eqref{eq:forC5} is $\leq 0.1(\sigma-1/2)\log^{4}{x}$. Because 
\[
2^{1-\sigma}\left(1-\left(\frac{2}{x}\right)^{2\sigma-1}\right) \leq 4\left(\sigma-\frac{1}{2}\right)\log{x},
\]
we finally obtain 
\begin{equation}
\label{eq:C5BoundSigmaHalf}
C_5(x,\sigma) \leq \left(0.1+\frac{4}{\log^{3}{x}}\right)\log^{2}{x}.
\end{equation}
Inequality~\eqref{eq:I4SPcl1} now follows by taking estimates~\eqref{eq:ABoundSigmaHalf},~\eqref{eq:BBoundSigmaHalf},~\eqref{eq:C1BoundSigmaHalf},~\eqref{def:C2} and~\eqref{eq:C3BoundSigmaHalf} in the first estimate from Lemma~\ref{lemma:ABC123Explicit}. Inequality~\eqref{eq:I4SPcl2} follows by taking estimates~\eqref{eq:ABoundSigmaHalf},~\eqref{eq:DBoundSigmaHalf},~\eqref{eq:C4BoundSigmaHalf},~\eqref{eq:EBoundSigmaHalf},~\eqref{eq:C1BoundSigmaHalf}, \eqref{def:C2} and~\eqref{eq:C5BoundSigmaHalf} in the second estimate from Lemma~\ref{lemma:ABC123Explicit}, while we also use the fact that $\sigma-1/2\leq 2\alpha/\log{x}$. The proof of Lemma~\ref{lem:I4SPcl} is thus complete.
\end{proof}

Moving to the second region, we have the following lemma for $\sigma\in(1/2,1)$.

\begin{lemma}
\label{lem:I4SPgen}
Let $\cL\in\cSP$ with a polynomial Euler product of order $m$, $1/2<\sigma<1$, $x=e^{2\pi\Delta}\geq 2$, $\nu_1>0$ and $\nu_2>1$. Then
\begin{flalign}
\label{eq:I4SPgen1}
\frac{1}{2m\pi}&\sum_{n=2}^{\infty}\frac{\left|\Lambda_{\cL}(n)\right|}{\sqrt{n}}\left|\widehat{g}_\Delta\left(\frac{\log{n}}{2\pi}\right)\right| \leq \frac{2\sigma-1}{\sigma(1-\sigma)}\cdot\frac{x^{1-\sigma}}{\log{x}} + \frac{\left(\nu_2\sigma\right)^{2}+(1-\sigma)^2}{\left(\sigma(1-\sigma)\right)^{2}}\cdot\frac{x^{1-\sigma}}{\log^{2}{x}} + \frac{1}{\nu_1^2}x^{\frac{1-\sigma}{\nu_2}} + \log{\log{x}} \nonumber \\
&+ \frac{2^{\frac{1}{2}-\sigma}\left(4\sigma-1+\sigma\left(2\sigma-1\right)\log{2}\right)+\frac{2\sigma-1}{x^{\sigma-1/2}}}{4\pi(2\sigma-1)^{2}} + \log{\frac{1}{\log{2}}} + \frac{1-\sigma 2^{1-\sigma}}{(1-\sigma)\log{2}} + \int_{0}^{\nu_1}\theta_{1}(u)\dif{u} + 0.3 \nonumber \\
&+ \frac{x^{1-\sigma}}{\left(x^{\sigma-\frac{1}{2}}+1\right)\log{x}}\left(\frac{1-2\sigma}{\sigma(1-\sigma)}+\frac{2^{1-\sigma}}{(1-\sigma)x^{1-\sigma}} + \frac{\sigma 2^{5/2-\sigma}\left(8+4(2\sigma-1)\log{2}+(2\sigma-1)^{2}\log^{2}{2}\right)}{(2\sigma-1)^{3}x^{1-\sigma}}\right) \nonumber \\
&+ \frac{2^{\sigma}(1-\sigma)}{\sigma\left(2\log{x}-\log{2}\right)x^{2\sigma-1}} + \frac{0.87}{(x^2-1)\log{x}}
\end{flalign}
and
\begin{flalign}
\label{eq:I4SPgen2}
&\frac{1}{2m\pi}\sum_{n=2}^{\infty}\frac{\left|\Lambda_{\cL}(n)\right|}{\sqrt{n}}\left|\widehat{m}_\Delta\left(\frac{\log{n}}{2\pi}\right)\right| \leq \frac{2\sigma-1}{\sigma(1-\sigma)}\left(1+\frac{1}{x^{\sigma-\frac{1}{2}}-1}\right)\frac{x^{1-\sigma}}{\log{x}} \nonumber \\
&+ \left(\frac{\left(\sigma\nu_2\right)^{2}+(1-\sigma)^{2}}{\sigma^2}+2\log{\frac{1}{1-x^{\frac{1}{2}-\sigma}}}\right)\frac{x^{1-\sigma}}{(1-\sigma)^{2}\log^{2}{x}} + \frac{1}{\nu_1^2}x^{\frac{1-\sigma}{\nu_2}} \nonumber \\
&+ \frac{\log{x}}{2\sigma-1}\Biggl(\left(\frac{1}{2^{\frac{3}{2}+\sigma}\pi}+\frac{2^{1-\sigma}(1-\sigma-\sigma^2)(2\sigma-1)}{\sigma(1-\sigma)\log^{2}{x}}\right)\frac{1}{x^{\sigma-\frac{1}{2}}-1} \nonumber \\
&+\left(\frac{\sigma}{2^{\frac{1}{2}+\sigma}\pi}+\frac{2^{2-\sigma}(2\sigma-1)}{\log^{2}{x}}\right)\log{\frac{1}{1-x^{\frac{1}{2}-\sigma}}}\Biggr) + \log{\log{x}} +\frac{2^{\frac{1}{2}-\sigma}\left(4\sigma-1+\sigma\left(2\sigma-1\right)\log{2}\right)+\frac{2\sigma-1}{x^{\sigma-1/2}}}{4\pi(2\sigma-1)^{2}} \nonumber \\ 
&+ \log{\frac{1}{\log{2}}} + \frac{1-\sigma 2^{1-\sigma}}{(1-\sigma)\log{2}} + \int_{0}^{\nu_1}\theta_{1}(u)\dif{u} + 0.3 + \frac{2^{\sigma}(1-\sigma)}{\sigma\left(2\log{x}-\log{2}\right)x^{2\sigma-1}} + \frac{0.87}{(x^2-1)\log{x}}.
\end{flalign}
Here, $\theta_1(u)$ is defined by~\eqref{def:theta}.
\end{lemma}

\begin{proof}
First note that 
\[
- \frac{1}{x^{2\sigma-1}}\int_{2}^{x}\frac{\dif{u}}{u^{1-\sigma}\left(2\log{x}-\log{u}\right)} \leq -\frac{x^{1-\sigma}}{\sigma\log{x}} + \frac{x^{1-\sigma}}{\sigma^2\log^{2}{x}} + \frac{2^{\sigma}}{\sigma\left(2\log{x}-\log{2}\right)x^{2\sigma-1}}
\]
by Lemma~\ref{lemma:secondIntegral} for $\sigma\mapsto 1-\sigma$. Therefore, 
\begin{multline}
\label{eq:A}
A_{f=u}(x,\sigma) \leq \frac{2\sigma-1}{\sigma(1-\sigma)}\cdot\frac{x^{1-\sigma}}{\log{x}} + \log{\log{x}} - \log{\log{2}} - \frac{2^{1-\sigma}-1}{(1-\sigma)\log{2}}+\frac{\left(\sigma\nu_{2}\right)^2+(1-\sigma)^{2}}{\left(\sigma(1-\sigma)\right)^2}\cdot\frac{x^{1-\sigma}}{\log^{2}{x}} \\
+\frac{1}{\nu_1^2}x^{\frac{1-\sigma}{\nu_2}} + \int_{0}^{\nu_1}\theta_{1}(u)\dif{u} + \frac{2^\sigma}{\sigma\left(2\log{x}-\log{2}\right)x^{2\sigma-1}}
\end{multline}
by~\cite[Estimates~(5.11) and~(5.12)]{PalojarviSimonic}, and 
\[
-B_{f=u}(x,\sigma) \leq \frac{-x^{1-\sigma}+2^{1-\sigma}}{\left(x^{\sigma-\frac{1}{2}}+1\right)(1-\sigma)\log{x}} + \frac{x^{\sigma}}{x^{2\sigma-1}\sigma\log{x}} - \frac{x^{1-\sigma}}{\sigma\log{x}}\left(\frac{x^{\sigma-\frac{1}{2}}}{x^{\sigma-\frac{1}{2}}+1}\right)
\]
by straightforward integration. This implies
\begin{multline}
\label{eq:AminusB}
A_{f=u}(x,\sigma)-B_{f=u}(x,\sigma) \leq \frac{2\sigma-1}{\sigma(1-\sigma)}\left(\frac{x^{\sigma-\frac{1}{2}}}{x^{\sigma-\frac{1}{2}}+1}\right)\frac{x^{1-\sigma}}{\log{x}} + \frac{\left(\nu_2\sigma\right)^{2}+(1-\sigma)^2}{\left(\sigma(1-\sigma)\right)^{2}}\cdot\frac{x^{1-\sigma}}{\log^{2}{x}} + \log{\frac{\log{x}}{\log{2}}} \\
-\frac{2^{1-\sigma}-1}{(1-\sigma)\log{2}} + \frac{2^{1-\sigma}}{\left(x^{\sigma-\frac{1}{2}}+1\right)(1-\sigma)\log{x}} + \frac{2^\sigma}{\sigma\left(2\log{x}-\log{2}\right)x^{2\sigma-1}} +\frac{1}{\nu_1^2}x^{\frac{1-\sigma}{\nu_2}} + \int_{0}^{\nu_1}\theta_{1}(u)\dif{u}.
\end{multline}
Considering $D_{f=u}(x,\sigma)$ and $E_{f=u}(x,\sigma)$, which are defined by~\eqref{def:Dsigma} and~\eqref{def:Esigma} respectively, we have
\begin{equation}
\label{eq:D}
D_{f=u}(x,\sigma) \leq \frac{2x^{1-\sigma}}{(1-\sigma)^2\log^{2}{x}}
\end{equation}
and
\begin{equation}
\label{eq:E}
E_{f=u}(x,\sigma) \leq \frac{1}{x^{\sigma-\frac{1}{2}}-1}\left(\frac{2\sigma-1}{\sigma(1-\sigma)}\frac{x^{1-\sigma}}{\log{x}}-\frac{2^{1-\sigma}}{(1-\sigma)\log{x}}+\frac{2^{\sigma}}{\sigma x^{2\sigma-1}\log{x}}\right)
\end{equation}
by straightforward integration.

In order to estimate $C_1(x,\sigma)$, which is defined by~\eqref{def:C1}, note first that
\begin{equation*}
0\leq \left(\frac{u^{\sigma-1}}{x^{2\sigma-1}\left(2\log{x}-\log{u}\right)}-\frac{1}{u^\sigma \log{u}}\right)' \leq \frac{1+\sigma\log{u}}{u^{1+\sigma}\log^{2}{u}} + \frac{x^{1-2\sigma}}{u^{2-\sigma}(2\log{x}-\log{u})^2}
\end{equation*}
by~\eqref{eq:auxC1}. Therefore,
\begin{equation}
\label{eq:C1}
C_1(x,\sigma) \leq \frac{2^{1-\sigma}}{\log{2}}-\frac{2^\sigma}{\left(2\log{x}-\log{2}\right)x^{2\sigma-1}} + \frac{1}{4\pi}\left(2^{\frac{1}{2}-\sigma}\left(4\sigma-1+\sigma\left(2\sigma-1\right)\log{2}\right)+\frac{2\sigma-1}{x^{\sigma-1/2}}\right)\frac{1}{(2\sigma-1)^{2}}.
\end{equation}
In order to estimate $C_3(x,\sigma)$, which is defined by~\eqref{def:C3}, note first that
\begin{equation*}
\left|\int_{2}^{x}\left(\psi(u)-u\right)\left(\frac{1}{u^\sigma}-\frac{u^{\sigma-1}}{x^{2\sigma-1}}\right)'\dif{u}\right| \leq \frac{\sigma 2^{5/2-\sigma}\left(8+4(2\sigma-1)\log{2}+(2\sigma-1)^{2}\log^{2}{2}\right)}{(2\sigma-1)^{3}}
\end{equation*}
by~\eqref{eq:auxC3}. Therefore,
\begin{equation}
\label{eq:C3}
C_3(x,\sigma) \leq \frac{\sigma 2^{5/2-\sigma}\left(8+4(2\sigma-1)\log{2}+(2\sigma-1)^{2}\log^{2}{2}\right)}{\left(x^{\sigma-1/2}+1\right)(2\sigma-1)^{3}\log{x}}.
\end{equation}
Turning to $C_4(x,\sigma)$, which is defined by~\eqref{def:C4}, we first have
\[
-\int_{2}^{x}\left(\psi(u)-u\right)\left(\frac{1}{u^{\sigma}}\log{\frac{x}{u}}\right)'\dif{u} \leq \frac{2^{\frac{1}{2}-\sigma}\sigma\log^{3}{x}}{4\pi(2\sigma-1)}
\]
by~\eqref{eq:auxC4}, which implies
\begin{equation}
\label{eq:C4}
C_4(x,\sigma) \leq \frac{\sigma\log{x}}{2^{\frac{1}{2}+\sigma}\pi(2\sigma-1)} + \frac{2^{2-\sigma}}{\log{x}}.
\end{equation}
Finally, an estimate for $C_{5}(x,\sigma)$, which is defined by~\eqref{def:C5}, is given as
\begin{equation}
\label{eq:C5}
C_{5}(x,\sigma) \leq \frac{1}{x^{\sigma-\frac{1}{2}}-1}\left(\frac{\log{x}}{2^{\frac{3}{2}+\sigma}\pi(2\sigma-1)}+\frac{2^{1-\sigma}}{\log{x}}\right),
\end{equation}
where we used~\eqref{eq:forC5} and the fact that the expression in the inner brackets is $\leq 1$ for $\sigma\in[1/2,1]$.

Inequality~\eqref{eq:I4SPgen1} now follows by taking estimates~\eqref{eq:AminusB},~\eqref{eq:C1},~\eqref{def:C2} and~\eqref{eq:C3} in the first estimate from Lemma~\ref{lemma:ABC123Explicit}. Inequality~\eqref{eq:I4SPgen2} follows by taking estimates~\eqref{eq:A},~\eqref{eq:D},~\eqref{eq:C4},~\eqref{eq:E},~\eqref{eq:C1},~\eqref{def:C2} and~\eqref{eq:C5} in the second estimate from Lemma~\ref{lemma:ABC123Explicit}. The proof of Lemma~\ref{lem:I4SPgen} is thus complete.
\end{proof}

\section{Proofs of the results from Section~\ref{sec:Main}}
\label{sec:ProofsMain}

In this section, we prove our main results. Firstly, we state a result that clearly shows which of the previously derived estimates we need and how we are going to use them.

\begin{corollary} 
\label{cor:EstimatesCombined}
Assume the Generalized Riemann Hypothesis for $\cL(s)$. Let $\sigma\in(1/2,1]$ and let $t$ and $\Delta$ be real numbers such that $\Delta|t| \geq 10^4$ and $\Delta\geq 0.8$,
and
\begin{equation*}
|t|\geq \max\left\{b^{+} + \sqrt{\left(1+a^{+}\right)\left(\frac{5}{2}+a^{+}\right)}, 4\left(a^{+}+b^{+}\right)+4\right\},
\end{equation*}
where $a^+$ and $b^+$ are as in~\eqref{def:a+b+}.
Additionally, let $\mathfrak{a}$ and $\mathfrak{b}$ be numbers such that $0<\mathfrak{a}\leq \frac{1}{4}\sqrt{\Delta|t|}$, $\mathfrak{b}>0$ and that
\begin{equation}
\label{assump:lowert}
\mathfrak{a}\sqrt{\frac{|t|}{\Delta}}\geq17 , \quad 
|t|-\frac{\mathfrak{b}}{\lambda^{-}}\sqrt{\Delta|t|}-b^{+}\geq 17,
\end{equation}
where $\lambda^-$ is as in~\eqref{def:tau}. Then
\begin{equation}
\label{eq:upperMain}
\log{\left|\mathcal{L}\left(\sigma+\ie t\right)\right|} \leq \frac{\log{\tau}}{2\pi\Delta}\log{\left(1+e^{-(2\sigma-1)\pi\Delta}\right)} + \left(\frac{\left|\mathcal{I}\right|}{2\pi} + \left|\mathcal{I}_3\right| + \left|\mathcal{I}_4\right|+ \left|\log{\left|\mathcal{L}\left(\frac{5}{2}+\ie t\right)\right|}\right| + L\right)
\end{equation}
and
\begin{equation}
\label{eq:lowerMain}
\log{\left|\mathcal{L}\left(\sigma+\ie t\right)\right|} \geq \frac{\log{\tau}}{2\pi\Delta}\log{\left(1-e^{-(2\sigma-1)\pi\Delta}\right)} - \left(\frac{\left|\mathcal{I}\right|}{2\pi} + \left|\mathcal{I}_3\right| + \left|\mathcal{I}_4\right|+ \left|\log{\left|\mathcal{L}\left(\frac{5}{2}+\ie t\right)\right|}\right| - L\right),
\end{equation}
where $\tau$ is defined by~\eqref{def:tau}, $L$ and $\mathcal{I}$ are defined as in Propositions~\ref{prop:LogSum} and~\ref{prop:FourierGamma}, respectively, and $\mathcal{I}_3$ and $\mathcal{I}_4$ as in Corollary~\ref{corollary:UpperGW}. The terms $\mathfrak{m}_i$ and $\mathfrak{m}_i'$ for $i\in\{1,2,3\}$ from Proposition~\ref{prop:FourierGamma} and Corollary~\ref{corollary:UpperGW} are given by Table~\ref{table:valuesm}. 
\end{corollary}

\begin{proof}
Inequalities~\eqref{eq:upperMain} and~\eqref{eq:lowerMain} follow by Proposition~\ref{prop:LogSum}, Corollary~\ref{corollary:UpperGW} and Lemma~\ref{lemma:mghat0}. The terms $\mathfrak{m}_i$ and $\mathfrak{m}_i'$ that are needed to bound $\mathcal{I}$ and $\mathcal{I}_3$, are given as follows (consult pairs from Table~\ref{table:valuesm}, listed from left to right and then from up to bottom): Lemmas~\ref{lemma:m2g},~\ref{lemma:gAbsolute},~\ref{lemma:m2g},~\ref{lemma:mUpperRealComplex},~\ref{lemma:mUpperReal}, and~\ref{lemma:mUpperRealComplex}. The conditions from the previously mentioned results are satisfied by the conditions from Corollary~\ref{cor:EstimatesCombined}.
\end{proof}

\begin{table}
\centering
\begin{tabular}{l|l|l|l|l|l|l} 
\toprule
 & $\mathfrak{m}_1$ & $\mathfrak{m}_1'$ & $\mathfrak{m}_2$ & $\mathfrak{m}_2'$ & $\mathfrak{m}_3$ & $\mathfrak{m}_3'$ \\ 
\midrule
 Eq.~\eqref{eq:upperMain} & $121$ & $0$ & $4$ & $0$ & $28$ & $0$ \\ 
 Eq.~\eqref{eq:lowerMain} & $24$ & $\frac{2}{\Delta^{2}}\log{\frac{2}{\sigma-\frac{1}{2}}}$ & $15+\left(\frac{2}{\pi\Delta}\right)^{2}\log{\frac{2}{\sigma-\frac{1}{2}}}$ & $\frac{4}{\pi^2}\log{\frac{2}{\sigma-\frac{1}{2}}}$ & $4$ & $\left(\frac{2}{\pi\Delta}\right)^{2}\log{\frac{2}{\sigma-\frac{1}{2}}}$ \\ 
\bottomrule
\end{tabular}
\caption{Terms $\mathfrak{m}_i$ and $\mathfrak{m}_i'$ that are used in Corollary~\ref{cor:EstimatesCombined} according to whether one is using~\eqref{eq:upperMain} or~\eqref{eq:lowerMain}.} 
\label{table:valuesm}
\end{table}

\begin{proof}[Proof of Theorem~\ref{thm:MainFullSelberg}]
Take $\Delta=\frac{1}{\pi}\log{\log{\tau}}$, and let $\tau$ and $t$ satisfy the conditions~\eqref{eq:conditions1} and~\eqref{eq:conditions2}. Then the conditions of Corollary~\ref{cor:EstimatesCombined} are satisfied for sufficiently large $\tau$ and $|t|$, where we take $\mathfrak{a}=1/\sqrt{\pi}$ and $\mathfrak{b}=\lambda^{-}\sqrt{\pi}$. We need to estimate terms on the right-hand sides of~\eqref{eq:upperMain} and~\eqref{eq:lowerMain}. Beginning with $L$, we have
\begin{equation}
\label{eq:LUpper}
L\leq O(1)\cdot\left(\max\left\{\sdeg,2\sdeg+\Re\left\{\xi_{\cL}\right\}\right\}\log{\left(1+\frac{1+a^{+}+b^{+}}{|t|}\right)}
+ \frac{\sdeg+m_{\cL}+f/\min\left\{1,\left(\lambda^{-}\right)^{3}\right\}}{t^2}\right)
\end{equation}
and
\begin{equation}
\label{eq:LLower}
L \geq -O(1)\cdot\left(\sdeg\log{\left(1+\frac{1+a^{+}+b^{+}}{|t|}\right)} + \frac{\sdeg\max\left\{1,\left(a^{+}\right)^2\right\}+f/\min\left\{1,\left(\lambda^{-}\right)^{3}\right\}}{t^{2}}\right)
\end{equation}
by Proposition~\ref{prop:LogSum} since $|t|-b^{+}\geq |t|/2$. Here we also used the fact that 
\begin{equation}
\label{eq:L3Up}
L_3^{\uparrow} \leq \frac{1}{2}\log{\left(1 + \left(\frac{5/2+a^{+}}{|t|-b^{+}}\right)^{2}\right)}\sum_{j=1}^{f} \Re\left\{\mu_{j}\right\} = \frac{1}{4}\log{\left(1 + \left(\frac{5/2+a^{+}}{|t|-b^{+}}\right)^{2}\right)}\left(\Re\left\{\xi_\cL\right\}+\sdeg\right)
\end{equation}
and 
\begin{equation}
\label{eq:L3Down}
L_3^{\downarrow} \geq -\frac{1}{4}\log{\left(1 + \left(\frac{5/2+a^{+}}{|t|-b^{+}}\right)^{2}\right)}\sum_{j=1}^{f} 1=-\frac{\sdeg}{4}\log{\left(1 + \left(\frac{5/2+a^{+}}{|t|-b^{+}}\right)^{2}\right)},
\end{equation}
together with $\log{\left(1+x^2\right)}\ll\log{\left(1+x\right)}$ for $x\geq0$. Turning to $\mathcal{I}_3$ and $\mathcal{I}_4$ from Corollary~\ref{corollary:UpperGW}, we have for $\mathcal{I}_3$ from~\eqref{eq:upperMain} that
\begin{equation}
\label{eq:I3Upper}
\left|\mathcal{I}_3\right| \leq \frac{28m_\cL (\log\log{\tau})^2(\log{\tau})}{\pi\left(\pi+|t|\log\log{\tau}\right)} \ll \frac{m_{\cL}\log{\tau}\log\log{\tau}}{|t|},
\end{equation}
while for $\mathcal{I}_3$ from~\eqref{eq:lowerMain}
\begin{flalign}
\label{eq:I3Lower}
\left|\mathcal{I}_3\right| &\leq m_\cL (\log\log{\tau})^2\left(\frac{4}{\pi\left(\pi+|t|\log\log{\tau}\right)}+\frac{4}{(\log\log{\tau})^2\left(\pi^2+t^2(\log\log{\tau})^2\right)}\log{\frac{2}{\sigma-\frac{1}{2}}}\right)\log{\tau} \nonumber \\
&\ll \frac{m_\cL\log{\tau}\log\log{\tau}}{|t|} + \frac{m_\cL \log{\tau}}{t^2(\log\log{\tau})^2}\log{\frac{2}{\sigma-\frac{1}{2}}}.
\end{flalign}
Estimates for $\mathcal{I}_4$ are provided by Lemma~\ref{lemma:I4general}, taking $x=\left(\log{\tau}\right)^{2}$ there. Finally,
\begin{equation}
\label{eq:IUpper}   
\mathcal{I} \ll \frac{\sdeg (1+a^++b^+)\max\left\{1,\lambda^{+}\right\}}{\min\{1, (\lambda^-)^4\}}\sqrt{\frac{\log\log{\tau}}{|t|}}\log{\left(1+\sqrt{|t|\log\log{\tau}}\right)} 
\end{equation}
for $\mathcal{I}$ from~\eqref{eq:upperMain}, and
\begin{equation}
\label{eq:ILower}
\mathcal{I} \ll \frac{\sdeg (1+a^++b^+)\max\left\{1,\lambda^{+}\right\}}{\min\{1, (\lambda^-)^4\}}\left(1+\frac{1}{\left(\log{\log{\tau}}\right)^{2}}\log{\frac{2}{\sigma-\frac{1}{2}}}\right)\sqrt{\frac{\log\log{\tau}}{|t|}}\log{\left(1+\sqrt{|t|\log\log{\tau}}\right)} 
\end{equation}
for $\mathcal{I}$ from~\eqref{eq:lowerMain}, where we used 
\[
|t|-\sqrt{|t|\log{\log{\tau}}}-b^{+}\geq \left(\frac{1}{2}\sqrt{\frac{|t|}{\log{\log{\tau}}}}-1\right)\sqrt{|t|\log{\log{\tau}}}\gg \sqrt{|t|\log{\log{\tau}}}
\]
and $\log{\left(1+x\right)}\ll\log{\left(1+\sqrt{x}\right)}$ for $x\geq 0$. The proof of Theorem~\ref{thm:MainFullSelberg} now follows by observing that $A_4$ comes from combining~\eqref{eq:LUpper},~\eqref{eq:I3Upper} and~\eqref{eq:IUpper}, while $A_5$ comes from combining~\eqref{eq:LLower},~\eqref{eq:I3Lower} and~\eqref{eq:ILower}.
\end{proof}

\begin{proof}[Proof of Corollary~\ref{cor:CLFullSelberg}]
Let $t\neq 0$ be an ordinate that is different from any $\gamma$. Then $\log{\left|\cL(\sigma+\ie t)\right|}$ is a continuous function in $\sigma\in[1/2,1]$. The proof is now similar to the proof of Theorem~\ref{thm:MainFullSelberg}, just with $\pi(1+2\varepsilon)\Delta=\log{\log{\tau}}$. Estimates~\eqref{eq:LUpper} and~\eqref{eq:LLower} are still valid since we have not used there any information on $\Delta$. The upper bounds for $|\mathcal{I}_3|$ and $\left|\mathcal{I}\right|$ in the case of~\eqref{eq:upperMain} are the same as the final estimates in~\eqref{eq:I3Upper} and~\eqref{eq:IUpper} since $\frac{1}{2}\log{\log{\tau}}\leq\pi\Delta\leq\log{\log{\tau}}$. It remains to estimate $\mathcal{I}_4$. By Lemma~\ref{lemma:I4general} we have
\begin{multline*}
\left|\mathcal{I}_4\right| \leq \frac{(1+2\varepsilon)A_1(\mathcal{C}_{\cL}^{R}(\varepsilon), \varepsilon,\sigma)}{2}\frac{\left(\log{\tau}\right)^{\frac{2(1-\sigma+\varepsilon)}{1+2\varepsilon}}}{\log{\log{\tau}}} +(1+\varepsilon)\mathcal{C}_{\cL}^{R}(\varepsilon)\log{\left(\frac{2\log{\log{\tau}}}{1+2\varepsilon}\right)} \\ 
+ O\left(\frac{\mathcal{C}_{\cL}^{R}(\varepsilon)\left(\log{\tau}\right)^{\frac{2(1-\sigma+\varepsilon)}{1+2\varepsilon}}}{(1-\sigma+\varepsilon)^2\left(\log{\log{\tau}}\right)^2}\right)+O\left(A_2\left(\mathcal{C}_{\cL}^{R}(\varepsilon)+\mathcal{C}_{\cL}^{E},\varepsilon, \theta, \sigma, \left(\log{\tau}\right)^{2/(1+2\varepsilon)}\right)\right)
\end{multline*}
for sufficiently large $\tau$. Taking $\sigma\to1/2$, we obtain
\[
\mathcal{I}_4 \ll \frac{\mathcal{C}_{\cL}^{R}(\varepsilon)\log{\tau}}{\left(\log{\log{\tau}}\right)^{2}} + \left(\mathcal{C}_{\cL}^{R}(\varepsilon)+\mathcal{C}_{\cL}^{E}\right)\left(\log{\tau}\right)^{\frac{2\max\{\theta,\varepsilon\}}{1+2\varepsilon}}\left(\log{\log{\tau}}\right)^{3}.
\]
The proof is finished by also taking $\sigma\to1/2$ in the first term on the right-hand side of~\eqref{eq:upperMain}.
\end{proof}

\begin{proof}[Proof of Theorem~\ref{thm:MainConjectures}]
The proof is similar to the proof of Theorem~\ref{thm:MainFullSelberg}, just now we are using Lemma~\ref{lemma:SumsPrimesUnderConjectures} with $x=(\log{\tau})^{2}$ to bound $\mathcal{I}_4$.
\end{proof}

\begin{proof}[Proof of Theorem~\ref{thm:Polynomial}]
The proof is similar to the proof of Theorem~\ref{thm:MainFullSelberg}, just now we estimate the term $|\mathcal{I}_4|$ using Lemmas~\ref{lem:I4SPcl} and~\ref{lem:I4SPgen} for $x=(\log{\tau})^{2}$. Let $0<\sigma-1/2\leq \alpha/\log{\log{\tau}}$ for $\alpha>0$ and $\nu\in(0,1)$. Then, for $|\mathcal{I}_4|$ from~\eqref{eq:upperMain}, we have
\begin{equation}
\label{eq:Thm4proof1}
\frac{1}{m}|\mathcal{I}_4| \leq \mathcal{A}_1(\alpha,\tau)\frac{\log{\tau}}{\left(\log{\log{\tau}}\right)^{2}} + \mathcal{A}_{2}(\alpha,\nu,\tau)\frac{\log{\tau}}{\left(\log{\log{\tau}}\right)^{4}} + \mathcal{A}_{3}(\alpha,\sigma,\tau)\left(\log{\log{\tau}}\right)^{2}
\end{equation}
by~\eqref{eq:I4SPcl1}, where $\mathcal{A}_1(\alpha,\tau)$, $\mathcal{A}_{2}(\alpha,\nu,\tau)$ and $\mathcal{A}_{3}(\alpha,\sigma,\tau)$ are defined by~\eqref{eq:calA1},~\eqref{eq:calA2} and~\eqref{eq:calA3}, respectively. Obviously, $\mathcal{A}_1(\alpha,\tau)\ll 1$, $\mathcal{A}_{2}(\alpha,\nu,\tau)\ll 1$ and $\mathcal{A}_{3}(\alpha,\sigma,\tau)\ll 1$ for fixed $\alpha$, fixed $\nu$ and sufficiently large $\tau$. Thus,~\eqref{eq:Polynomial1} follows. On the other hand, for $|\mathcal{I}_4|$ from~\eqref{eq:lowerMain} we have
\begin{multline}
\label{eq:Thm4proof2}
\frac{1}{m}|\mathcal{I}_4| \leq \mathcal{B}_{1}(\alpha,\tau)\frac{\log{\tau}}{\left(\log{\log{\tau}}\right)^{2}} + \mathcal{A}_{2}(\alpha,\nu,\tau)\frac{\log{\tau}}{\left(\log{\log{\tau}}\right)^{4}} + \mathcal{B}_{2}(\tau)\left(\log{\log{\tau}}\right)^{2} \\ 
+ \left(\frac{2\log{\tau}}{\left(\log{\log{\tau}}\right)^{2}}+0.2\left(\log{\log{\tau}}\right)^{2}\right)\log{\frac{1}{1-(\log{\tau})^{1-2\sigma}}}
\end{multline}
by~\eqref{eq:I4SPcl2}, where $\mathcal{B}_{1}(\alpha,\tau)$ and $\mathcal{B}_{2}(\tau)$ are defined by~\eqref{eq:calB1B2}. Then~\eqref{eq:Polynomial2} follows since $\mathcal{B}_{1}(\alpha,\tau)\ll 1$ and $\mathcal{B}_{2}(\tau)\ll 1$ for fixed $\alpha$, fixed $\nu$ and sufficiently large $\tau$. Considering the second region, let $\sigma\in(1/2,1)$, $\nu_1>0$ and $\nu_2>1$. Then, for $|\mathcal{I}_4|$ from~\eqref{eq:upperMain}, we have
\begin{multline}
\label{eq:Thm4proof3}
\frac{1}{m}|\mathcal{I}_4| \leq \frac{2\sigma-1}{2\sigma(1-\sigma)}\frac{\left(\log{\tau}\right)^{2-2\sigma}}{\log{\log{\tau}}} + \frac{\left(\nu_2\sigma\right)^{2}+(1-\sigma)^{2}}{\left(2\sigma(1-\sigma)\right)^{2}}\frac{\left(\log{\tau}\right)^{2-2\sigma}}{\left(\log{\log{\tau}}\right)^{2}} \\ 
+ \frac{1}{\nu_1^2}\left(\log{\tau}\right)^{\frac{2(1-\sigma)}{\nu_2}}
+ \log{\left(2\log{\log{\tau}}\right)} + \mathcal{A}_{4}(\nu_1,\sigma) + \frac{\mathcal{A}_{5}(\sigma,\tau)+\mathcal{A}_{6}(\sigma,\tau)}{(\log{\tau})^{2\sigma-1}}
\end{multline}
by~\eqref{eq:I4SPgen1}, where $\mathcal{A}_{4}(\nu_1,\sigma)$, $\mathcal{A}_{5}(\sigma,\tau)$ and $\mathcal{A}_{6}(\sigma,\tau)$ are defined by~\eqref{eq:calA4},~\eqref{eq:calA5} and~\eqref{eq:calA6}, respectively. One can easily see that
\[
\mathcal{A}_{4}(\nu_1,\sigma) + \frac{\mathcal{A}_{5}(\sigma,\tau)+\mathcal{A}_{6}(\sigma,\tau)}{(\log{\tau})^{2\sigma-1}} \ll \frac{1}{(2\sigma-1)^2} + \frac{(\log{\tau})^{1-2\sigma}}{2\sigma-1}\left(1+\frac{1}{(1-\sigma)(2\sigma-1)^{2}\log{\log{\tau}}}\right)
\]
for fixed $\nu_1>0$ and sufficiently large $\tau$. Thus,~\eqref{eq:Polynomial3} follows by fixing also $\nu_2>1$. On the other hand, for $|\mathcal{I}_4|$ from~\eqref{eq:lowerMain} we have
\begin{flalign}
\label{eq:Thm4proof4}
\frac{1}{m}|\mathcal{I}_4| &\leq \frac{2\sigma-1}{2\sigma(1-\sigma)}\left(1+\frac{1}{(\log{\tau})^{2\sigma-1}-1}\right)\frac{(\log{\tau})^{2-2\sigma}}{\log{\log{\tau}}} \nonumber \\
&+\frac{1}{4(1-\sigma)^{2}}\left(\frac{(\sigma\nu_2)^{2}+(1-\sigma)^{2}}{\sigma^2}+2\log{\frac{1}{1-(\log{\tau})^{1-2\sigma}}}\right)\frac{(\log{\tau})^{2-2\sigma}}{(\log{\log{\tau}})^{2}} + \frac{1}{\nu_{1}^{2}}(\log{\tau})^{\frac{2(1-\sigma)}{\nu_2}} \nonumber \\
&+\frac{2}{2\sigma-1}\left(\frac{\mathcal{B}_{3}(\sigma,\tau)}{(\log{\tau})^{2\sigma-1}-1}+\mathcal{B}_{4}(\sigma,\tau)\log{\frac{1}{1-(\log{\tau})^{1-2\sigma}}}\right)\log{\log{\tau}} \nonumber \\ 
&+ \log{\left(2\log{\log{\tau}}\right)}
+\mathcal{A}_{4}(\nu_1,\sigma) + \frac{\mathcal{A}_{5}(\sigma,\tau)}{(\log{\tau})^{2\sigma-1}}
\end{flalign}
by~\eqref{eq:I4SPgen2}, where $\mathcal{B}_{3}(\sigma,\tau)$ and $\mathcal{B}_{4}(\sigma,\tau)$ are defined by~\eqref{eq:calB3B4}. We have $\mathcal{B}_{4}(\sigma,\tau)\ll 1$ and there exists a constant $C>0$ such that $\mathcal{B}_{3}(\sigma,\tau)\leq C$ for sufficiently large $\tau$. Thus,~\eqref{eq:Polynomial4} follows by fixing $\nu_1>0$ and $\nu_2>1$. The proof is thus complete.
\end{proof}


\begin{proof}[Proofs of Theorems~\ref{thm:ExplicitUpper} and~\ref{thm:ExplicitLower}] 
Take $\Delta=\frac{1}{\pi}\log{\log{\tau}}$, and let $\tau$ and $t$ satisfy the conditions from Theorem~\ref{thm:ExplicitUpper}. Also, take $\mathfrak{a}=\mathfrak{b}=1$. Then the conditions of Corollary~\ref{cor:EstimatesCombined} are satisfied. First, we will provide estimates for all terms except $\mathcal{I}_4$ on the right-hand side of~\eqref{eq:upperMain} and~\eqref{eq:lowerMain}. 

Because of the existence of a polynomial Euler product, we have
\begin{equation}
\label{eq:Thm6logL}
\left|\log{\left|\mathcal{L}\left(\frac{5}{2}+\ie t\right)\right|}\right| \leq m\sum_p\sum_{k=1}^\infty \frac{1}{kp^{5k/2}}=m\log{\zeta\left(\frac{5}{2}\right)}\leq 0.3m.
\end{equation} 
Because $|t|\geq 231.2$, we have $L_1\leq 4.3m_{\cL}/t^2$. Moreover, because
\[
\frac{5}{2}+a^{+}+b^{+} \leq \frac{5}{2}\left(1+a^{+}+b^{+}\right) \leq \frac{5}{34}\sqrt{\frac{\pi|t|}{\log{\log{\tau}}}}, \quad b^{+} \leq 1+a^{+}+b^{+} \leq \frac{1}{17}\sqrt{\frac{\pi|t|}{\log{\log{\tau}}}},
\]
and $|t|\log{\log{\tau}}\geq 10^{4}\pi$, we also have
\[
L_2^{\uparrow} \leq \frac{5\sdeg}{4}\log{\frac{1+\frac{5}{34}\sqrt{\frac{\pi}{|t|\log{\log{\tau}}}}}{1-\frac{1}{17}\sqrt{\frac{\pi}{|t|\log{\log{\tau}}}}}} \leq \frac{5\sdeg}{4}\log{\left(1+\frac{0.37}{\sqrt{|t|\log{\log{\tau}}}}\right)} \leq \frac{0.5\sdeg}{\sqrt{|t|\log{\log{\tau}}}}.
\]
Similar reasoning gives also $L_2^{\downarrow} \geq -0.261\sdeg/\left(\sqrt{|t|\log{\log{\tau}}}\right)$, where we used $\log{(1-x)}\geq -1.001x$ for $x\in[0,0.0012]$. Next,
\[
L_3^{\uparrow} \leq \frac{0.02\left(\Re\left\{\xi_\cL\right\}+\sdeg\right)}{\left(\log{\log{\tau}}\right)^{2}}, \quad 
L_3^{\downarrow} \geq \frac{-0.02\sdeg}{\left(\log{\log{\tau}}\right)^{2}}
\]
by~\eqref{eq:L3Up},~\eqref{eq:L3Down}, and also by using $|t|-b^{+}\geq 2\sqrt{\frac{1}{\pi}|t|\log{\log{\tau}}}$. Finally, $L_4^{\uparrow} \leq 3.41\sdeg/\left(|t|\log{\log{\tau}}\right)$ and $L_4^{\downarrow}\geq -0.06\sdeg/\left(\log{\log{\tau}}\right)^{2}$ since $|t|-b^{+}\geq 217$, 
\[
\left(a^+\right)^{2} + \frac{7}{2}a^{+}+\frac{13}{2} \leq \frac{13}{2}\left(1+a^{+}\right)^{2} \leq \frac{13\pi|t|}{578\log{\log{\tau}}}, \quad \frac{\log{\log{\tau}}}{|t|}\leq \frac{\pi}{289}.
\]
Combining all gives
\begin{equation}
\label{eq:Thm6L}
-\sdeg\left(\frac{0.08}{\left(\log{\log{\tau}}\right)^{2}}+\frac{0.261}{\sqrt{|t|\log{\log{\tau}}}}\right)\leq L\leq \sdeg\left(\frac{0.02}{\left(\log{\log{\tau}}\right)^{2}}+\frac{0.52}{\sqrt{|t|\log{\log{\tau}}}}\right) + \frac{0.02\Re\left\{\xi_\cL\right\}}{\left(\log{\log{\tau}}\right)^{2}} + \frac{4.3m_{\cL}}{t^2}.
\end{equation}
Turning to $\mathcal{I}_3$, we have by~\eqref{def:I3I4} and Table~\ref{table:valuesm} that
\begin{equation}
\label{eq:Thm6I3}
\left|\mathcal{I}_3\right| \leq \frac{28m_{\cL}\log{\tau}\log{\log{\tau}}}{\pi|t|}, \quad \left|\mathcal{I}_3\right| \leq \frac{4m_{\cL}\log{\tau}\log{\log{\tau}}}{\pi|t|} + \frac{4m_{\cL}\log{\tau}}{\left(|t|\log{\log{\tau}}\right)^{2}}\log{\frac{2}{\sigma-\frac{1}{2}}}
\end{equation}
for~\eqref{eq:upperMain} and~\eqref{eq:lowerMain}, respectively. It remains to estimate $\mathcal{I}$. We have $|\mathcal{I}|\leq \widehat{\mathcal{I}}_1+\widehat{\mathcal{I}}_2$, where these two functions are defined in Proposition~\ref{prop:FourierGamma}. In the case of $\mathcal{I}$ from~\eqref{eq:upperMain} we have
\begin{equation}
\label{eq:Thm6Ip1}
\widehat{\mathcal{I}}_1 \leq 78.22\sdeg\sqrt{\frac{\log{\log{\tau}}}{|t|}}\log{\left(1+\sqrt{\frac{1}{\pi}|t|\log{\log{\tau}}}\right)}, \quad 
\widehat{\mathcal{I}}_2 \leq 5.5\sdeg\sqrt{\frac{\log{\log{\tau}}}{|t|}}\log{\left(1+\sqrt{\frac{1}{\pi}|t|\log{\log{\tau}}}\right)},
\end{equation}
where we also used $1/2+a^{+}\leq\log{\left(1+\sqrt{\frac{1}{\pi}|t|\log{\log{\tau}}}\right)}$ for the latter inequality. On the other hand,
\begin{gather}
\widehat{\mathcal{I}}_1 \leq \sdeg\left(34+\frac{3.71}{\log{\log{\tau}}}\log{\frac{2}{\sigma-\frac{1}{2}}}\right)\sqrt{\frac{\log{\log{\tau}}}{|t|}}\log{\left(1+\sqrt{\frac{1}{\pi}|t|\log{\log{\tau}}}\right)}, \nonumber \\
\widehat{\mathcal{I}}_2 \leq \sdeg\left(14.7+\frac{1.6}{\log{\log{\tau}}}\log{\frac{2}{\sigma-\frac{1}{2}}}\right)\sqrt{\frac{\log{\log{\tau}}}{|t|}}\log{\left(1+\sqrt{\frac{1}{\pi}|t|\log{\log{\tau}}}\right)} \label{eq:Thm6Ip2}
\end{gather}
for $\mathcal{I}$ from~\eqref{eq:lowerMain}. By~\eqref{eq:Thm6logL}, by the right-hand side of~\eqref{eq:Thm6L}, by the first inequality in~\eqref{eq:Thm6I3}, and by~\eqref{eq:Thm6Ip1} we have
\begin{equation}
\label{eq:Thm6res1}
\frac{1}{2\pi}\left|\mathcal{I}\right| + \left|\mathcal{I}_3\right| + \left|\log{\left|\mathcal{L}\left(\frac{5}{2}+\ie t\right)\right|}\right| + L \leq E_{\cL}^{\uparrow}(t)
\end{equation}
for~\eqref{eq:upperMain}, where $E_{\cL}^{\uparrow}(t)$ is defined by~\eqref{def:EUpper}. On the other hand, by~\eqref{eq:Thm6logL}, by the left-hand side of~\eqref{eq:Thm6L}, by the second inequality in~\eqref{eq:Thm6I3}, and by~\eqref{eq:Thm6Ip2} we have
\begin{equation}
\label{eq:Thm6res2}
\frac{1}{2\pi}\left|\mathcal{I}\right| + \left|\mathcal{I}_3\right| + \left|\log{\left|\mathcal{L}\left(\frac{5}{2}+\ie t\right)\right|}\right| - L \leq E_{\cL}^{\downarrow}(\sigma,t)
\end{equation}
for~\eqref{eq:lowerMain}, where $E_{\cL}^{\downarrow}(\sigma,t)$ is defined by~\eqref{def:ELower}.

Estimates for $\left|\mathcal{I}_4\right|$ are provided in the proof of Theorem~\ref{thm:Polynomial}. Taking~\eqref{eq:Thm6res1} and~\eqref{eq:Thm4proof1} in~\eqref{eq:upperMain} gives~\eqref{eq:ExplicitUpper1}, while taking~\eqref{eq:Thm6res2} and~\eqref{eq:Thm4proof2} in~\eqref{eq:lowerMain} gives~\eqref{eq:ExplicitLower1}. Also, taking~\eqref{eq:Thm6res1} and~\eqref{eq:Thm4proof3} in~\eqref{eq:upperMain} gives~\eqref{eq:ExplicitUpper2}, while taking~\eqref{eq:Thm6res2} and~\eqref{eq:Thm4proof4} in~\eqref{eq:lowerMain} gives~\eqref{eq:ExplicitLower2}. The proof is thus complete.
\end{proof}

\begin{proof}[Proof of Corollary~\ref{cor:RiemannZeta}]
In~\cite{SimonicSonRH} we introduce an additional parameter $\lambda_3$ that allowed us (conditionally) to improve unconditional results for relatively small $t$. Similarly, we would try to use $\pi\Delta=\nu_1\log{\log{t}}-\nu_2$ for $\nu_1\in(0,1]$ and $\nu_2\geq 0$ in this paper, where the statement of Theorem~\ref{thm:ExplicitUpper} follows by taking $\nu_1=1$ and $\nu_2=0$. Comparison with~\eqref{eq:ChandeeEst} reveals that the minimal $t\approx\exp{\left(\exp{(18.8)}\right)}$ occurs for $\nu_1=1$, $\nu_2\approx1.98$ and $\nu\approx0.36$. 
\end{proof}

\begin{proof}[Proof of Corollary \ref{cor:RiemannZeta2}]
If we take $\nu_1=1$ and $\nu_2\geq0$ is bounded in the proof of Corollary~\ref{cor:RiemannZeta}, then
\[
\left|\zeta\left(\frac{1}{2}+\ie t\right)\right| \leq \exp{\left(\frac{\log{2}}{2}\frac{\log{t}}{\log{\log{t}}}+\left(\frac{\nu_2\log{2}}{2}+2e^{-\nu_2}\right)\frac{\log{t}}{\left(\log{\log{t}}\right)^2}+O\left(\frac{\log{t}}{\left(\log{\log{t}}\right)^{3}}\right)\right)}
\]
for sufficiently large $t$. Optimizing the second term on $\nu_2$ gives $\nu_2\approx1.7528$, and thus the final result.
\end{proof}

\begin{proof}[Proof of Corollary~\ref{cor:RiemannZeta3}]
Fix $\nu\in(0,1)$. Then the result follows from~\eqref{eq:ExplicitLower1} since $m=1$, 
\[
\mathcal{B}_{1}(\alpha,\tau) = 4(1+\alpha) + O\left(\frac{1}{\left(\log{\log{\tau}}\right)^{2}}\right),
\]
$\mathcal{A}_{2}(\alpha,\nu,\tau)\ll 1$, $\mathcal{B}_{2}(\tau)\ll 1$, and
\[
E_{\cL}^{\downarrow}(\sigma,t) \ll 1 + (\log{t})\sqrt{\frac{\log{\log{t}}}{t}}\left(1+\frac{1}{\log{\log{t}}}\log{\frac{2}{\sigma-\frac{1}{2}}}\right)
\]
for sufficiently large $t$.
\end{proof}

\begin{proof}[Proof of Corollary~\ref{cor:CLforSP}]
We are using Theorem~\ref{thm:ExplicitUpper}. Let $t\neq 0$ be an ordinate that is different from any $\gamma$. Then $\log{\left|\cL(\sigma+\ie t)\right|}$ is a continuous function in $\sigma\in[1/2,1]$. Then~\eqref{eq:CriticalLine} follows from~\eqref{eq:ExplicitUpper1} after first taking $\sigma\to 1/2$, and then also $\alpha\to0$. The second statement is clear by the proof of Theorem~\ref{thm:Polynomial}.
\end{proof}



\begin{thebibliography}{CCM19}

\bibitem[AP17]{AistleitnerPankowski2017}
C.~Aistleitner and L.~Pa\'nkowski, \emph{{Large values of $L$-functions from
  the Selberg class}}, J. Math. Anal. Appl. \textbf{446} (2017), no.~1,
  345--364.

\bibitem[BEP24]{BEHP2024}
G.~{Bhowmik}, A.~{Ernvall-Hyt{\"o}nen}, and N.~{Paloj{\"a}rvi}, \emph{{Explicit
  estimates for the Goldbach summatory function}}, preprint available at arXiv:2411.00323 (2024).

\bibitem[Boo06]{Booker2006}
A.~R. Booker, \emph{Artin's conjecture, {T}uring's method, and the
  {R}iemann hypothesis}, Experiment. Math. \textbf{15} (2006), no.~4, 385--407.
  \MR{2293591}

\bibitem[BC24]{BC2024}
M.~Bordignon and G.~Cherubini, \emph{{Coprime-Universal Quadratic Forms}}, preprint available at arXiv:2406.01533 (2024).

\bibitem[BF24]{BF2023}
H.~M. Bui and A.~Florea, \emph{Negative moments of the {R}iemann
  zeta-function}, J. Reine Angew. Math. \textbf{806} (2024), 247--288.

\bibitem[CC11]{CarneiroChandee}
E.~Carneiro and V.~Chandee, \emph{Bounding {$\zeta(s)$} in the critical strip},
  J. Number Theory \textbf{131} (2011), no.~3, 363--384.

\bibitem[CCM19]{Carneiro}
E.~Carneiro, A.~Chirre, and M.~B. Milinovich, \emph{Bandlimited approximations
  and estimates for the {R}iemann zeta-function}, Publ. Mat. \textbf{63}
  (2019), no.~2, 601--661.

\bibitem[CLV13]{CLV2013}
E.~Carneiro, F.~Littmann, and J.~D. Vaaler, \emph{{Gaussian subordination for
  the Beurling-Selberg extremal problem}}, Trans. Amer. Math. Soc. \textbf{365}
  (2013), no.~7, 3493--3534.

\bibitem[CM24]{CM2024}
E.~Carneiro and M.~B. Milinovich, \emph{{On Littlewood's estimate for the
  modulus of the zeta function on the critical line}}, preprint available at arXiv:2403.17803 (2024).

\bibitem[\v{C}23]{Cech}
M.~\v{C}ech, \emph{{Mean value of real Dirichlet characters using a double
  Dirichlet series}}, Canad. Math. Bull. \textbf{66} (2023), no.~4, 1135--1151.

\bibitem[Cha09]{ChandeeExplBounds}
V.~Chandee, \emph{Explicit upper bounds for {$L$}-functions on the critical
  line}, Proc. Amer. Math. Soc. \textbf{137} (2009), no.~12, 4049--4063.

\bibitem[CS11]{ChandeeSound}
V.~Chandee and K.~Soundararajan, \emph{Bounding {$|\zeta(\frac12+it)|$} on the
  {R}iemann hypothesis}, Bull. Lond. Math. Soc. \textbf{43} (2011), no.~2,
  243--250.

\bibitem[CKS23]{Chaubey23}
S.~Chaubey, S.~S. Khurana, and A.~I. Suriajaya, \emph{Zeros of derivatives of
  {$L$}-functions in the {S}elberg class on {$\Re(s)<1/2$}}, Proc. Amer. Math.
  Soc. \textbf{151} (2023), no.~5, 1855--1866.

\bibitem[Chi19]{ChirreANote}
A.~Chirre, \emph{A note on entire {$L$}-functions}, Bull. Braz. Math. Soc.
  (N.S.) \textbf{50} (2019), no.~1, 67--93.

\bibitem[CG22]{chirreGoncalves}
A.~Chirre and F.~Gon\c{c}alves, \emph{Bounding the log-derivative of the
  zeta-function}, Math. Z. \textbf{300} (2022), no.~1, 1041--1053.

\bibitem[CHS24]{ChirreSimonicHagen}
A.~Chirre, M.~V. Hagen, and A.~Simoni\v{c}, \emph{Conditional estimates for the
  logarithmic derivative of {D}irichlet ${L}$-functions}, Indag. Math. (N.S.)
  \textbf{35} (2024), no.~1, 14--27.

\bibitem[Dav00]{Davenport}
H.~Davenport, \emph{Multiplicative number theory}, 3rd ed., Graduate Texts in
  Mathematics, Springer New York, NY, 2000.

\bibitem[EHP22]{EHP}
A.-M. Ernvall-Hyt\"onen and N.~Paloj\"arvi, \emph{Explicit bound for the number
  of primes in arithmetic progressions assuming the generalized riemann
  hypothesis}, Math. Comp. \textbf{91} (2022), no.~335, 1317--1365.

\bibitem[GG07]{GoldstonGonek}
D.~A. Goldston and S.~M. Gonek, \emph{A note on {$S(t)$} and the zeros of the
  {R}iemann zeta-function}, Bull. Lond. Math. Soc. \textbf{39} (2007), no.~3,
  482--486.

\bibitem[IK04]{IKANT}
H.~Iwaniec and E.~Kowalski, \emph{Analytic number theory}, American
  Mathematical Society Colloquium Publications 53, AMS, Providence, RI, 2004.

\bibitem[Lit24]{LittlewoodOnTheZeros}
J.~E. Littlewood, \emph{On the zeros of the {R}iemann zeta-function}, Proc.
  Cambridge Philos. Soc. \textbf{22} (1924), no.~3, 295--318.

\bibitem[Lit25]{LittlewoodOnTheRiemann}
\bysame, \emph{On the {R}iemann {Z}eta-{F}unction}, Proc. London Math. Soc. (2)
  \textbf{24} (1925), no.~3, 175--201.

\bibitem[Mon97]{MontgomeryRiemann}
H.~L. Montgomery, \emph{{Extreme values of the Riemann zeta function}},
  Comment. Math. Helv. \textbf{52} (1997), 511--518.

\bibitem[MV07]{MontgomeryVaughan}
H.~L. Montgomery and R.~C. Vaughan, \emph{Multiplicative number theory. {I}.
  {C}lassical theory}, Cambridge Studies in Advanced Mathematics, vol.~97,
  Cambridge University Press, Cambridge, 2007.

\bibitem[Od{\v z}11]{odzak2011}
A.~Od{\v z}ak, \emph{{On the $H$-invariants in the Selberg class}}, Math.
  Balkanica \textbf{25} (2011), no.~5, 511--518.

\bibitem[Olv74]{Olver}
F.~W.~J. Olver, \emph{Asymptotics and special functions}, Academic Press, New
  York, 1974.

\bibitem[Pal19]{Palojarvi}
N.~Paloj\"{a}rvi, \emph{On the explicit upper and lower bounds for the number
  of zeros of the {S}elberg class}, J. Number Theory \textbf{194} (2019),
  218--250.

\bibitem[PS22]{PalojarviSimonic}
N.~Paloj\"{a}rvi and A.~Simoni\v{c}, \emph{Conditional estimates for
  ${L}$-functions in the {S}elberg class}, preprint available at
  arXiv:2211.01121 (2022).

\bibitem[PS13]{PankowskiSteuding2013}
L.~Pa\'nkowski and J.~Steuding, \emph{{Extreme values of $L$-functions from the
  Selberg class}}, Int. J. Number Theory \textbf{9} (2013), no.~5, 1113--1124.

\bibitem[Sch76]{SchoenfeldSharperRH}
L.~Schoenfeld, \emph{Sharper bounds for the {C}hebyshev functions {$\theta
  (x)$} and {$\psi (x)$}. {II}}, Math. Comp. \textbf{30} (1976), no.~134,
  337--360.

\bibitem[Sel92]{SelbergOldAndNew}
A.~Selberg, \emph{Old and new conjectures and results about a class of
  {D}irichlet series}, Proceedings of the {A}malfi {C}onference on {A}nalytic
  {N}umber {T}heory ({M}aiori, 1989), Univ. Salerno, Salerno, 1992,
  pp.~367--385.

\bibitem[Sim22a]{SimonicCS}
A.~Simoni\v{c}, \emph{Explicit estimates for $\zeta(s)$ in the critical strip
  under the {R}iemann {H}ypothesis}, Q. J. Math. \textbf{73} (2022), no.~3,
  1055--1087.

\bibitem[Sim22b]{SimonicSonRH}
\bysame, \emph{On explicit estimates for {$S(t)$}, {$S_1(t)$}, and
  {$\zeta(1/2+{\rm i}t)$} under the {R}iemann {H}ypothesis}, J. Number Theory
  \textbf{231} (2022), 464--491.

\bibitem[Sim23]{SimonicLfunctions}
\bysame, \emph{Estimates for {$L$}-functions in the critical strip under {GRH}
  with effective applications}, Mediterr. J. Math. \textbf{20} (2023), no.~2,
  Paper No. 87, 24.

\bibitem[Sou08]{SoundResonance2008}
K.~Soundararajan, \emph{{Extreme values of zeta and $L$-functions}}, Math. Ann.
  \textbf{342} (2008), no.~2, 467--486.

\bibitem[Sou09]{SoundMoments}
\bysame, \emph{Moments of the {R}iemann zeta function}, Ann. of Math. (2)
  \textbf{170} (2009), no.~2, 981--993.

\bibitem[Ste07]{SteudingBook}
J.~Steuding, \emph{Value-distribution of {$L$}-functions}, Lecture Notes in
  Mathematics, vol. 1877, Springer, Berlin, 2007.

\bibitem[Tit28]{TitchRiemann}
E.~C. Titchmarsh, \emph{{On an inequality satisfied by the zeta-function of
  Riemann}}, Proc. London Math. Soc. (2) \textbf{28} (1928), no.~1, 70--80.

\bibitem[Tit86]{Titchmarsh}
\bysame, \emph{The theory of the {R}iemann zeta-function}, 2nd ed., The
  Clarendon Press, Oxford University Press, New York, 1986.

\bibitem[XY22]{XiaoYang2022}
X.~Xiao and Q.~Yang, \emph{{A note on large values of $L(\sigma,\chi)$}}, Bull.
  Aust. Math. Soc. \textbf{105} (2022), no.~3, 412--418.

\bibitem[{Yan}23]{YangD2023}
D.~{Yang}, \emph{{Omega theorems for logarithmic derivatives of zeta and
  $L$-functions}}, arXiv e-prints (2023), arXiv:2311.16371.

\bibitem[Yan24]{Yang2024}
Q.~Yang, \emph{{Large values of $\zeta(s)$ for $1/2<\text{Re}(s)<1$}}, J.
  Number Theory \textbf{254} (2024), 199--213.

\bibitem[You91]{Young}
R.~M. Young, \emph{Euler's constant.}, Math. Gaz. \textbf{75} (1991), 187--190.


\end{thebibliography}

\providecommand{\bysame}{\leavevmode\hbox to3em{\hrulefill}\thinspace}
\providecommand{\MR}{\relax\ifhmode\unskip\space\fi MR }
\providecommand{\MRhref}[2]{%
  \href{http://www.ams.org/mathscinet-getitem?mr=#1}{#2}
}

\providecommand{\href}[2]{#2}

\end{document}